\documentclass[11pt]{aims}
\usepackage{amsmath, amssymb, mathrsfs}
  \usepackage{paralist}
  \usepackage{graphics} 
  \usepackage{epsfig} 
\usepackage{graphicx} 
 \usepackage{epstopdf}
 \usepackage[colorlinks=true]{hyperref}
 \hypersetup{urlcolor=blue, citecolor=blue}
 \usepackage[toc,page]{appendix}
\usepackage{chngcntr}
\usepackage{hyperref}

\usepackage{titlesec}
\titleformat{\section}{\Large\bfseries}{\thesection.}{4pt}{}
\titleformat{\subsection}{\large\bfseries}{\thesection.\arabic{subsection}.}{4pt}{}
\titleformat{\subsubsection}{\bfseries}{\thesection.\arabic{subsection}.\arabic{subsubsection}.}{4pt}{}
\titleformat*{\paragraph}{\bfseries}
\titleformat*{\subparagraph}{\bfseries}
\usepackage[margin=1in]{geometry}

  \def\currentyear{}
   \def\currentmonth{}
    \def\ppages{}
     \def\DOI{}

\newtheorem{theorem}{Theorem}[section]
\newtheorem{corollary}[theorem]{Corollary}

\newtheorem{lemma}[theorem]{Lemma}
\newtheorem{proposition}[theorem]{Proposition}
\theoremstyle{definition}
\newtheorem{definition}[theorem]{Definition}
\newtheorem{remark}[theorem]{Remark}

\newcommand{\ep}{\varepsilon}

\newtheorem{acknowledgment}[theorem]{Acknowledgment}

\numberwithin{equation}{section}

\title{  Profile  of a  touch-down   solution  to a  nonlocal  MEMS Model}

\author[G. K. Duong and   H. Zaag ]{}
\subjclass{Primary: 35K50, 35B40; Secondary: 35K55, 35K57.}
 \keywords{Blowup solution, Blowup profile, Stability, Semilinear heat equation, non variation heat equation}

\thanks{\today}
\begin{document}
\maketitle

\centerline{Giao Ky DUONG \footnote{ G. K. Duong  is supported   by the project INSPIRE. This project has received funding from the European Union’s Horizon 2020 research and innovation programme under the Marie Sk\l odowska-Curie grant agreement No 665850.}  and   Hatem ZAAG    } 
\medskip
{\footnotesize
  \centerline{Universit\'e Paris 13, Sorbonne Paris Cit\'e, LAGA, CNRS UMR 7539, Villetaneuse, France.}
}

\bigskip
\begin{center}\thanks{\today}\end{center}

\begin{abstract} 
 In this paper,   we   are  interested in   the mathematical  model  of  MEMS devices which  is   presented   by the following   equation  on  $(0,T) \times \Omega:$
\begin{eqnarray*}
 \partial_t u = \Delta u   +\displaystyle  \frac{\lambda }{  (1-u)^2 \left( 1 +\displaystyle \gamma \int_{\Omega} \frac{1}{1-u} dx \right)^2}, \quad   0 \leq u <1,  
\end{eqnarray*} 
where  $\Omega$ is a bounded  domain in  $\mathbb{R}^n$ and   $\lambda, \gamma > 0$. In  this work,  we   have succeeded to  construct   a  solution  which   quenches  in  finite time T     only  at   one interior point    $a \in \Omega$.  In particular,   we   give  a  description   of  the  quenching  behavior   according to the following  final  profile 
$$  1 - u(x,T)    \sim   \theta^*\left[  \frac{|x-a|^2}{|\ln|x-a||}  \right]^\frac{1}{3} \text{  as } x \to  a, \theta^*  > 0.$$
The  construction   relies on    some connection between  the  quenching   phemonenon  and the  blowup phenomenon. More precisely,  we   change  our problem to   the construction of  a  blowup solution  for a related PDE  and  describe  its behavior. The  method is inspired by the  work of Merle and Zaag \cite{MZnon97} with a  suitable modification.   In addition to that,   the proof  relies      on two  main steps:\textit{ A reduction   to a finite dimensional problem} and  \textit{ a topological argument  based on  Index theory}.   The main difficulty and novelty of this work is    that  we   handle    the     nonlocal   integral   term in the above equation.   The interpretation of the finite  dimensional  parameters in terms of the blowup point and the blowup time  allows  to derive  the stability of the  constructed solution with respect to initial data.
\end{abstract}

\maketitle
\section{Introduction.}
We   are interested  in   the  motion   of some  elastic membranes  which is   usually found  in  Micro-Electro Mechanical System (MEMS) devices, which are available  in a variety of electronic devices, such as microphones, transducers, sensors, actuators and so on.    Described briefly,  MEMS  devices contain  an  elastic  membrane   which  is hanged  above a  rigid  ground plate   connected in series with a fixed voltage source and a fixed capacitor.  For more details  on the physical background and possible applications, we refer the reader to  \cite{FMPSSAIM2006},   \cite{GTK2014}, \cite{PBCH2002}  and \cite{PTJEM2001}.
 
 \medskip
For a  MEMS device (in \cite{GSSAIM2015} and \cite{GTK2014}),  the distance  between the rigid  ground  plate  and the  elastic  membrane changes with time. It is refered   to as the  \textit{deflection} of the membrane. Here,    we    assume that this distance  is very small compared to the device.   In fact,  we    can fully   describe  the behavior   of the deflection  by the following hyperbolic   equation 
\begin{equation}\label{MEMS-u-hyper+para}
\left\{  \begin{array}{rcl}
 \varepsilon^2  \partial_{tt} u +  \partial_t u &=& \Delta u   + \displaystyle \frac{\lambda f(x,t) }{  (1 - u)^2 \left( 1 + \gamma \displaystyle \int_{\Omega} \frac{1}{1-u} dx \right)^2}, \quad   x \in \Omega, t > 0, \\[0.5cm]
\\
u (x,t)  & =& 0, x \in    \partial \Omega, t > 0,\\
\\
\\
u(x, 0) & =& u_0 (x), x  \in \bar  \Omega.
\end{array} \right.
\end{equation}
where   $\Omega$ is  considered as the domain of  the  rigid  plate, $u$ is the deflection    of the   membrane   to the plate, $\lambda >0, \gamma > 0$ and  $f$ is continuous. Here, the distance  between  the rest position of the membrane and the rigid plate is normalized to 1. When the device is under  voltage,  $u$ will vary in the interval $[0,1)$.    In addition to that, the  parameter $\lambda$   represents  the ratio  of  the reference electrostatic force  to the reference  elastic  force and  $\varepsilon $   is the ratio of the  interaction    of   the  inertial and damping  terms in our model. Moreover,  the function  $f $   represents    the varying   dielectric properties  of  the membrane.  See   \cite{GKDCDS2012}  for more details.

\medskip
 In fact,  we    are interested  in   a simpler  case of \eqref{MEMS-u-hyper+para} considered in  the following  parabolic equation:   
\begin{equation}\label{equa-mems-u}
\left\{  \begin{array}{rcl}
 \partial_t u &=& \Delta u   +\displaystyle \frac{\lambda }{  (1-u)^2 \left( 1 +\displaystyle \gamma \int_{\Omega} \frac{1}{1-u} dx \right)^2},  \quad x \in \Omega, t > 0, \\[0.7cm]
u (x,t) & =& 0,  x \in \partial \Omega,  t > 0,\\
u( x, 0) & =&  u_0(x), x \in \Omega.
\end{array} \right.
\end{equation}
Moreover, we  are also interested in   the  following  generalization of  problem \eqref{equa-mems-u}:
\begin{equation}\label{equa-mems-u-general}
\left\{  \begin{array}{rcl}
 \partial_t u &=& \Delta u   +\displaystyle \frac{\lambda }{  (1 - u)^p \left( 1 +\displaystyle \gamma \int_{\Omega} \frac{1}{1-u} dx \right)^q},  \quad x \in \Omega, t > 0, \\[0.7cm]
u (x,t) & =& 0,  x \in \partial \Omega,  t > 0,\\
u( x, 0) & =&  u_0(x), x \in \Omega,
\end{array} \right.
\end{equation}
 where  $p, q > 0$.
 Introducing
\begin{equation}\label{defini-Q-T}
Q_T =  (0,T) \times  \Omega, \text{ where } T> 0,
\end{equation}
we say that  $u$ is a  \textit{classical  solution} of  \eqref{equa-mems-u} (in the sense of Proposition 1.2.2  page 13 in  Kavallaris  and Suzuki \cite{NKMI2018})  if $u$   is  a  function    in   $C^{2,1} (Q_T) \cap C(\bar Q_T)$ that     satisfies  \eqref{equa-mems-u} at every point  in  $Q_T$  as well as   the boundary and  initial  conditions, with 
$$  u  (x,t) \in [0,1),  \forall x \in  \Omega, t  \in (0,T).$$
According to the above mentioned  reference in \cite{NKMI2018},  the local Cauchy problem of \eqref{equa-mems-u} is   solved. Then, either     our solution is   global in time or   there exists  $T > 0$ such that 
\begin{equation}\label{defi-quenching-finite}
 \liminf_{   t\to T} \left[  \min_{x \in \bar \Omega }\{  1 -   u(t,x)     \}  \right]  = 0.
\end{equation}
We can see that if  the above condition occurs,    the   right-hand side    of  \eqref{equa-mems-u} may become singular.   This  phenomenon     is  refered to as     \textit{ touch-down }  in finite  time $T$ in reference to the physical phenomenon, where the membrane ''touches'' the rigid ground plate which is placed below. In fact, in our setting, we follow the literature and place the regid plate at $u=1$, above the membrane  which is located at $u(x,t)$.   Note that in case of \textit{touch-down}, the MEMS device breaks down.

\bigskip
Mathematically, we may refer to the behavior in  \eqref{defi-quenching-finite} as finite-time quenching. Moreover,  $a \in \Omega$  is a  quenching point if and  only if   there exist  sequences  $(a_n ,t_n) \in  \Omega \times  (0,T)  $ such that 
$$ u(a_n, t_n) \to   1, \text{ as  }  n \to + \infty.$$

\medskip

The  \textit{ touch-down} phenomenon     has been strongly  studied      in  recent decades. In one space dimension, we  would like  to mention the paper by Guo, Hu and Wang in \cite{GHWQAM09} who gave a sufficient condition for quenching, and also a lower bound on the quenching final profile (see Remark \ref{remark-constant-Gua-Hu-Wang} below). There is also the paper by Guo and Hu in \cite{GHJDE2018} who find a constant limit for  the  similarity variables version valid only  on compact sets, and yielding  the quenching rate.

\medskip

In higher dimensions,  let us for example mention the following result by Guo and Kavallaris in \cite{GKDCDS2012}:

\medskip
\textit{
Consider  $\Omega$  such that   $ |\Omega | \leq \frac{1}{2}$. Then, for all  $\lambda > 0$ fixed and  $\gamma > 0$,   there exist initial data with a small    energy     such that    problem   \eqref{equa-mems-u} has a solution  which quenches  in finite time.}

\medskip

 In our   paper,   we   are  interested  in     proving a   general  quenching result with no  restriction on any $\lambda > 0, \gamma >0 $ and $C^2$ bounded domain  $ \Omega$.   In fact, we do much  better  that   \cite{GHJDE2018} and  \cite{GHWQAM09},   and   give  a sharp decription of   the asymptotics  of   the  solution    near  the  quenching   region. The following  is  the  main result:

\begin{theorem}[Existence    of  a  \textit{touch-down}  solution]\label{theorem-existence}
Consider    $\lambda > 0,  \gamma > 0   $ and  $\Omega  $ a $ C^2$   bounded domain   in  $\mathbb{R}^n,$  containing   the origin.  Then,  there exist  initial data  $u_0  \in C^\infty (\bar \Omega)$ such that  the  solution    of   \eqref{equa-mems-u} quenches  in finite time    $T = T(u_0) > 0$  only at  the origin.    In particular,   the following holds:
\begin{itemize}
\item[$(i)$] The   intermidiate  profile: For  all  $t\in [0,T)$
\begin{equation}\label{profile-intermidiate}
\left\| \frac{(T-t)^{\frac{1}{3}} }{ 1 - u (.,t)}     -   \theta^*  \left(  3 + \frac{9}{8}  \frac{|.|^2}{\sqrt{(T-t)|\ln(T-t)|}}  \right)^{-\frac{1}{3}} \right\|_{L^\infty (\Omega)} \leq \frac{C}{\sqrt{|\ln(T-t)|}},
\end{equation} 
for some   $\theta^*  = \theta^* (\lambda, \gamma, \Omega, T)  > 0$.  
\item[$(ii)$] The final profile:  There  exists  $u^* \in C^2(\Omega  ) \cap C (\bar \Omega) $ such that  $u $ uniformly converges   to $u^*$  as  $t \to T,$ and   
\begin{equation}\label{profile-final}
 1 -u^*(x)    \sim   \theta^* \left[  \frac{9 }{16} \frac{|x|^2}{|\ln|x||}  \right]^\frac{1}{3} \text{  as } x \to  0.
\end{equation}
\end{itemize}
\end{theorem}

\begin{remark}
Note that when   $\gamma =0$,   our problem   coincides with the work of Filippas and Guo \cite{FGQAM93} and also   Merle and Zaag \cite{MZnon97}.  Our paper is then meaningful  when  $\gamma \ne  0$, and the whole issue is how to control the non local term. Note that \cite{FGQAM93} derived  the final quenching profile, however, only in one space dimension, whereas   \cite{MZnon97} constructed  a quenching solution  in higher dimensions, proved its stability with respect to initial data, and gave its intermediate and final  profiles. 
\end{remark}

\begin{remark}
For simplicity, we choose  to write our result when the solution quenches at the origin. Of course, we can make it quenches at any arbitrary $a \in \Omega$, simply replace $x$ by $x-a$ in the statement. 
\end{remark}
\begin{remark} In  Theorem \ref{theorem-existence},     we can describe  the  evolution  of  our solution at  $x = 0$ as follows:
$$ 1 - u(0,t)      \sim     \frac{(T-t)^\frac{1}{3}}{  \sqrt[3]{3}   \theta^* }, \text{ as }  t \to T.$$
\end{remark}

\begin{remark}\label{remark-constant-Gua-Hu-Wang}
From \eqref{profile-final}, we see that  the final profile  $u^* $ has a cusp at the origin which is equivalent to  
$$   \frac{C_0 |x|^{\frac{2}{3}}}{|\ln|x||^{\frac{1}{3}}}. $$
This description is  in fact  much better than the  result of Guo, Hu and Wang  in \cite{GHWQAM09} who gave some sufficient  conditions  for quenching  in one space dimension, and proved the existence of  a cusp at the quenching point bounded  from below by  $C (\beta) |x|^\beta  $ for any $ \beta  \in \left( \frac{2}{3}, 1\right)$, which   is less accurate than our estimate \eqref{profile-final}.
\end{remark}
\begin{remark} Note that  we  can explicitly write the  formula of the initial data
\begin{equation}\label{form-initial-MEMS}
u (x,0)  =  \frac{\bar u (x,0)}{ \bar u(x,0) + 1},
\end{equation}
where 
$$  \bar u (x,0)   =    \frac{\bar  \theta (0)}{\lambda^{\frac{1}{3}}}  U (x,0), $$
with  
$$U (x,0)   =   T^{-\frac{1}{3}}   \left[   \varphi ( \frac{x}{\sqrt{T}}, - \ln T) +  \left( d_0 + d_1 \cdot z  \right)\chi_0 \left( \frac{16 |z|}{K_0^2}\right)  \right]  \chi_1 (x,0)  + (1 - \chi_1 (x,0)) H^* (x),$$
$$  z  = \frac{x}{ \sqrt{ T |\ln T|}},   $$
$$ \chi_1 (x,0)   =  \chi_0 \left(   \frac{|x|}{ \sqrt{T} |\ln T|} \right),$$
and  $\bar \theta (0)$ is the  unique positive  solution of the following equation
$$ \bar \theta (0)  = \left( 1 + \gamma |\Omega| +   \frac{\gamma}{\sqrt[3]{\lambda}} \bar \theta (0)  \int_{\Omega} U (0) dx  \right)^{\frac{2}{3}},$$
and  $  \chi_0, \varphi,   H^*$  are defined in \eqref{defini-chi-0}, \eqref{defini-varphi},  \eqref{defini-H-epsilon-0}, respectively. Here,  $T$ is small enough and  the parameters  $d_0, d_1$ are  fine-tuned in order to get the desired behavior.
\end{remark}
\begin{remark}[An open question] How big can $\theta^*$ be?   This question   is related  to the work  of Merle and Zaag in  \cite{MZnon97} (see the  Theorem on  page 1499),  which  corresponds to the case where $\gamma =0$. For that  case,  the answer is $\theta^* = \frac{1}{\sqrt[3]{\lambda}}$.  It is  very interesting  to  answer the question in    the general case.  By a  glance to \eqref{defini-2-bar-the-t}, we know that  $\theta^*$ is stricly greater than $ \frac{(1 + \gamma |\Omega|)^{\frac{2}{3}}}{\sqrt[3]{\lambda}}$.  Let us   define 
$$ \mathcal{T}_{max} =\left( \frac{(1 + \gamma |\Omega|)^{\frac{2}{3}}}{\sqrt[3]{\lambda}}, + \infty\right), $$
and 
$$ \mathcal{T} = \left\{  \theta^* \in \mathbb{R} \text{ such that } \eqref{profile-intermidiate} \text{ holds with }    u  \text{ a positive solution to  } \eqref{equa-mems-u}, \text{ for some } T >0   \right\}.$$
Then,  by a fine modification in the proof, we can construct  a solution such that  $\theta^*$ arbitrarily takes  large values in   $\mathcal{T}_{max}$. In particular, we can prove that $\mathcal{T}$ is   a dense  subset of $\mathcal{T}_{max}$.  We would like to make    the following  conjecture  
$$ \mathcal{T} =  \mathcal{T}_{max}.$$
\end{remark}

\bigskip
Now, we would like to mention that  our proof of Theorem \ref{theorem-existence} holds in a   more general setting. More precisely, if we consider  problem  \eqref{equa-mems-u-general}     in the following  regime
\begin{equation}\label{condition-p-q}
 n  - \frac{2}{ p + 1} > 0,     \text{ and }    q  > 0  \text{ and } n \geq 1,
\end{equation} 
then,    Theorem \ref{theorem-existence}  changes   as follows:
\begin{theorem}[Existence    of a \textit{touch-down}  solution to \eqref{equa-mems-u-general}]\label{theorem-existence-general}
Consider   $ \lambda, \gamma > 0, $ and   $\Omega$  a $C^2$ bounded domain  in  $\mathbb{R}^n$ and condition   \eqref{condition-p-q} holds. Then,   there exist initial data $\hat u_0 $  in  $C^\infty (\bar{\Omega})$ such that  the solution  of equation \eqref{equa-mems-u-general}  touches down in finite time   only at  the  origin. In particular,  the  following holds:
\begin{itemize}
\item[$(i)$] The   intermidiate  profile, for all  $t\in [0,T)$
\begin{equation}\label{profile-intermidiate-general}
\left\| \frac{(T-t)^{\frac{1}{p+1}} }{ 1 - u (.,t)}     -   \hat \theta^*  \left(  p + 1 + \frac{(p+1)^2}{ 4 p}  \frac{|.|^2}{\sqrt{(T-t)|\ln(T-t)|}}  \right)^{-\frac{1}{p+1}} \right\|_{L^\infty (\Omega)} \leq \frac{C}{\sqrt{|\ln(T-t)|}},
\end{equation} 
for some   $ \hat \theta^* (\lambda, \gamma, \Omega, p, q)> 0$.  
\item[$(ii)$]  The exists  $\hat u^* \in C^2(\Omega  ) \cap C (\bar \Omega) $ such that  $u $ uniformly converges   to $ \hat u^*$  as  $t \to T,$ and   
\begin{equation}\label{profile-final-general}
1- \hat u^*(x)    \sim  \hat \theta^* \left[   \frac{(p+ 1)^2}{8 p} \frac{|x|^2}{|\ln|x||}  \right]^\frac{1}{p+1} \text{  as } x \to  0.
\end{equation}
\end{itemize}
\end{theorem}
\begin{remark}
When $\gamma = 0$, we have no nonlocal term in the equation,  and this result was already proved  in \cite{MZnon97}. We also  mention that in one space dimension,  the final profile in \eqref{profile-final-general}  was also derived  in \cite{FGQAM93}. 
\end{remark}

\begin{remark}
We don't  give the proof  of Theorem \ref{theorem-existence-general} here because the  techniques  are the same as for Theorem \ref{theorem-existence}.  In fact,   for simplicity, we will only give  the proof    for the MEMS  case
$$ p = q = 2,$$
considered  in equation \eqref{equa-mems-u} and Theorem \ref{theorem-existence}.  Of course,  all  our   estimates can be  carried  on for   the general case. 
\end{remark}

In addition to that,    we can apply     the techniques   of Merle in \cite{Mercpam92} to create  a soluiton   which  quenches  at  arbitrary given points.
\begin{corollary}\label{corollary-k-point} For any  $k$ points   $ a_1, a_2, ...., a_k$ in $\Omega,$ there exist  initial data  such that  \eqref{equa-mems-u-general} has  a solution  which  quenches exactly  at  $a_1,..., a_k$. Moreover,    the  local behavior  at  each  $a_i$ is also  given    by  \eqref{profile-intermidiate-general}, \eqref{profile-final-general} by replacing   $x$ by  $x - a_i$    and    $L^\infty (\Omega)$ by $L^\infty (|x - a_i| \leq  \omega_0),$ for some  $\omega_0 > 0,$ small enough.
\end{corollary}

As  a  consequence  of  our   techniques, we can derive the  stability of the  quenching  solution  which we  constructed   in Theorem   \ref{theorem-existence-general} under the perturbations of  initial data.  
\begin{theorem}[Stability of  the  constructed solution]\label{theor-stability}   Let us consider $ \hat u,  $ the solution which we constructed  in  Theorem  \ref{theorem-existence-general} and we   also  define  $\hat T$ as  the quenching time of the solution and $\hat \theta^*$ as the coefficient in front of the profiles \eqref{profile-intermidiate-general} and  \eqref{profile-final-general}. Then, there exists   an open  subset    $  \hat{\mathcal{U}}_0  $   in $ C (\bar{\Omega}),$ containing  $\hat  u (0)$ such that   for all initial data  $u_0 \in   \hat{\mathcal{U}}_0,$   equation  \eqref{equa-mems-u-general}  has  a unique solution   $u$ quenching  in finite time  $T (u_0)$  at  only one quenching  point  $a(u_0)$.  Moreover,  the   asymptotics  \eqref{profile-intermidiate-general} and  \eqref{profile-final-general}  hold by replacing $\hat{u} (x,t)$ by $u (x - a (u_0), t),$ and $\hat \theta^*$ by some $\theta^* (u_0)$ . Note that, we have 
$$  (a(u_0), T(u_0), \theta^*(u_0)) \to   (0, \hat T, \hat \theta), \text{ as }   \|u_0  - \hat u_0\|_{C(\bar \Omega)} \to 0.$$
 \end{theorem} 
 Let us now  comment on the organization of the paper. As we have stated  earlier, Theorem \ref{theorem-existence} is a special case of Theorem \ref{theorem-existence-general}. For simplicity in the notations, we only  prove Theorem \ref{theorem-existence}. The interested  reader may derive  the general case by inspection. Moreover, we don't  prove  Corollary \ref{corollary-k-point}  and Theorem \ref{theor-stability}, since the former follows  from Theorem \ref{theorem-existence-general} and the techniques of  Merle in \cite{Mercpam92}, and the latter follows also from Theorem \ref{theorem-existence-general}  by the method of Merle and Zaag in \cite{MZdm97}. In conclusion, we only prove Theorem \ref{theorem-existence} in this paper.
 
 \medskip
  The paper is organized  as follows:
  
  - In Section 2, we give  a different formulation of the problem, and show how the profile in  \eqref{profile-intermidiate}  arises naturally.
  
  - In Section 3, we give  the proof  without technical details.
  
  - In Section 4, we prove  the technical details.
  
Some appendices are added at the end.

\begin{acknowledgment}
We would  like  to thank  the referees for the careful reading of the paper and for pointing  out various earlier references.
\end{acknowledgment}
\section{ Setting  of the problem}

\subsection{ Our main idea} 
We aim  in   this subsection    at  explaining  our  key  idea   in this paper. The rigorous proof will be given    later. Introducing
\begin{equation}\label{defini-alpha-t}
\alpha (t)  = \displaystyle \frac{\lambda}{ \left( 1   +  \gamma  \displaystyle \int_{\Omega}  \frac{1}{1 -u(t)}  dx \right)^2},
\end{equation}
  we rewrite   \eqref{equa-mems-u}   as the  following
 \begin{equation}\label{equa-u-alpha-t}
  \partial_t u  =  \Delta u   +  \frac{\alpha (t)}{(1 - u)^2}.
 \end{equation}
Under this general form, we see our equation \eqref{equa-mems-u} as a step by step generalization,  starting from a much simpler context:

\medskip

- \textbf{Problem 1: Case where  $\alpha (t) \equiv \alpha_0$}. This case was considered by Merle and Zaag in \cite{MZnon97} where, the  authors  constructed    a solution  $u_{\alpha_0}$ satisfying 
 $$  u_{\alpha_0}(x,t) \to 1 \text{ as }  (x,t) \to (x_0,T) , $$
 for some  $ T > 0,$   and  $x_0 \in \Omega$. In particular, they    gave  	 a sharp    description  for the  quenching   profile. Technically, the authors  in that work introduced 
 $$  \bar u = \frac{1}{1 - u} - 1 = \frac{u}{1 - u},$$
 and constructed  a blowup solution for the following equation derived  from  \eqref{equa-u-alpha-t}:
 \begin{equation}\label{equa-bar-u-merle-zaag}
 \partial_t \bar u = \Delta \bar u - 2\frac{|\nabla \bar u|^2}{\bar u} + \alpha_0 \bar u^4, \text{ with } \alpha (t) \equiv \alpha_0, 
 \end{equation}
 (see equation (III), page 1500 in \cite{MZnon97} for more details).  

\medskip
- \textbf{Problem 2: Case  where  $  0 <  \alpha_1 \leq   \alpha (t) \leq  \alpha_2 $ for all $ t> 0 $ for some  $  0 < \alpha_1 < \alpha_2$}.   This  case is indeed a reasonable generalization which follows  with no difficulty from the stategy of  \cite{MZnon97} for \textbf{Problem 1}.

\medskip
 
 - \textbf{Problem 3: Equation \eqref{equa-mems-u}}.  Our  idea  here  is to see \eqref{equa-mems-u}  as a coupled system between \textbf{Problem 2} and  \eqref{defini-alpha-t}:
 \begin{eqnarray*}
 \left\{  \begin{array}{rcl}
  \partial_t u  &=  \Delta u   +  \frac{\alpha (t)}{(1 - u)^2}, \\[0.4cm]
 \alpha (t)  &=  \frac{\lambda}{ \left( 1 + \gamma \int_\Omega \frac{1}{1 - u} dx \right)^2} .
 \end{array} \right.
\end{eqnarray*}  
  
 A simple idea would be to try a kind  of fixed-point argument starting from  some solution to \textbf{Problem 1}, then defining  $\alpha (t)$ according to \eqref{defini-alpha-t} defined with this solution, then solving \textbf{Problem 2} with this $\alpha (t)$, then defining a new $\alpha (t)$ with the new solution, and so  forth.
 
 \medskip
 In order to make this  method to work, one has to check  whether the iterated  $\alpha (t)$ stay  away from $0$ and  $+ \infty$, as requested  in the context of   \textbf{Problem 2}. We checked whether this holds  when $u$ solves \textbf{Problem 1}.  Fortunately, this was   the case, and this gave us a serious hint to  treat our equation \eqref{equa-mems-u} as a pertubation of  \textbf{Problem 1}.
 
 \bigskip
 In fact, our proof uses no interation, and we diredly apply the stategy of  Merle and Zaag  in \cite{MZnon97}  to control the various terms (including the  nonlocal term), in order to find  a solution  which stays near  the desired behavior. 
\subsection{Formulation  of the problem}
 In this  section, we aim at   setting the  mathematical   framework  of our   problem.       The     rigorous proof will be given later.  Our  aim  is  to  construct a solution  for equation  \eqref{equa-mems-u}, defined for all  $(x, t) \in  \Omega \times [0,T),$ for some  $T > 0$ with $0 \leq  u (x,t) < 1,$ and   
$$ u(x,t)  \to  1  \text{ as }  (x,t) \to (x_0,T), $$ 
 for some  $ x_0 \in \Omega$.   Without loss of  generality,
  we   assume that 
 $$ x_0 = 0 \in  \Omega.$$
 Introducing, 
\begin{equation}\label{defini-bar-u}
\bar  u  =   \frac{1}{1 - u}  - 1 =   \frac{u}{1 -u} \in [0, + \infty),
\end{equation}
we   derive  from \eqref{equa-mems-u}  the  following   equation  on  $\bar u$  
\begin{equation}\label{equa-bar-u}
\left\{  \begin{array}{rcl}
\displaystyle  \partial_t \bar u  &=&  \Delta \bar u   - 2 \frac{ |  \nabla  \bar u|^2}{ \bar  u + 1}  + \displaystyle \frac{\lambda (\bar u + 1 )^4}{ (1 + \gamma|\Omega|  +\displaystyle \gamma  \int_{\Omega} \bar u dx)^2}, x \in \Omega, t > 0,  \\[0.5cm]
\bar  u (x,t)  & =& 0, x \in \partial \Omega, t > 0,\\[0.2cm]
\bar u(x, 0) & =& \bar u_0 (x),\quad     x \in \bar \Omega.
\end{array} \right.
\end{equation}
  Our aim  becomes  then to construct a blowup solution for equation \eqref{equa-bar-u} such that
  $$\bar u (0,t) \to + \infty   \text{ as } t \to T.$$
  In order to see  our equation as a (not so small) perturbation of  the standard case in \eqref{equa-bar-u-merle-zaag}, we suggest   to make  one more scaling by introducing
\begin{equation}\label{defini-U-x-t}
U (x,t) = \frac{\lambda^{\frac{1}{3}}}{ \bar \theta (t)}  \bar u (x,t), \quad  U (x,t)  \geq  0, \quad \forall (x,t) \in \Omega \times [0,T),
\end{equation}
with 
\begin{equation}\label{defini-2-bar-the-t}
\bar  \theta (t)  =  \left( 1 +  \gamma |\Omega| + \gamma  \int_{\Omega} \bar u (t) dx  \right)^{\frac{2}{3}}.
\end{equation}
Then,  thanks to    equation   \eqref{equa-bar-u}, we  deduce   the  following equation to be satisfied by   $U$: 
\begin{equation}\label{equa-U}
\left\{  \begin{array}{rcl}
\displaystyle  \partial_t  U  &=&  \Delta U    - 2 \frac{ |  \nabla  U|^2}{ U  +  \frac{\lambda^{\frac{1}{3}}}{\bar \theta(t)}}   + \left( U +   \frac{\lambda^{\frac{1}{3}}}{ \bar \theta (t)}  \right)^4  - \frac{\bar \theta ' (t)}{\bar  \theta (t)} U,  x \in \Omega, t> 0,  \\[0.5cm]
 U (x,t) & =& 0,    x \in   \partial \Omega, t > 0,\\[0.2cm]
U(x,0 ) & =& U_0 (x), x \in \bar \Omega.
\end{array} \right.
\end{equation}
 Note that in the  blowup regime, which is our focus, $U$ is large and equation \eqref{equa-U}  appears indeed as a perturbation  of equation \eqref{equa-bar-u-merle-zaag}. 
 
 \noindent
Introducing   the following  notation  
\begin{eqnarray}
\bar \mu (t) & = &  \int_{\Omega} U  (t) dx \label{defini-bar-mu},
  \end{eqnarray}
  we may   rewrite  \eqref{defini-2-bar-the-t} as the following equation 
 \begin{equation}\label{relation-bar theta-bar-mu}
\bar  \theta (t) = \left(  1 + \gamma |\Omega| +   \frac{\gamma}{ \lambda^\frac{1}{3} }  \bar \theta (t) \bar  \mu (t)  \right)^\frac{2}{3}.
 \end{equation}
This implies   that $ \bar\theta (t)$ solves the following cubic equation
\begin{equation}\label{ordre 3-equ-bat-theta-bar-mu}
\theta^3(t)  = \left( 1 + \gamma |\Omega |  +  \frac{\gamma}{ \lambda^{\frac{1}{3}}}  \bar \theta (t) \bar \mu (t)\right) ^2 = (A + B(t) \bar \theta (t))^2 =  A^2 + 2 A B(t) \bar \theta (t) + B ^2(t)  \bar \theta^2 (t),
\end{equation}
where
$$ A = 1 + \gamma |\Omega| \text{ and }   B (t) = \frac{\gamma }{ \lambda^{\frac{1}{3}}} \bar \mu (t).$$
Since it happens  that  $\bar \theta (t) $ is the unique positive solution of  \eqref{ordre 3-equ-bat-theta-bar-mu},   we  may    solve  \eqref{ordre 3-equ-bat-theta-bar-mu}   and  express $\bar \theta (t)$ in terms of $\bar \mu(t)$  as follows

\begin{eqnarray}
\bar{\theta} (t) &= &\frac{  \sqrt[3]{27 A^2 +3\sqrt[3]{3}  \sqrt{27 A^2 + 4 A^3 B^3(t)} + 18 A B^3 (t) +2 B^6 (t)} }{3 \sqrt[3]{2}}   + \frac{B^3(t)}{3}\label{defini}\\
& +&  \frac{\sqrt[3]{2} (6 AB (t) + B^4(t))}{\sqrt[3]{27 A^2 +3\sqrt[3]{3}  \sqrt{27 A^2 + 4 A^3 B^3(t)} + 18 A B^3(t) +2 B^6(t)} }. \nonumber
\end{eqnarray}

\medskip
\noindent
Particularly,  we   show here  the equivalence between  equation  \eqref{equa-bar-u} and \eqref{equa-U}.
  \begin{lemma}[Equivalence between \eqref{equa-bar-u} and \eqref{equa-U}]\label{equivalent-bar u-U} Consider $\lambda > 0, \gamma > 0$ and $\Omega $ a bounded  domain   in $\mathbb{R}^n$. Then, the following   holds:
  
$(i)$  We consider  $\bar u $ a solution of equation  \eqref{equa-bar-u} on $[0,T),$ for some $T> 0$ and  introduce
 $$U (t) =  \frac{\lambda^{\frac{1}{3}}}{\bar \theta (t)} \bar u (t),  
 $$
 where  $\bar  \theta (t)  = \left( 1 + \gamma |\Omega| + \gamma \int_{\Omega} \bar u(t) dx\right)^{\frac{2}{3}}$. Then,  $U$ is   a solution of equation \eqref{equa-U} on $[0,T)$.
 
 $(ii)$ Otherwise,  we consider  $U $  a solution  of  equation \eqref{equa-U}  on  $ [0,T),$ for some $T> 0$ and  introduce
 $$   \bar u (t)   = \frac{\bar \theta (t)}{\lambda^{\frac{1}{3}}}  U(t), \forall t \in [0,T), $$
where     $\bar  \theta (t)$ is defined as in   relation \eqref{relation-bar theta-bar-mu},  then $\bar  u $ is the solution of equation \eqref{equa-bar-u} on $[0,T)$.  In particular, the uniqueness of the solution is preserved.
 \end{lemma}
 \begin{proof}
 The proof is  easily deduced from  the definition in this lemma. We  kindly  ask  the   reader to  self-check.
 \end{proof}
 \begin{remark}
From   settings    \eqref{defini-bar-u} and \eqref{defini-U-x-t} and the   local well-posedness of equation \eqref{equa-mems-u}  in the sense  of  classical solutions  (see  Proposition 1.2.2 at  page 12 in Kavallaris and Suzuki \cite{NKMI2018}), we    can derive the   local existence and  uniqueness of classical solutions of equations \eqref{equa-bar-u} and \eqref{equa-U}.   Since nonnegativity is preserved  for  these equation, we will assume that $\bar u$ and $U$ are  nonegative.
 \end{remark}
Thanks to Lemma  \ref{equivalent-bar u-U}, our problem is  reduced  to  constructing  a nonnegative solution to  \eqref{equa-U}, which blows up   in finite time  only at  the origin. We also    aim at    describing its    asymptotics at the singularity. 

\noindent
Since we defined $U$ in \eqref{defini-U-x-t} on purpose so that  \eqref{equa-U} appears as  a perturbation  of equation \eqref{equa-bar-u-merle-zaag} for $U$ large, it is reasonable to make the following hypotheses:
\begin{itemize}
\item[$(i)$] $1 \leq  \bar \theta (t) \leq C_0$ for some $C_0  > 0 $. Note that  from \eqref{relation-bar theta-bar-mu}, we have  $ \bar \theta (t) \geq 1.$\label{condition-bund-bar-theta}\\
\item[$(ii)$] $ |\bar \theta' (t)|  \ll U^3 (t) \text{ when } U \text{ large}.$   
\end{itemize}

\noindent
It is  then  reasonable   to expect for equation \eqref{equa-U} the same profile as the one   constructed in \cite{MZnon97} for equation \eqref{equa-bar-u-merle-zaag}. So,  it is   natural to  follow  that work by  introducing the following    \textit{Similarity-Variables}:
\begin{equation}\label{similarity-variables}
W(y,s)  =  (T-t)^{\frac{1}{3}} U(x,t), \text{ and }  s = - \ln(T-t) \text{ and  }  y   = \frac{x}{\sqrt{T-t}}.
\end{equation}
Using   equation  \eqref{equa-U}, we  write the equation  of  $W$ in  $(y,s)$ as follows
\begin{equation}\label{equa-W}
\left\{  \begin{array}{rcl}
\partial_s W &=&  \Delta W -\frac{1}{2} y \cdot \nabla W  -\frac{W}{3}  - 2 \frac{\left| \nabla W\right|^2}{W   +    \frac{\lambda^{\frac{1}{3}}  e^{-\frac{s}{3}}}{ \theta (s)}}  + \left(W  +  \frac{\lambda^{\frac{1}{3}}e^{-\frac{s}{3}}}{\theta (s)}    \right)^{4}  - \frac{ \theta' (s)}{\theta (s)}  W,  \\[0.5cm]
 W (y,s) & =& 0,    y \in   \partial \Omega_s, s  > -\ln T,\\[0.2cm]
W(y, -\ln T ) & =& W_0 (y), y \in \bar \Omega_s,
\end{array} \right.
\end{equation} 
where  
 \begin{equation}\label{defini-theta-s}
 \theta (s) =  \bar  \theta (t(s)) = \bar \theta (T -  e^{-s} ),
\end{equation} 
and  
\begin{equation}\label{defini-Omega-s}
 \Omega_s  =     e^{\frac{s}{2}} \Omega,
\end{equation}
with  $\bar \theta $  satisfies   \eqref{relation-bar theta-bar-mu} and  \eqref{defini}.   

\medskip
\noindent
  We  observe in equation \eqref{equa-W}  that      $\Omega_s$  changes    as  $s \to + \infty$.     This is a major difficulty in comparison with  the situation where  $\Omega = \mathbb{R}^n$. In order to  overcome this difficulty, we  intend to introduce some cut-off of the solution,  so that we reduce  to the case $\Omega = \mathbb{R}^n$. Of course, there is a price to pay, in the sense that   we will need to handle  some cut-off terms. Our model for this will be  the work made   by   Mahmoudi, Nouaili  and  Zaag in \cite{MNZNon2016}  for the construction of a  blowup solution  to  the semilinear heat equation defined on the circle. Let us  note that  the situation with $\Omega$ bounded was  already mentioned in \cite{MZnon97}. However, the authors in that work avoided  the problem by giving the proof only in the case where $\Omega  = \mathbb{R}^n$.  In this work, we are happy to handle   the case with a bounded $\Omega$, following  the ideas  of       Mahmoudi, Nouaili  and  Zaag in \cite{MNZNon2016}.    Let us mention that Vel\'azquez was also faced in \cite{VJDE93} by the question of reducing a problem  defined on a bounded interval  to a problem considered on the whole real line. He made the reduction thanks to the extension  of the solution  defined on the interval to another solution defined on the whole line, thanks to some  truely $1$-$d$ techniques. In our case, given that we work in higher dimensions, we use a different method, based on the localization of the equation, thanks to some cut-off functions.

  \medskip
    More precisely,  we  introduce the following  cut-off function 
\begin{equation}\label{defini-chi-0}
 \chi_0 \in C_0^\infty ([0,+\infty)), \quad supp (\chi_0) \subset [0,2], \quad 0 \leq \chi_0(x) \leq 1, \forall x \text{ and  }   \chi_0 (x)= 1, \forall x \in [0,1].
 \end{equation}
Then,   we   define  the following  function
\begin{equation}\label{defini-psi-M-0-cut}
  \psi_{M_0} (y,s)  =   \chi_0 \left(   M_0  y e^{- \frac{s}{2}}       \right), \text{ for some } M_0 > 0.
\end{equation}
Let us  introduce
\begin{equation}\label{defini-w-small}
w (y,s)  = \left\{   \begin{array}{rcl}
W (y,s)  \psi_{M_0} (y,s)  &  \text{  if }  &  y \in  \Omega_s,\\[0.3cm]
0  & & \text{otherwise }.  
\end{array}  \right.
\end{equation}
We  remark  that  $w$ is defined  on $\mathbb{R}^n$ and $ s \geq   - \ln T $  and     $w \equiv   0  $ whenever  $|y| \geq \frac{2}{M_0}  e^{\frac{s}{2}} $.  Note that  $M_0$ will    be fixed large enough together with  others  parameters at the end of  our proof.

\medskip
\noindent
Using  equation \eqref{equa-W}, we  derive from \eqref{relation-bar theta-bar-mu}  the  equation satisfied by $w$ as follows
\begin{equation}\label{equa-w}
\partial_s w = \Delta w  -  \frac{1}{2} y \cdot \nabla w - \frac{1}{3} w  - 2 \frac{|\nabla w|^2}{ w  + \frac{\lambda^{\frac{1}{2}}  e^{-\frac{s}{3}}}{\theta (s)}}     +  \left(w  +  \frac{\lambda^{\frac{1}{3}}e^{-\frac{s}{3}}}{\theta (s)}    \right)^{4}  - \frac{ \theta' (s)}{\theta (s)}  w + F(w,W),
\end{equation}
where  $F(w,W)$  encapsulates  the cut-off terms and   is  defined  as  follows
\begin{equation}\label{defini-F-1}
F(w,W) = \left\{   \begin{array}{rcl} & &
W \left[ \partial_s \psi_{M_0}   -  \Delta \psi_{M_0}  +  \frac{1}{2} y \cdot \nabla \psi_{M_0} \right] - 2  \nabla \psi_{M_0} \cdot \nabla W\\[0.3cm]
 & &+  2 \frac{|\nabla w|^2}{ w  + \frac{\lambda^{\frac{1}{2}}  e^{-\frac{s}{3}}}{\theta (s)}} - 2 \frac{|\nabla W|^2 \psi_{M_0}}{ W  + \frac{\lambda^{\frac{1}{2}}  e^{-\frac{s}{3}}}{\theta (s)}} +  \psi_{M_0}\left(W  +  \frac{\lambda^{\frac{1}{3}}e^{-\frac{s}{3}}}{\theta (s)}    \right)^{4} -  \left(w  +  \frac{\lambda^{\frac{1}{3}}e^{-\frac{s}{3}}}{\theta (s)}    \right)^{4} \\[0.5cm]
 & & \text{ if } y\in \Omega e^{\frac{s}{2}},\\[0.4cm]
 & &  0   \text{ otherwise}. 
\end{array}   \right. 
\end{equation}
We   remark  that  $F \equiv 0 $ on   the region $ \{ y \in \mathbb{R}^n \left| \right.  |y|   \leq \frac{1}{M_0} e^{\frac{s}{2}}  \text{ or }   |y| \geq  \frac{2}{M_0} e^{\frac{s}{2}}  \}$ and  that we  have from the conditions $(i)$ and $(ii)$ on $\bar \theta (t)$ on page  \pageref{condition-bund-bar-theta} that 
$$  1 \leq \theta (s) \leq C_0,  \text{ and } | \theta' (s)| \ll W^3(y,s)  . $$
 Making the further  assumption  that 
$$  \theta' (s) \to 0, $$
we see that    equation  \eqref{equa-Q} is  almost  the same as equation (15) at page 1502  in  \cite{MZnon97} at least  when $|y| \leq \frac{e^{\frac{s}{2}}}{M_0}$   . Hence,    it  is reasonable  to expect  for  equation  \eqref{equa-w}  the same profile  as the authors found in \cite{MZnon97} for equation (15) in that work, namely  
\begin{equation}\label{defini-varphi}
\varphi (y,s)  =  \left(  3 +  \frac{9}{8} \frac{|y|^2}{s} \right)^{-\frac{1}{3}} + \frac{  (3)^{-\frac{1}{3}} n}{ 4 s },
\end{equation}
(note that, this  profile  was also  defined    in \cite{MZnon97} for a general $p >2$, and  that here we need to take    $p =4  $ hence  $\kappa =  (3)^{-\frac{1}{3}}$).  In particular, we would like to construct  $w$ as a  perturbation  of $\varphi$. So, we  introduce  the following function
\begin{equation}\label{defini-q}
q =  w   -   \varphi . 
\end{equation}
 Using   equation  \eqref{equa-W}, we easily write  the   equation of   $q$ 
 \begin{equation}\label{equa-Q}
\partial_s q   =  ( \mathcal{L} + V)  q +   T(q)   +  B(q)  + N(q)+ R(y,s) + F(w,W),   
\end{equation}
where
\begin{eqnarray}
\mathcal{L}  & =&  \Delta  - \frac{1}{2} y  \cdot  \nabla  + Id,\label{defini-ope-mathcal-L}\\
V(y,s) & = &        4 \left(  \varphi^3(y,s)    -  \frac{1}{3}     \right) ,\label{defini-potential-V}\\  
J(q, \theta (s))   &=& -2   \frac{\left|  \nabla q +  \nabla \varphi   \right|^2}{ q +  \varphi + \frac{\lambda^{\frac{1}{3}}e^{-\frac{s}{3}}}{\theta (s)} }      + 2 \frac{ |\nabla  \varphi  |^2  }{\varphi + \frac{\lambda^{\frac{1}{3}}e^{-\frac{s}{3}}}{\theta (s)} }, \label{defini-T-Q}  \\
B(q)  &    =&     \left(q +  \varphi  + \frac{\lambda^{\frac{1}{3}}e^{-\frac{s}{3}}}{\theta (s)} \right)^4  -   \varphi^4 -  4 \varphi^3 q      , \label{defini-B-Q}\\
R(y,s) & = &  -   \partial_s\varphi +       \Delta \varphi - \frac{1}{2}  y \cdot \nabla \varphi - \frac{\varphi}{3}     + \varphi^4   - 2 \frac{|\nabla  \varphi |^2}{\varphi +  \frac{\lambda^{\frac{1}{3}}e^{-\frac{s}{3}}}{\theta (s)} }, \label{defini-rest-term}\\
N (q)  & =&  - \frac{ \theta' (s)}{\theta (s)}   \left(  q+ \varphi \right), \label{defini-N-term} 
\end{eqnarray}
with $\theta (s)$  defined in \eqref{defini-theta-s} and  $F(w,W)$  given in \eqref{defini-F-1}.

\medskip
In particular,   we assume that  $U$   and  $q$  have   good conditions  such that     Lemmas  \ref{lemma-bound-B-Q}, \ref{lemma-bound-T-Q},   \ref{lemma-Bound-R},  \ref{lemma-Bound-N-Q} and  \ref{lemma-F-y-s}  hold. Then,  it is easy to see that  all terms  in the  right-hand side of \eqref{equa-Q}    become very   small, except for    $(\mathcal{L} + V)q$.  As a matter of fact, this term   plays the most important  role   in our analysis.   Therefore, we show here    some  main properties  on   the  linear operator $\mathcal{L}$ and the    potential $V$ (see more details   in \cite{DNZtunisian-2017}, \cite{DU2017})

- \textit{Operator $\mathcal{L}$:} This operator is  self-adjoint in  $\mathcal{D} (\mathcal{L}) \subset L^2_\rho (\mathbb{R}^n),$  where  $L^2_\rho$ is  defined  as  follows
$$   L^2_\rho  (\mathbb{R}^n) = \left\{        f \in L^2_{loc} (\mathbb{R}^n) \left| \right. \int_{\mathbb{R}^n}  |f(y)|^2 \rho (y) dy  < + \infty   \right\},$$
and
$$ \rho (y)  =  \frac{e^{-\frac{|y|^2}{4}}}{ (4\pi )^{\frac{n}{2}}}. $$
This is the spectrum set of operator $\mathcal{L}$ 
$$   \text{Spec} (\mathcal{L}) =  \left\{    1 -  \frac{m}{2}    \left|  \right.      m \in \mathbb{N}\right\}   .$$
The eigenspace  which  corresponds  to  $\lambda_m = 1  - \frac{m}{2}$ is  given  
$$     \mathcal{E}_m =  \left\{   h_{m_1} (y_1). h_{m_2} (y_2).... h_{m_n} (y_n)   \left|  \right.  m_1 + ...+m_n = m    \right\}, $$
where   $h_{m_i}$ is the  (rescaled ) Hermite  polynomial in one  dimension. 

- \textit{ Potential $V$:} It  has  two   impotant  properties:

\begin{itemize}
\item[$(i)$] The  potential  $V(., s) \to 0$  in $L^2_\rho(\mathbb{R}^n)$ as  $s \to + \infty$:  In particular,    in the  region  $|y| \leq K_0 \sqrt s$ ( the singular domain),     $V$ has  some weak perturbations    on   the   effect of   operator  $\mathcal{L}$.
\item[$(ii)$]   $V(y,s)$ is almost a constant  on the region $|y| \geq K_0 \sqrt s$:   For all  $\epsilon > 0$, there  exists   $ \mathcal{C}_\epsilon > 0$ and  $s_\epsilon$ such that 
$$     \sup_{ s \geq  s_\epsilon, \frac{|y|}{ \sqrt s}  \geq \mathcal{C}_\epsilon} \left|  V(y,s)   - \left( -\frac{4}{3}\right)  \right|   \leq  \epsilon. $$
Note  that   $ - \frac{4}{3 }  <   -1  $  and  that  the  largest  eigenvalue  of  $\mathcal{L}$ is  $1$. Hence, roughly speaking, we may assume that  $\mathcal{L} + V$ admits a  strictly  negative  spectrum. Thus,  we  can easily control   our solution in the  region $\{ |y| \geq K_0 \sqrt{s} \}$ with $K_0$ large enough.
 \end{itemize}

\medskip 
 From these properties, it appears that the behavior   of $\mathcal{L} + V$  is not  the  same  inside  and outside   of  the  singular  domain $\{ |y| \leq   K_0 \sqrt{s}\}$.   Therefore, it is natural  to   decompose  every    $r  \in L^\infty (\mathbb{R}^n)$ as follows:
\begin{equation}\label{R=R-b+R-e  }
r (y)=   r_b (y) +  r_e (y)  \equiv  \chi (y,s)   r(y) + (1- \chi (y,s) )    r (y),
\end{equation} 
where   $\chi (y,s)$  is defined as follows 
\begin{equation}\label{defini-chi-y-s}
\chi (y,s)  = \chi_0 \left(  \frac{|y|}{K_0 \sqrt s}   \right),
\end{equation}
and  $\chi_0 $ is given in \eqref{defini-chi-0}.   From the above  decomposition, we immediately
have the following: 
\begin{eqnarray*}
\text{Supp }(r_b)  & \subset & \{ |y| \leq 2 K_0 \sqrt{s} \},\\
\text{Supp }(r_e) & \subset & \{ |y| \geq K_0 \sqrt{s}\}.
 \end{eqnarray*}
In addition to that,  we  are interested in expanding    $r_b$  in   $L^2_\rho \left( \mathbb{R}^n \right)$ according to the basis  which  is created by    the eigenfunctions of  operator $\mathcal{L}$:
\begin{eqnarray*}
r_b (y)  & =  &r_0  + r_1 \cdot y +   y^T \cdot r_2 \cdot y  - 2 \text{ Tr}(r_2)  + r_-(y), \\
& \text{ or } & \\
r_b (y) & = & r_0  + r_1 \cdot y +    r_\perp  (y),
\end{eqnarray*}
where  
\begin{equation}\label{defini-R-i}
r_i =   \left(   P_\beta ( r_b )  \right)_{\beta \in \mathbb{N}^n, |\beta|= i}, \forall  i \geq 0,
\end{equation}
with $P_\beta(r_b)$  being  the projection  of  $r_b$ on   the   eigenfunction    $h_\beta$ defined as follows:
\begin{equation}\label{defin-P-i}
P_\beta (r) =  \int_{\mathbb{R}^n}  r_b   \frac{h_\beta}{\|h_\beta\|_{L^2_\rho}}   \rho dy, \forall \beta \in \mathbb{N}^n,
\end{equation}
and
\begin{equation}\label{defini-R-perp}
r_\perp =  P_\perp (r)  =   \sum_{\beta \in \mathbb{R}^n, |\beta| \geq 2}   h_\beta P_\beta (R_b),
\end{equation}
and
\begin{equation}\label{defini-R--}
r_-   =   \sum_{\beta \in \mathbb{R}^n, |\beta| \geq 3}   h_\beta P_\beta (r_b).
\end{equation}
In other  words,  $r_\perp $ is   the part  of  $r_b$  which    is  orthogonal   to the   eigenfunctions   corresponding to   eigenvalues $0$ and  $1$  and  $r_-$  is  orthogonal   to the  eigenfunctions   corresponding to   eigenvalues $ 1, \frac{1}{2}$ and  $0$.     We  should note that  $r_0$ is a scalar,  $r_1$ is a vector  and  $r_2$ is a  square  matrix   of   order $n$  and  that they are  the components of $r_b$ not $r$.        Finally,   we   write   $r$  as follows
\begin{eqnarray}
r (y) &  =  &  r_0 + r_1 \cdot y  +  y^T \cdot r_2 \cdot y  - 2 \text{ Tr}( r_2)   	 + r_- (y) + r_e (y),\label{represent-non-perp}\\
& \text{ or }&\nonumber \\
r (y) &  =  &   r_0 +  r_1 \cdot y  +  r_\perp  (y)    + r_e (y).\label{represent-with-perp}
\end{eqnarray}

\medskip
\textbf{   \textit{A summary   of  our problem:}}     Even though  we   created  many  extra functions  from  $U$  to  $q$,  we   always    concentrate  on   solution  $U$  to  equation \eqref{equa-U}.    More precisely,  we  aim at constructing   $U$  blowing up in finite time.  Then, we  will use  equation \eqref{equa-Q} as a crucial formulation  in  our  proof. Indeed,   in order to control $U$  blowing up  in finite time, it is enough to  control   the  transform function $q$  of  $U$      (see definitions  \eqref{similarity-variables},    \eqref{defini-w-small}   and \eqref{defini-q}) satisfying
\begin{equation}\label{norm-Q-infinity-to-0}
 \left\|  q(.,s)    \right\|_{L^\infty (\mathbb{R}^n)}  \to  0, \text{ as  } s \to + \infty. 
\end{equation}
\section{The  proof   of   the existence  result  assuming   technical details}
In this  section, we aim at giving  a proof without     technical details to Theorem \ref{theorem-existence}. We would like to summarize the structure of this section as follows:

\textit{- Construction of a shrinking set:} We rely here on  the ideas of the   Merle and Zaag's work in \cite{MZnon97} to introduce a shrinking set that will guarantee the convergence to zero  for $q$ defined in \eqref{defini-q}. This set will constrain our solution as we want.    Once our solution  is trapped in,  we may show the  main  asymptotics of our solution. In particular, \eqref{norm-Q-infinity-to-0} holds and our result follows.

\textit{ - Preparation  of initial data:}   We    introduce a familly of initial data  to equation \eqref{equa-U} depending on some  finite set parameters. As a matter of fact, we will  choose  these parameters such that  our solution  belongs to the shrinking set for all $t \in [0,T)$.

\textit{ - The existence of a trapped solution:} Using  a reduction  to a finite  dimensional problem (corresponding to the finite parameters introduced in our initial data) and a topological argument, we  can derive  the existence of a blowup solution in finite time, trapped in the shrinking set.   More precisely, we show  in this part that there exist  initial data in that family of initial data such that our solution  is completely confined in the shrinking set.

\textit{- The conclusion  of  Theorem \ref{theorem-existence}:} Finally, we rely on   the existence of a blowup solution, trapped in the shrinking set  to  get the conclusion of  Theorem \ref{theorem-existence}.

\subsection{ Shrinking set}
In order to control  the  solution    $U$   blowing up in finite time and satisfying    \eqref{norm-Q-infinity-to-0},   we   adopt the   general ideas given by Merle and  Zaag in  \cite{MZnon97}.   For each $K_0 > 0,  \epsilon_0 > 0, \alpha_0 > 0$ and   $t \in  [0, T) $  with $T > 0$, we define
\begin{eqnarray}
P_1 (t)  &=&  \left\{   x \in \mathbb{R}^n  \left| \right.  |x| \leq K_0 \sqrt{(T-t) |\ln(T-t)|}    \right\} \label{defini-P-1-t},\\
P_2 (t) & =& \left\{   x \in \mathbb{R}^n   \left| \right.   \frac{K_0}{4} \sqrt{(T-t) |\ln(T-t)|}  \leq  |x| \leq \epsilon_0   \right\} \label{defini-P-2-t},\\
P_3(t) & =&  \left\{  x \in \mathbb{R}^n  \left|  \right.  |x| \geq \frac{\epsilon_0}{4}       \right\}.\label{defini-P-3-t}
\end{eqnarray}
 As a matter of fact,  we have
$$ \Omega \subset  \mathbb{R}^n =   P_1 (t) \cup P_2 (t) \cup P_3 (t), \text{ for all } t \in [0,T).$$
 We aim at controlling    our problem  on  $P_1 (t), P_2 (t)$ and $P_3 (t)$ as follows:

- On   region  $P_1(t)(\text{\textit{blowup region})}$:  We     control  $w$ (see  \eqref{similarity-variables}) instead of $U$. More precisely,    we  show that    $w$ is  a perturbation   of the profile  $\varphi$ (the blowup profile, introduced in    \eqref{defini-varphi}). Then,      \eqref{norm-Q-infinity-to-0}  will follow from the control of $w$.  

- On   region $P_2(t)(\text{\textit{intermidiate region})}$:   We control  a  rescaled  function   $\mathcal{U}$ instead of $U$. More precisely,        $\mathcal{U}$  is defined  as follows: For all $x \in P_2 (t), \xi \in   (T - t(x))^{-\frac{1}{2}} (\bar \Omega - x)   $ 	and  $\tau \in \left[- \frac{t(x)}{T-t(x)}, 1 \right),$ we define
\begin{equation}\label{rescaled-function-U}
\mathcal{U} (x, \xi, \tau) =  \left(T-t (x) \right)^{\frac{1}{3}}  U \left( x + \xi \sqrt{ T- t(x)},   (T-t(x)) \tau  +  t(x)     \right),
\end{equation}
where  $t(x)$  is defined as  the solution of the following equation
\begin{eqnarray}
|x| &  =& \frac{K_0}{4} \sqrt{ (T-t(x) ) | \ln(T-t(x))|} \text{ and  } t(x) < T.\label{defini-t(x)-}
\end{eqnarray}
We   remark that  if $\epsilon_0 $ is  small enough, then $t(x)$ is   well defined for all $x $ in $P_2(t)$. In addition to that,  using \eqref{defini-t(x)-}, we have the following  asymptotic 
$$  t(x)  \to  T, \text{ as }  x \to 0.$$
 For convenience, we     introduce 
\begin{equation}\label{defini-theta}
\varrho (x) = T - t(x).
\end{equation}
Then,  the following holds
$$\varrho(x) \to 0  \text{ as } x  \to 0.$$

\noindent
As a matter of fact, using  \eqref{equa-U},   we   write the  equation  satisfied by $\mathcal{U}$  in $(\xi, \tau) \in  \varrho^{-\frac{1}{2}}(x) (\bar \Omega - x) \times  \left[   -  \frac{t(x)}{\varrho (x)}, 1 \right)$  as follows:
 \begin{equation}\label{equa-xi}
 \partial_\tau  \mathcal{U}  =   \Delta_\xi \mathcal{U}   -  2 \frac{\left| \nabla \mathcal{U}  \right|^2  }{  \mathcal{U}  +  \frac{\lambda^{\frac{1}{3}}  \varrho^{\frac{1}{3}}(x)}{\tilde \theta ( \tau )}  } +  \left(\mathcal{U}  +  \frac{\lambda^{\frac{1}{3}} \varrho^{\frac{1}{3} }(x) }{\tilde \theta ( \tau )} \right)^4 -\frac{\tilde \theta_\tau' (\tau)}{\tilde \theta (\tau)}  \mathcal{U},
 \end{equation}
where   
\begin{equation}\label{defini-tilde-theta}
\tilde   \theta (\tau)   =         \bar \theta ( \tau\varrho(x)    + t(x)),  
\end{equation} 
with  $\bar{\theta} (t)$  defined in \eqref{defini-theta}.
 We now consider the following domain
 $$  |\xi| \leq  \alpha_0  \sqrt{|\ln( \varrho(x))|}  \text{ and }   \tau \in \left[ -\frac{  t(x)}{ \varrho(x)}, 1 \right).  $$
 When $\tau = 0$, we are in region $P_1(t(x))$, in fact (note that $P_1 (t(x))$ and $P_2(t(x))$ have some overlapping by definition). From our constraints in $P_1(t(x))$, we derive that $\mathcal{U}(x,\xi, 0)$ is flat in the sense that 
$$  \mathcal{U}(x, \xi, 0) \sim \left( 3 + \frac{9}{8} \frac{K_0^2}{16}\right)^{-\frac{1}{3}}. $$ 
Our idea is to show that this flatness will be conserved for all $\tau \in [0, 1)$ (that is for all $t \in [t(x),T)$), in the sense that the solution will not depend that much on space. In one word,   $\mathcal{U}$ is regarded   as a perturbation  of  $\hat{ \mathcal{U}}(\tau),$ where  $\hat{ \mathcal{U}}(\tau)$ is   defined as follows
 \begin{equation}\label{equa-hat-mathcal-U}
 \left\{  \begin{array}{rcl}
 \partial_\tau  \hat{\mathcal{U}} (\tau)   & = &  \hat{ \mathcal{U}}^4 (\tau),\\
 \hat{\mathcal{U}} (0)  & = &  \left(  3  + \displaystyle \frac{9}{8} \displaystyle \frac{K_0^2}{16} \right)^{-\frac{1}{3}}.
\end{array} \right.
 \end{equation}
Note that,  we can  give  an explicit  formula  to the  solution  of equation \eqref{equa-hat-mathcal-U} 
 \begin{equation}\label{defin-hat-mathcal-U-tau}
 \hat{\mathcal{U}} (\tau) =  \left(   3 (1 -\tau)  +  \displaystyle \frac{9}{8} \frac{K_0^2}{16}  \right)^{-\frac{1}{3}}.
 \end{equation}
 
 - On  region $P_3(t)(\text{\textit{regular region})}$:   Thanks to  the well-posedness property  of  the Cauchy problem  for equation  \eqref{equa-Q}, we control  the solution $U$ as a  perturbation of     initial data $U(0).$ Indeed, the blowup time $T$ will be chosen small in our analysis.
 
 \medskip
 Relying on those ideas,   we give in the following  the definition of  our  shrinking set:
\begin{definition}[Definition of  $S(t)$]\label{defini-shrinking-set-S-t} Let us consider    $ T> 0, K_0 > 0,  \epsilon_0 > 0, \alpha_0 > 0, A > 0, \delta_0 > 0, C_0> 0, \eta_0  > 0$ and   $t \in [0, T).$ Then,  we introduce  the following  set    
$$ S(T,  K_0 ,\epsilon_0, \alpha_0, A, \delta_0, \eta_0,  C_0, t) \quad   (S(t) \text{ for short}),$$
as a  subset  of         $C^{2,1}\left( \Omega  \times (0,t)  \right) \cap C(   \bar \Omega\times [0,t] ),$  containing  all functions  $U$  satisfying the following conditions:
\begin{itemize}
\item[$(i)$] \textbf{Estimates  in  $P_1(t)$}: We have   $q(s) \in V_{K_0,A} (s),$ where  $q(s)$ is  introduced   in \eqref{defini-q}, $s = -  \ln(T-t)$ and    $V_{K_0, A} (s)$ is   a  subset of  all function $r$  in $L^\infty (\mathbb{R}^n),$  satisfying the following estimates:
\begin{eqnarray*}
| r_i   |   &  \leq  & \frac{A}{s^{2}},  (i=0,1),  \text{ and  }  | r_2 | \leq  \frac{A^2 \ln s}{s^2},\\
|r_-(y)| & \leq  &  \frac{A^2}{s^2} (1 + |y|^3) ,  \text{ and }  \| r_e \|_{L^\infty(\mathbb{R}^n)} \leq    \frac{A^2}{ \sqrt s},\\
\left|   \left(   \nabla r  \right)_\perp    \right|  & \leq &  \frac{A}{s^2} (1 + |y|^3),  \forall y \in \mathbb{R}^n,
\end{eqnarray*}    
where the definitions  of   $r_i, r_-, \left( \nabla r \right)_\perp $   are given in    \eqref{defini-R-i}, \eqref{defini-R-perp} and \eqref{defini-R--}, respectively.  
\item[$(ii)$] \textbf{ Estimates in  $ P_2 (t)$:}     For all   $|x|    \in \left[  \frac{K_0}{4 }  \sqrt{(T-t)|\ln(T-t)|},  \epsilon_0 \right], \tau (x,t)  =   \frac{t- t(x)}{\varrho (x)}$    and $|\xi|  \leq  \alpha_0  \sqrt{|\ln \varrho(x)|},$    we have  the following
\begin{eqnarray*}
\left|   \mathcal{U}  (x, \xi, \tau(x,t))  - \hat{\mathcal{U}}(\tau(x,t))  \right|  & \leq &   \delta_0, \\
\left|  \nabla_\xi  \mathcal{U}  (x, \xi, \tau(x,t))  \right|  & \leq  & \frac{C_0}{\sqrt{|\ln \varrho(x)|}},\\
\end{eqnarray*}
where   $\mathcal{U},$  $\hat{\mathcal{U}}$ and  $\varrho(x) $  are  given \eqref{rescaled-function-U}, \eqref{defini-theta} and  \eqref{defin-hat-mathcal-U-tau},   respectively.
\item[$(iii)$] \textbf{Estimates in  $P_3(t)$:}  For all  $  x \in  \{ |x| \geq \frac{\epsilon_0}{4} \} \cap  \Omega,$ we have 
\begin{eqnarray*}
\left|    U(x,t)  -  U(x,0)  \right|  & \leq  &  \eta_0,\\
\left|  \nabla U(x,t) -  \nabla e^{t \Delta} U(x, 0) \right| & \leq  &  \eta_0. 
\end{eqnarray*}
\end{itemize}
\end{definition}

\noindent
In addition to that,   we  would like to introduce    the  set  $S^* (K_0, \epsilon_0,  \alpha_0,  A,  \delta_0,  C_0, \eta_0,T)$ as follows:
\begin{definition}\label{defini-S-*} 	For all $  T> 0, K_0 > 0, \epsilon_0 > 0, \alpha_0 > 0, A > 0, \delta_0 > 0, C_0 > 0 ,$ and  $\eta_0  > 0$,  we  introduce   $S^* (T, K_0,  \epsilon_0,  \alpha_0,  A,  \delta_0,  C_0, \eta_0)$ ($S^*(T) $  for short)  as   the  subset    of  all functions  $U$ in  	 $C^{2,1}( \Omega \times (0,T) ) \cap  C(\bar \Omega \times  [0, T)),$ satisfying the following:     for all $t\in  [0, T),$  we have  $$ U \in S( T, K_0,  \epsilon_0,  \alpha_0,  A,  \delta_0,   C_0, \eta_0,t).$$
\end{definition}

\begin{remark}\label{remark-change-merle-zaag}
The shrinking set  $S(t)$ is inspired by the work of Merle and Zaag  in  \cite{MZnon97}. However, we've made two  major changes:

- A simplification, by removing an unnecessary condition on the second derivative  in space  in region $P_2 (t)$.

- A  smart change in region $P_3 (t)$, by replacing $\nabla U$ by $ \nabla e^{t \Delta} U(0)$. This change  is crucial  since we are working on a bounded domain $\Omega$.

\end{remark}

\begin{remark}\label{remark-role-p-123}
 The  conditions   in    $P_2$ and $P_3$ in Definition \ref{defini-shrinking-set-S-t}   are  designed to	 make    our solution   more regular and these   conditions help us  to control $U $ in $P_1$ and  $q(s) \in V_{K_0, A} (s)$.  Finally, the main purpose is to satisfy \eqref{norm-Q-infinity-to-0}. In other words,  the control $U$ in  $P_1$ is   the  main issue.
\end{remark}
\begin{remark}
In  our  paper, we  use  a lot of   parameters to control our solution.   However, they will be  fixed at the end of the proof.  In addition to that,  we  would like  to give some  conventions  on the   universal constant in our  paper:   We    use  $C$ for universal  constants which depend only $n, \Omega, \gamma, \lambda$ and we  write   $C(K_0, \epsilon_0,...)$ for   constants  which depend  $K_0, \epsilon_0,...$, respectively.
\end{remark}

\medskip
As we mentioned in Remark \ref{remark-role-p-123}, we would like to  show here some estimates of the sizes of $q$ and $ \nabla q$, where  $q$ is the transformed function of $U$ when $U \in S(t)$.
\begin{lemma}[Sizes of  $q$ and $\nabla q$] \label{lemma-properties-V-A-s}  Let us consider  $K_0 \geq 1$ and  $\epsilon_0  > 0$. Then, there exist $T_1 (K_0, \epsilon_0)$ and  $\eta_1 (\epsilon_0)$ such that  for all $\alpha_0>0, A> 0, \delta_0 \leq  \frac{1}{2} \hat{\mathcal{U} }(0) $ (see  \eqref{defin-hat-mathcal-U-tau}), $C_0>0 ,\eta_0 \leq \eta_1, T \leq T_1$ and $t \in  [0,T)$:  if  $U \in S(K_0, \epsilon_0, \alpha_0,A, \delta_0, C_0,\eta_0,t),$ then,  the  following   holds:  
\begin{itemize}
\item[$(i)$] The  estimates on  $q$:  For all  $ y \in  \mathbb{R}^n$ and $s = - \ln (T-t),$ we have 
$$   |q (y,s)| \leq  \frac{C (K_0)A^2 }{\sqrt s }   \text{ and }    \left| q (y,s)\right| \leq   \frac{C (K_0) A^2 \ln s}{s^2} (1 +   |y|^3). $$
\item[$(ii)$] The estimates  on  $\nabla q$: For all  $y \in   \mathbb{R}^n$, we have 
$$  |  \nabla q (y,s)|  \leq  \frac{C (K_0,C_0) A^2}{ \sqrt s},  \quad  	 \left|  \nabla  q  (y,s)\right|  \leq  \frac{C(K_0,C_0) A^2 \ln s}{s^2} (1 + |y|^3), $$ 
and
$$ \left|(1 - \chi(y,s) ) \nabla q (y,s) \right|  \leq  \frac{C(K_0)}{ \sqrt s}.$$
\end{itemize}
\end{lemma}
\begin{proof}  The conclusion directly  follows from the definition of  the shrinking set $S(t)$ and $V_{K_0,A}(s)$. In addition to that,  these definitions are  almost the same as  in \cite{MZnon97}. Therefore,  we kindly refer  the reader to see Lemma B.1 at page 1537 in \cite{MZnon97}.
\end{proof}
\medskip

\subsection{Initial data}
In this subsection,    we   will concentrate  on     introducing  our   initial data  to equation \eqref{equa-U}  so that it is   trapped  in $S(0)$. In order to  do that,  we first introduce  the following cut-off function:
\begin{eqnarray}
\chi_1 (x)   & = &  \chi_0 \left(   \frac{|x|}{ \sqrt T |\ln T|}   \right),  \label{defini-chi-1} 
\end{eqnarray}
where  $\chi_0$ is given in \eqref{defini-chi-0}. In addition to that,  we  introduce  $H^*$ as  a function in $C^\infty_0( \mathbb{R}^n  \setminus \{0\})$ satisfying 

\begin{equation}\label{defini-H-epsilon-0}
H^*  (x)  =  \left\{   \begin{array}{rcl}
&  &  \left[  \displaystyle \frac{9}{16} \frac{|x|^2}{|\ln|x||}    \right]^{ -\frac{1}{3}}, \quad  \forall   |x| \leq  \min\left(  \frac{1}{4} d (0 ,\partial \Omega), \frac{1}{2}  \right), x \neq 0, \\[0.7cm]
&  & 0,   \quad   \forall   |x| \geq  \frac{1}{2}  d(0, \partial \Omega)  ,\\[0.5cm]
\end{array}      \right.
\end{equation}
and for all $x \in \mathbb{R}^n, x \neq 0,$ the following condition holds
$$  0 \leq H^* (x) \leq  \left[  \frac{9}{16} \frac{|x|^2}{|\ln |x||}\right]^{-\frac{1}{3}}.$$
\noindent
\medskip
We now give the definition of  our initial data corresponding to equation \eqref{equa-U}: For all  $(d_0, d_1 ) \in   \mathbb{R}^{1 + n}$, we define
\begin{eqnarray}
U_{d_0, d_1}(x,0)  & = &   T^{-\frac{1}{3}}  \left[  \varphi \left( \frac{x}{\sqrt{T}}, - \ln s_0  \right)   +  (d_0 + d_1 \cdot z)  \chi_0 \left(  \frac{|z|}{\frac{K_0}{32}}\right) \right]  \chi_1 (x) \label{defini-initial-data}\\
 & + & H^*(x) \left(  1 -  \chi_1 (x)    \right) ,\nonumber
\end{eqnarray}
where   $ z = \displaystyle  \frac{ x}{ \sqrt{T  |\ln T|}} $ and $\varphi, \chi_0,  \chi_1, H^*  $ are defined as in  \eqref{defini-varphi}, \eqref{defini-chi-0}, \eqref{defini-chi-1} and  \eqref{defini-H-epsilon-0}, respectively.

\noindent
From \eqref{defini-initial-data},   we would like to give  the definition of initial data corresponding to equation \eqref{equa-Q},   $q_{d_0, d_1}(s_0) $ with   $s_0 = - \ln T$:
 \begin{equation}\label{defini-q-s-0-d-0-d-1}
 q_{d_0, d_1} (y, s_0) =   e^{-\frac{s_0}{3}} U_{d_0,d_1} \left(  y  e^{-\frac{s_0}{3}} , 0  \right) \psi_{M_0} (y, s_0) -  \varphi (y,s_0),
 \end{equation}
 where   and $ \psi_{M_0}, \varphi$ and $U_{d_0,d_1}$ are introduced in  \eqref{defini-psi-M-0-cut}, \eqref{defini-varphi} and  \eqref{defini-initial-data}, respectively.
 \begin{remark} We would like to explain in brief  how  our initial data  $U_{d_0,d_1}$ has naturally  the form shown in  \eqref{defini-initial-data}.  As we mentiond at the beginning of this section, our purpose  is  to control initial data  in  $S(0)$. More precisely,   our inital data have to satisfy items $(i)$ and $(ii)$ in Definition \ref{defini-shrinking-set-S-t}. As a matter of fact, when $T $ is  small enough,  the second term in  the right hand side of \eqref{defini-initial-data} is zero on $P_1 (0)$. Then, our initial data has only  the first term and we adopt the idea  given   in \cite{MZdm97} (see also \cite{MZnon97}, \cite{GNZpre16a}), we  use     $d_0,d_1$ in order to  control $q(s_0)$ in  $V_{K_0,A}(s_0)$. In addition to that, we would like to mention  that   Propostion \ref{proposition-reduction-finite} below states that when $q$ is trapped in $V_{K_0, A} (s)$,  it  has only two components $(q_0,q_1)(s)$ which may attain their  upper bound,  the others  being strictly less than  their upper bound specified in the definition of $V_{K_0, A} (s)$. This  is   indeed the reason to use $(d_0,d_1)$ in our initial data.   More precisely,     these $1 + n$ parameters   allows us to   a reduction to a  finite dimensional problem.  We now mention   the control   in  $P_2$.  In    that  region, $|x|$ is small enough and  we may  consider that  $U$ is near  the final profile
$$  \left[  \displaystyle \frac{9}{16} \frac{|x|^2}{|\ln|x||}    \right]^{ -\frac{1}{3}}. $$ 
 As a matter of fact, it is reasonable to    introduce $H^*$ as the  main asymptotic of our initial data in $P_2(0)$. Using some priori estimates, we can derive    good estimates  in  $P_2(0)$. More precisely, the following proposition is our statement:  
 \end{remark}
\begin{proposition}[Preparation of initial data]\label{proposiiton-initial-data} There  exists $K_{2}> 0$ such that  for all  $K_0 \geq K_{2} $ and $  \delta_{2 } > 0,$ there exist $\alpha_{2} (K_0, \delta_{2}) > 0 $  and $C_{2}(K_0) > 0$ such that for  every  $\alpha_0 \in   (0, \alpha_{2}],$     there exists  $\epsilon_{2} (K_0, \delta_{2}, \alpha_0) > 0$ such that   for every  $\epsilon_0 \in (0,\epsilon_{2}]$  and  $A \geq 1,$ there exists  $T_{2} (K_0,  \delta_{2}, \epsilon_0, A,C_2 ) > 0 $  such that  for all  $T \leq    T_2$ and $ s_0 = - \ln T$. The following  holds:

(I)  We can find a  set  $\mathcal{D}_{ A}     \subset [-2,2] \times [-2,2]^n$  such that   if we define the following mapping
\begin{eqnarray*}
\Gamma : \mathbb{R} \times \mathbb{R}^n     & \to &   \mathbb{R} \times \mathbb{R}^n\\
(d_0, d_1)   & \mapsto   &  \left( q_0, q_1  \right)(s_0),
\end{eqnarray*}
 then, $\Gamma$  is  affine,  one to  one   from   $\mathcal{D}_{A}$   to 	$\hat{\mathcal{V}}_A (s_0),$ where $\hat{\mathcal{V}}_A (s)$ is defined as  follows
\begin{equation}\label{defini-hat-mathcal-V-A}
\hat{\mathcal{V}}_A (s) =  \left[-\frac{A}{s^2}, \frac{A}{s^2} \right]^{1+n}.
\end{equation}  
Moreover,  we have     
$$\Gamma \left|_{\partial \mathcal{D}_{A}}  \right.   \subset   \partial \hat{\mathcal{V}}_A (s_0), $$
and 
\begin{equation}\label{deg-Gamma-1-neq-0}
 \text{deg} \left(  \Gamma \left. \right|_{ \partial  \mathcal{D}_{A} }  \right) \neq 0,
\end{equation}
  where   $q_0,q_1$  are defined  as  in \eqref{represent-non-perp}, considered as the   components of $q_{d_1,d_1} (s_0),$ which       is a transform function of  $U_{d_0, d_1} (0),$ given  in \eqref{defini-q}.

(II) We now consider    $ (d_0, d_1) \in \mathcal{D}_{ A}$.  Then,  initial data  $U_{d_0,d_1}(0)$   belongs to  
$$S(K_0, \epsilon_0, \alpha_0, A, \delta_{2},   C_{2}, 0, 0)= S(0),$$
 where  $S( 0) $ is defined  in Definition \ref{defini-shrinking-set-S-t}.  Moreover, the following   astimates  hold
\begin{itemize}
\item[$(i)$] Estimates in  $P_1(0):$   We have   $q_{d_0,d_1} (s_0) \in \mathcal{V}_{K_0, A} (s_0)$ and  
$$ |q_0 (s_0)| \leq  \frac{A}{s_0^2}, \quad   |q_{1,j}  (s_0)| \leq \frac{A}{s_0^2}, \quad  |q_{2,i,j}(s_0)| \leq \frac{\ln s_0}{s_0^2},   \forall i,j \in \{1,..., n\},$$
$$ |q_- (y,s_0) | \leq 	\frac{1}{s_0^2} (|y|^3 + 1), \quad  |\nabla q_\perp (y,s_0) | \leq \frac{1}{s_0^2} (|y|^3 + 1),  \forall y \in \mathbb{R}^n,$$
and
$$  q_e   \equiv 0,$$
 where  the  components  of $q_{d_0, d_1} (s_0)$  are defined in  \eqref{defini-R-perp}.
\item[$(ii)$] Estimates  in  $P_2(0)$:  For  every $|x| \in \left[  \frac{K_0}{4} \sqrt{T |\ln T|}  , \epsilon_0 \right] , \tau_0(x) = - \frac{ t(x) }{ \varrho (x)}$  and $|\xi|  \leq \alpha_0 \sqrt{|\ln \varrho(x)|},$ we have
$$   \left|  \mathcal{U} (x, \xi, \tau_0(x) )    - \hat{\mathcal{U}} (\tau_0(x))  \right| \leq   \delta_{2} \text{ and }    |\nabla_\xi \mathcal{U} (x, \xi, \tau_0(x))| \leq \frac{C_{2}}{\sqrt{|\ln \varrho(x)|}} ,
 $$
\end{itemize}
where   $\mathcal{U}, \hat{\mathcal{U}}, $ and  $\varrho(x)$ are  defined  as in  \eqref{rescaled-function-U}, \eqref{defin-hat-mathcal-U-tau} and  \eqref{defini-theta} respectively.
 \end{proposition}
 \begin{proof}
 The  proof of   Proposition  \ref{proposiiton-initial-data}  will  be given    in Appendix  \ref{preperation-initial-data}. We now  assume that  this  proposition holds and  continue     to  get to the conclusion  of  Theorem \ref{theorem-existence}.
 \end{proof}
\subsection{ Existence of a  solution  trapped  in  $S^* (T)$}
  In this subsection,   we would  like to derive  the  existence  of a  blowup  solution $U$ to equation \eqref{equa-U},    trapped in  $S^* (T)$. As we said earlier, our proof will be a (non trivial)  adaptation  of the proof designed by Merle and Zaag in \cite{MZnon97} for the more standard case  \eqref{equa-bar-u-merle-zaag}.  However, in comparision with \eqref{equa-bar-u-merle-zaag},   we observe in equation  \eqref{equa-U}  a new feature,  the nonlocal term involving $\bar{\theta} (t)$. As a matter of fact, it  is  important to  study   this term and it derivative.  In particular,  in the works which we used to make the main idea for our work (such as  \cite{MZnon97}, \cite{MZdm97}, \cite{GNZpre16a}), the authors only  studied for constant coefficients parabolic equations.  Hence, it makes      a main  highlight in our work.   For that reason, we  show here some estimates   on  $\bar{\theta}(t)$ (also on  $\bar{\mu}(t)$). The following is our statement:
  
  \begin{proposition}[Some estimates of $\bar{\theta} (t)$ and  $\bar{\mu} (t)$]\label{propo-bar-mu-bounded} Let us consider  $\lambda > 0, \gamma > 0$ and  $\Omega$ a $C^2$ bounded domain. 
Then, there  exists $K_{3}> 0$ such that  for all  $K_0 \geq K_{3},  \delta_{0 } > 0,$ there exist  $\alpha_{3} (K_0, \delta_{0}) > 0$  such that for all  $\alpha_0 \leq   \alpha_{3},$ there exists       $\epsilon_{3} (K_0, \delta_{0}, \alpha_0) > 0$ such that   for all   $\epsilon_0  \leq \epsilon_3$  and  $A \geq 1, C_0 > 0, \eta_0 > 0$, there exists $T_{3}   > 0$ such that  for  all  $T  \leq T_3$ the following holds:  Assuming   $U$ is a non negative  solution  of   equation \eqref{equa-U}    on $[0, t_1],$ for some $t_1 < T$, $U \in S(T,K_0, \epsilon_0, \alpha_0, A, \delta_0, C_0,\eta_0,t)= S(t) $ for all  $t \in [0, t_1]$ and  $U(0) = U_{d_0, d_1} (0),$ given in \eqref{defini-initial-data} with some $(d_0,d_1) \in \mathbb{R}^{1 + n}, |d_0|, |d_1| \leq 2$,  the following statements  holds:
\begin{itemize}
\item[$(i)$]  For all $t\in [0,t_1],$  $\bar{\mu} (t)$ and  $\bar{\theta} (t)$ are  positive and   these estimates hold
 \begin{eqnarray}
0 \leq   \bar \mu (t)    & \leq &   C,\label{bound-bar-mu-t}\\
1 \leq \bar \theta  (t)  &  \leq &  C. \label{bound-bar-theta}
\end{eqnarray}

\noindent
Moreover, for all $t \in (0,t_1)$, we have
\begin{eqnarray}
|\bar \mu' (t)|  & \leq &   C  (T-t)^{\frac{3n - 8}{6}} |\ln  (T-t)|^n \label{estima-derivative-of-bar-mu},\\
\left| \bar \theta' (t) \right|   & \leq  &   C  (T-t)^{\frac{3n - 8}{6}} |\ln  (T-t)|^n. \label{bound-bar-theta-'}
\end{eqnarray}  
 \item[$(ii)$] In particular, if  $U \in S(t)$  for all   $ t \in [0, T)$, then    $\bar \mu (t)$ and  $\bar{\theta}(t)$ converge respectively   to  $\bar \mu_T $  and $\bar \theta_{T} \in \mathbb{R}_+^*$ as $t \to T $.
\end{itemize} 
\end{proposition}

\begin{remark} 
 Although we know from item (ii)  that $\bar \theta (t)$ converges to $\bar{ \theta}_T,$  we don't know how  big is  $\bar{\theta}_T$.   In  particular,  the dependence of  these  constants on $\gamma, \lambda,\Omega$ and  $T, \epsilon_0,...,$  is not clear yet. 
  \end{remark}
\begin{proof}
We can see that   item $(ii)$ is a direct consequence  of  $(i)$. So,    we   give only the proof  of   item $(i)$. Using \eqref{defini-bar-mu}  and the fact that  $U(t) \geq 0$ for all $t$, we derive the following 
$$\bar  \mu  (t)   = \int_\Omega U (t) dx \geq 0.$$
   In addition to that,  we write 
\begin{equation}\label{estima-bar-mu-P-1-P-2-P-3}
  \bar \mu (t) \leq    \int_{\Omega} U(t) dx \leq   \int_{P_1 (t)} U(t) dx +   \int_{P_2 (t)} U(t) dx +   \int_{P_3 (t)} U(t) dx,
\end{equation}
where  $P_1(t), P_2(t), P_3(t)$ are given in  \eqref{defini-P-1-t}, \eqref{defini-P-2-t} and \eqref{defini-P-3-t}, respectively. Remembering  $\varrho(x),$ defined in  \eqref{defini-theta}, we see that  the following holds
$$ \varrho(x) \sim    \frac{8}{K_0^2} \frac{|x|^2}{|\ln|x||}  \text{ as  }  x \to  0. $$
In particular, using  Definition  \ref{defini-shrinking-set-S-t}, we get the following estimates: for all $t \in [0,t_1]$
\begin{eqnarray*}
\text{ On } P_1 (t),       |U(x,t)|  & \leq  & (T-t)^{-\frac{1}{3}} \left[ \frac{CA^2}{\sqrt{|\ln (T-t)|}}   + |\varphi(0, -\ln (T-t))|\right] \leq  C   (T-t)^{-\frac{1}{3}} ,\label{estima-P-1-t}\\
\text{ On }   P_2 (t),    |U(x,t)|  & \leq  & \varrho^{-\frac{1}{3}}(x) \left[ \hat{\mathcal{U}}(\tau(x,t))+ \delta_0 \right] \leq C  \left[ \frac{|x|^2}{|\ln|x||} \right]^{-\frac{1}{3}}, \label{estima-P-2-t}\\
\text{ On  } P_3 (t),  |U(x,t)|  & \leq &     |U (x, 0) |  + \eta_0 \leq  |U(x,0)| +  1,
\end{eqnarray*}
 provided  that   $K_0 \geq  K_{3,1}$   $ \delta_0  \leq 1$ and $ \eta_0 \leq 1$.  Integrating  $U$ on  each  $P_i(t), i = 1,2,3$, we obtain  the following  
\begin{eqnarray*}
\int_{P_1 (t)}  U(t) dx   &  \leq  & C(K_0) (T-t)^{\frac{n}{2} - \frac{1}{3}}  | \ln(T-t)|^{\frac{n}{2}}, \\
\int_{P_2 (t)} U(t) dx  & \leq &  C \int_{|x| \leq \epsilon_0} \left[  \frac{|x|^2}{|\ln|x||}   \right]^{-\frac{1}{3}} dx  \leq   C  \epsilon_0^{n - \frac{2}{3}} |\ln \epsilon_0|^\frac{1}{3},\\
\int_{P_3 (t)} U (t)  & \leq &    \left[   \int_{\frac{\epsilon_0}{4} \leq  |x|, x \in \Omega }\left[ H^*  + 1 \right] dx   \right], 
 \end{eqnarray*} 
 where   $H^*$  is   defined   in  \eqref{defini-H-epsilon-0}.   Using  \eqref{estima-bar-mu-P-1-P-2-P-3} and the above estimates,  it is easy to obtain the following estimate
 $$  \mu (t) \leq   C, \text{ for all } t \in [0,t_1],$$
provided that    $K_0  \geq  K_{3,2} ( \lambda, \gamma),  \epsilon_0 \leq \epsilon_{3,1} (\lambda, \gamma), \eta_0 \leq   \eta_{3,1} (\lambda, \gamma)$ and  $T \leq T_{3,1} (K_0, \lambda, \gamma)$.
This  yields  \eqref{bound-bar-mu-t}  and     \eqref{bound-bar-theta}  also follows   by using    \eqref{relation-bar theta-bar-mu}  and \eqref{bound-bar-mu-t}. We now give a  proof  to  \eqref{estima-derivative-of-bar-mu}.  Integrating two sides of   equation  \eqref{equa-U}, we get the following ODE
\begin{equation}\label{relation-bar-mu-bar-theta}
\bar \mu'(t)  + \frac{\bar \theta' (t)}{ \bar \theta (t) } \bar \mu (t)  =  \int_{ \Omega}  \Delta U(t) dx  + \int_{\Omega
}  \left( \left(  U (t)+ \frac{\lambda^{\frac{1}{3}}}{ \bar \theta (t)}\right)^4  - 2 \frac{ |  \nabla  U(t)|^2}{ U(t)  +  \frac{\lambda^{\frac{1}{3}}}{\bar \theta(t)}}   \right) dx. 
\end{equation}
We  aim at proving the following estimate
\begin{equation}\label{estimate-rest-term-U-4-2-nabla-U}
\left|  \int_{\Omega
}  \left( \left(  U(t) + \frac{\lambda^{\frac{1}{3}}}{ \bar \theta (t)}\right)^4  - 2 \frac{ |  \nabla  U(t)|^2}{ U (t) +  \frac{\lambda^{\frac{1}{3}}}{\bar \theta(t)}}   \right) dx \right|  \leq C (T-t)^{\frac{3n  - 8 }{6}} |\ln(T-t)|^{n} . 
\end{equation}
In order to do so, we first prove  that
\begin{eqnarray}
\int_\Omega  U^4 (t) dx  & \leq &  C (T-t)^{  \frac{ 3n - 8}{6}} | \ln (T-t)|^{n} ,\label{estima-int-U-4}\\
\int_\Omega  \frac{\left| \nabla U(t) \right|^2}{ U (t)  + \frac{\lambda^{\frac{1}{3}}}{ \bar \theta (t)}}dx & \leq &   C  (T -t)^{ \frac{ 3n - 8}{6}} |\ln(T-t)|^{n}.\label{estima-int-nabla-U}
\end{eqnarray}
The techniques  of  proofs  \eqref{estima-int-U-4} and \eqref{estima-int-nabla-U}  are  the  same. Therefore,  we  only give here the proof of  \eqref{estima-int-nabla-U}.  Let us consider
$$ I (x, t)=  \frac{\left| \nabla U(x,t) \right|^2}{ U (x, t)  + \frac{\lambda^{\frac{1}{3}}}{ \bar \theta (t)}}.  $$
Then,
$$ \int_\Omega I(x,t) dx   \leq  \int_{P_1 (t)}   I(x,t) dx  + \int_{P_2 (t)}   I(x, t) dx + \int_{P_3 (t)}   I(x, t)dx.      $$
Now we  claim   the following lemma:
\begin{lemma}\label{lemma-bound-I-x-t}
 Under  the hypothesis in  Proposition \ref{propo-bar-mu-bounded},     for  all $t \in (0,t_1]$, the following estimates  hold: 
\begin{eqnarray}
\text{ On } P_1 (t):  I(x,t)  & \leq & C(K_0) (T-t)^{ -\frac{4}{3}}, \label{estima-bound-T-P-1}\\
\text{ On } P_2 (t): I(x,t) & \leq & C(K_0) \varrho^{-\frac{4}{3}}(x) \leq  C(K_0)\left[ \frac{|x|^2}{|\ln |x||} \right]^{- \frac{4}{3}}, \label{estima-bound-T-P-2}\\
\text{ On }  P_3 (t): I(x,t) & \leq & C \left( \left|  \nabla U (x,0)  \right|^2 + \eta_0^2 \right) = C (|\nabla H^*(x)| +  \eta_0^2 ).\label{estima-bound-T-P-3}  
\end{eqnarray}
\end{lemma}
\begin{proof}
From the definition of $S(t),$ we easily  derive  \eqref{estima-bound-T-P-3}. So, we only give here the proofs  of  \eqref{estima-bound-T-P-1} and \eqref{estima-bound-T-P-2}.  We now  start with \eqref{estima-bound-T-P-1}.      Let us  consider  $x \in P_1(t) $ and  we   use the condition of  $U$ in $P_1 (t)$ to  get the following 
\begin{equation}\label{bound-1-C-CP-1}
\frac{1}{C(K_0)} (T-t)^{-\frac{1}{3}} \leq  U(x, t) \leq  C(K_0) (T-t)^{-\frac{1}{3}}. 
\end{equation}
In addition to that, thanks to item $(ii)$ in Lemma \ref{lemma-properties-V-A-s}, we get 
$$ | \nabla_y W \left( \frac{x}{\sqrt{T-t}},- \ln (T-t) \right)  | \leq   \frac{C(K_0) A^2}{ \sqrt{|\ln (T-t)|}},$$
which yields
\begin{equation}\label{boun-P-1-nabla}
 \left| \nabla U(x, t)  \right| \leq C (K_0)(T - t)^{-\frac{5}{6}}.
\end{equation} 
Then,  \eqref{estima-bound-T-P-1} follows by \eqref{bound-1-C-CP-1} and  \eqref{boun-P-1-nabla}. 

\noindent
We now consider  $x \in P_2(t)$. It is easy   to derive from item $(ii)$ in Definition \ref{defini-shrinking-set-S-t}  that
\begin{eqnarray*}
\frac{1}{C(K_0)} \varrho^{\frac{1}{3}} (x) & \leq  & U(x,t)  \leq  C(K_0) \varrho^{\frac{1}{3}} (x),\\
                 \left| \nabla U(x,t) \right| & \leq &   C \varrho^{	-\frac{5}{6}} (x),
\end{eqnarray*}
  provided that  $\delta_0 \leq \delta_{3,1}$ and $ \epsilon_0 \leq \epsilon_{3,2}$.  This gives \eqref{estima-bound-T-P-2} and concludes the proof of Lemma  \ref{lemma-bound-I-x-t}.
\end{proof}
\medskip
\noindent
We  now  continue the proof of Proposition  \ref{propo-bar-mu-bounded}. Considering  $t \in (0,t_1)$ and taking the  integral on  two sides of   \eqref{estima-bound-T-P-1}, we write
\begin{eqnarray*}
\int_{P_1 (t)} | I (x,t)| dx  & \leq  &   C(K_0) \int_{ |x| \leq K_0 \sqrt{(T-t)|\ln(T-t)|}}  (T-t)^{-\frac{4}{3}} dx\\[0.5cm] 
&  \leq  &  C(K_0)   (T-t)^{\frac{n}{2} - \frac{4}{3}} \left|  \ln(T-t)\right|^{\frac{n}{2}}.  
\end{eqnarray*} 
Integrating the   two sides of \eqref{estima-bound-T-P-2} and  using the following  fact 
$$ \varrho(x) \sim    \frac{8}{K_0^2}  \frac{|x|^2}{|\ln |x||}  \text{ as }  x \to 0,$$
we obtain the following 
\begin{eqnarray*}
\int_{P_2 (t)} \left| I (x,t)\right| \leq    C(K_0)  \left[  \epsilon_0^{n - \frac{8}{3}} |\ln \epsilon_0 |^\frac{4}{3}  -     ((T-t)|\ln (T-t)|)^{\frac{3n -8}{6}} |\ln ((T-t)|\ln (T-t)|)|^{\frac{4}{3}}  \right].
\end{eqnarray*}
In addition to that,  from   \eqref{estima-bound-T-P-3}, we have
$$  \int_{P_3 (t)}  |I(x,t)| dx  \leq C.$$
 Hence,   \eqref{estima-int-nabla-U}   holds.
 
 \noindent
In addition to that, using  \eqref{estima-rough-nabla-2-U},  we  can derive that 
$$ \int_{\Omega} \Delta U (t)dx < \infty, \forall t \in (0,t_1).$$
Therefore, we have
$$  \lim_{\upsilon \to 0} \int_{\{x , d(x, \partial \Omega) >  \upsilon \}} \Delta U dx  = \int_\Omega \Delta U(t) dx. $$
Moreover, for all $\upsilon > 0 $ small enough and from  item  $(iii)$ of  Definition \eqref{defini-shrinking-set-S-t}, we have
 \begin{equation}\label{bounded-on-boundary-omega}
 \left| \int_{\{x , d(x, \partial \Omega) >  \upsilon \}} \Delta U dx \right|  =  \left| \int_{\partial \{x , d(x, \partial \Omega) >  \upsilon \}} \nu(x) \cdot \nabla U (x,t) dS\right| \leq  C.
 \end{equation}
 This  implies that
 \begin{equation}\label{int-Delta-U-leq-C}
 \int_{\Omega} \Delta U (t)dx  \leq C.
 \end{equation}
 Hence, from   \eqref{relation-bar-mu-bar-theta},   \eqref{estimate-rest-term-U-4-2-nabla-U} and \eqref{int-Delta-U-leq-C},  we   derive the following 
 \begin{equation}\label{estima-bar-mu--bar-theta}
\left|  \bar \mu'(t)  + \frac{\bar \theta' (t)}{ \bar \theta (t) } \bar \mu (t)  \right| \leq C (T-t)^{ \frac{3n - 8}{6}} | \ln (T-t)|^n .
 \end{equation}
 In addition to that, from the  relation between  $\bar \mu$ and   $\bar  \theta$ in  \eqref{relation-bar theta-bar-mu}, we write 
 \begin{eqnarray*}
   \frac{ \bar \theta' (t)}{ \bar \theta (t)}  =  \frac{2 \gamma }{3 \lambda^{\frac{1}{3}}} \left( 1  - \frac{\frac{2 \gamma}{3 \lambda^{\frac{1}{3}}} \bar \mu(t) }{ \left( 1 + \gamma |\Omega|  + \frac{\gamma}{\lambda^{\frac{1}{3}}} \bar \theta(t) \bar \mu(t) \right)^{\frac{1}{3}}  } \right)^{-1} \left( 1 + \gamma |\Omega|  + \frac{\gamma}{\lambda^{\frac{1}{3}}} \bar \theta (t)\bar \mu (t) \right)^{-\frac{1}{3}} \mu' (t).
 \end{eqnarray*}
 
 \noindent
We also  have the fact  that
$$   \sqrt{\bar \theta (t)} \geq  \frac{\gamma}{ \lambda^{\frac{1}{3}}} \bar \mu (t),$$
 which  yields  that 
 $$1 \leq \left( 1  - \frac{\frac{2 \gamma}{3 \lambda^{\frac{1}{3}}} \bar \mu(t) }{ \left( 1 + \gamma |\Omega|  + \frac{\gamma}{\lambda^{\frac{1}{3}}} \bar \theta(t) \bar \mu(t) \right)^{\frac{1}{3}}  } \right)^{-1}  \leq 3. $$
 Hence,  $\bar \theta' (t)$ and $\bar \mu' (t)$ have  the same sign and  we can use    \eqref{estima-bar-mu--bar-theta}  to  conclude that
\begin{equation}\label{estima-theta-s-theta-s}
 \left| \bar \mu' (t)\right| \leq     C (T-t)^{ \frac{3n - 8}{6}} | \ln (T-t)|^n .
\end{equation}
This  yields  \eqref{estima-derivative-of-bar-mu} and  \eqref{bound-bar-theta-'}. Thus, we get the conclusion  of the proof of Proposition \ref{propo-bar-mu-bounded}.
\end{proof}

\begin{proposition}[Existence  of a solution to  equation \eqref{equa-U}, confined in  $S^*$]\label{proposition-existence-U}  We can find   parameters  $T> 0, K_0 > 0, \epsilon_0 > 0, \alpha_0 > 0, A> 0, \delta_0 > 0, C_0 > 0,  \eta_0 > 0$ such that  there exist $(d_0, d_1) \in \mathbb{R} \times  \mathbb{R}^n$
 such that with  initial data $U_{d_0, d_1}(0)$(given  in \eqref{defini-initial-data}), the solution  $U$ of equation \eqref{equa-U} exists  on  $ \Omega \times  [0,T)$ and   
   $$ U \in S^*(T, K_0, \epsilon_0, \alpha_0, A, \delta_0, C_0, \eta_0, T), $$  where  $S^*(T, K_0, \epsilon_0, \alpha_0, A, \delta_0, C_0, \eta_0, T) $ given in  \eqref{defini-S-*}.

\end{proposition}
\begin{proof}
 As a matter of fact, this Proposition plays a central role in our problem.  In other words, it will imply Theorem \ref{theorem-existence} (see subsection \ref{subsection-theorem-11}  below).  The proof of this Proposition     will be presented in      two steps:     

-  \textit{First step:}  We use a  reduction   of  our problem  to a   finite dimensional one. More  precisely,   we  prove that  the controling   $U $ in $ S(t)$ for all $t \in [0,T)$ is reduced to the control of $(q_0,q_1) (s)$ in  $\hat{\mathcal{V}}_A (s)$ (see Proposition  \ref{proposition-reduction-finite} below).

- \textit{Second  step:}   In this step, we aim at  proving that there exist $(d_0, d_1) \in \mathbb{R}^{1 + n}$ such that  $U \in S^*(T,K_0, \epsilon_0, \alpha_0,A,\delta_0, C_0, \eta_0,T)$ with suitable parameters.     Then, the conclusion follows from a topological argument based on Index theory. 

We  now give     two main  steps with  more technical details:

  $a)$  \textit{ Reduction  to a finite dimensional problem:}       In this step,    we derive that the  control  of  $U \in S(t)$ with $ t\in [0, T)$ is reduced  to  the control of   the transform function $q(s)$ such that  two  first components     $(q_0, q_1) (s)$ are trapped  in $\hat{\mathcal{V}}_A(s)$ (see  \eqref{defini-hat-mathcal-V-A}), where  $s =  - \ln (T- t)$. More  precisely,  the following  proposition is our  statement:
\begin{proposition}[Reduction to a  finite  dimensional problem]\label{proposition-reduction-finite} There exist $T > 0,   K_0 > 0, \epsilon_0 > 0, \alpha_0 > 0, A > 0, \delta_0 > 0,  C_0 > 0, \eta_0> 0$ such that  the following holds: We consider $U$ a solution  of equation \eqref{equa-U} that  exists on $[0, t_1],$ for some $t_1 < T,$ with  initial data $U_{d_0, d_1}(0)$ given in \eqref{defini-initial-data}, for some  $(d_0, d_1) \in \mathcal{D}_{ A}$.  We also assume that we have  $ U \in   S (t)$ for all $ \forall t \in [0, t_1] \text{ and }  U \in  \partial S (t_1)$ (see the  definiton of $S(t) = S(T, K_0, \epsilon_0, \alpha_0, A, \delta_0,  C_0, \eta_0,t)$ in Definition \ref{defini-shrinking-set-S-t} and  the  set $\mathcal{D}_A$ given  in  Proposition \ref{proposiiton-initial-data}). Then, the following  statments  hold: 
	\begin{itemize}
	\item[$(i)$] We have  $(q_0, q_1) (s_1)  \in \partial \hat{\mathcal{V}}_A (s_1),$ where $(q_0, q_1) (s)$ are components  of  $q(s)$ given in \eqref{represent-non-perp}  and  $q(s)$ is the  transform function  of  $U$ defined  in \eqref{defini-q}  and  $s_1 = \ln(T-t_1)$.
	\item[$(ii)$]   There exists  $\nu_0 > 0$ such that  for all $\nu \in (0, \nu_0),$ we have
	$$  (q_0, q_1) ( s_1  +  \nu) \notin    \hat{\mathcal{V}}_A (s_1 +  \nu).$$ 
	Consequently,  there  exists  $\nu_1 > 0$ such that 
	$$ U  \notin    S(t_1  +  \nu), \forall \nu \in (0, \nu_1).$$

	\end{itemize}
	\end{proposition}    
	The idea of   the  proof  is  inspired (in a non trivial way) by  the  ideas  given  by Merle and   Zaag in \cite{MZnon97}. Since the proof is long and technical, we leave it to Section \ref{section-reduction}. Therefore, we assume here that  Proposition \ref{proposition-reduction-finite} holds and go forward  to  the conclusion of  Proposition  \ref{proposition-existence-U}.

$b)$ \textit{ Topological argument and  the conclusion of Proposition   \ref{proposition-existence-U}:}  
In this  step, by using   Proposition \ref{proposition-reduction-finite} and  a topological argument  based on  Index theory, we conclude     Proposition \ref{proposition-existence-U}.  More precisely, we     prove that  there exist   $ T, K_0, \epsilon_0, \alpha_0, A, \delta_0,  C_0, \eta_0$ and  $(d_0, d_1) \in \mathcal{D}_{ A}$  such that with initial data  $U_{d_0, d_1} (0)$ (defined  in  \eqref{defini-initial-data}), the solution  of equation \eqref{equa-U}  exists    on $[0, T)$   and   belongs to  $S^*(T)$ where  $S^*(T)$ is  defined in  Definiton \ref{defini-S-*}.    Indeed, let us consider parameters  $T > 0,   K_0 > 0, \epsilon_0 > 0 , \alpha_0 > 0 , A > 0 , \delta_0 > 0 ,  C_0, \eta_0 > 0 $  such that  Propositions \ref{proposiiton-initial-data} and \ref{proposition-reduction-finite} hold.  Using Proposition  \ref{proposiiton-initial-data}, we  have  the following
$$   \forall (d_0, d_1) \in \mathcal{D}_{ A},  \quad U_{d_0, d_1} (0) \in S (0).$$
  In particular, it follows   from   Proposition 1.2.2 page 12 in Kavallaris and Suzuki \cite{NKMI2018} together with Lemma \ref{equivalent-bar u-U} that    equation \eqref{equa-U} is locally in time  well-posed in $C^{2,1} (\Omega \times (0,T_0)) \subset C(\bar \Omega \times [0,T_0])$, for some $T_0 > 0$ . Therefore,   for   every $  (d_0, d_1)   \in  \mathcal{D}_{A}$, we  define   $t^*(d_0, d_1) \in [0,T) $ as   the   maximum  time, satisfying
$$   U_{d_0, d_1}   \in  S(t),  \forall  t \in [0,t^* (d_0,d_1)), $$ 
 where  $U_{d_0, d_1} $  is the  solution  of  \eqref{equa-U} corresponding  to initial data $U_{d_0,d_1} (0),$ introduced in  \eqref{defini-initial-data}. Then, we have two  possible   cases:
 
 $a)$ Either  $ t^* (d_0, d_1)  = T $  for some    $(d_0,d_1 ) \in \mathcal{D}_{A}, $    then,   we  get  the conclusion  of the  proof.
 
 $b)$ Or  $t^* (d_0, d_1) < T,$  for all $ (d_0, d_1) \in  \mathcal{D}_{A}$. This case in fact never occurs, as we will show in the following. 
 
 \medskip
 Indeed, assuming by contradiction that case  $b)$ hold and  using   the  continuity of   the  solution in time  and the  definition  of the  maximun time  $t^*(d_0, d_1), $ we have 
 $$ U_{d_0, d_1}(t^* (d_0, d_1)) \in   \partial S(t^* (d_0, d_1)).$$
 Thanks  to the finite dimensional reduction property given in item  $(i)$ of  Proposition \ref{proposition-reduction-finite}, we derive  the following  
 $$\left( q_0, q_1 \right)(s_* (d_1, d_2))  \in \partial \hat{\mathcal{V} }_A (s_* (d_0, d_1)),$$
where  $q_0, q_1 $  are defined   in \eqref{represent-non-perp} as the  components  of  $q_{d_0, d_1 },$ which  is a  transformed  function of  $U_{d_0, d_1}$ (see   \eqref{defini-q}) and  $s_*(d_0, d_1)  =  - \ln  (T -t^*(d_0, d_1)).$  Then,  we  may  define the following  mapping 
\begin{eqnarray*}
\Lambda :  \mathcal{D}_{A}  & \to &    \left( [-1,1] \times  [-1,1]^n \right)   \\[0.3cm]
\left( d_0, d_1 \right)  &  \mapsto  &  \frac{s_*^2 (d_0,  d_1)}{A}  \left(  q_0, q_1 \right) \left(  s_* (d_0, d_1) \right). 
\end{eqnarray*}
From the definition  of $t^* (d_1, d_2), $ the  components $(q_0,q_1)$ and  the transversal crossing property given in  item $(ii)$  in Proposition  \ref{proposition-reduction-finite}, we  see that   $  \Lambda$  is continuous   on  $\mathcal{D}_{A}$.  In addition to that, from   item $(i)$ of Proposition   \ref{proposiiton-initial-data}, we can derive   that for all $ (d_0, d_1) \in \partial \mathcal{D}_{A}$
$$ \left(q_0, q_1 \right) (s_0)  \in \partial \hat{ \mathcal{V}}_{A} (s_0), \quad s_0 = - \ln T.$$
However,  using  item $(ii)$  of  Proposition  \ref{proposition-reduction-finite}   again and  the definition  of  $t^* (d_0, d_1)$ we deduce that  
$$  t^*(d_0, d_1) =  0,$$  
which yields
$$   s_* (d_0, d_1)  =  s_0 \text{ and } \Lambda  (d_0,d_1) = \frac{s_0^2}{A} \Gamma (d_0,d_1), $$
where $\Gamma$ is defined in item $(I)$ of Proposition \ref{proposiiton-initial-data}. Hence,    thanks  to      \eqref{deg-Gamma-1-neq-0}, we conclude
$$   \text{ deg }  \left(  \Lambda\left. \right|_{\mathcal{D}_A}    \right) \neq  0.    $$
In fact,  such a mapping    $\Lambda$  can not exist    by using  Index theory.
Hence,  case  $b)$ doesn't occur  only    case $a)$ occurs. Thus,    the conclusion  of   Proposition \ref{proposition-existence-U} follows.
\end{proof} 
\subsection{The conclusion of   Theorem  \ref{theorem-existence}} \label{subsection-theorem-11}
In  this  subsection, we would like to  give a  complete  proof of   Theorem \ref{theorem-existence}.  We  now  consider  the  solution  $U$ which has been constructed   in Proposition \ref{proposition-existence-U}. Then, $U$ exists on $[0,T) $ and   
$$  U (t) \in S (t), \forall   t \in [0,T).$$
  Using  item $(i)$ in Definition \ref{defini-shrinking-set-S-t},  we have the following 
  \begin{equation}\label{norm-q-prove-theorem-1}
   q \text{ exists  on } [- \ln T, + \infty)  \text{ and }  \|q(.,s)\|_{L^\infty (\mathbb{R}^n)}  \leq  \frac{C}{ \sqrt s}, \forall  s \in [- \ln T,  + \infty),
  \end{equation}
for some constant  $C > 0$.  Thanks to  \eqref{defini-bar-u}, \eqref{defini-U-x-t},     \eqref{similarity-variables} and  \eqref{defini-w-small}, we  have 
\begin{equation}\label{estima-1-u-T-t-Omega-3-4}
 \left\|   \frac{(T-t)^{\frac{1}{3}} \lambda^{\frac{1}{3}}}{ \bar \theta(t) (1- u(.,t))}    -  \left(  3 + \frac{9}{8} \frac{|.|^2}{ (T-t)|\ln (T-t)|}    \right)^{-\frac{1}{3}}\right\|_{L^\infty (\Omega)} \leq  \frac{C}{ \sqrt{|\ln (T-t)|}}.
\end{equation}
Using      \eqref{bound-bar-theta}    and     \eqref{bound-bar-theta-'}, we can derive that $\bar \theta(t)$ converges to $\theta_T > 0$ with 
$$ \left|   \bar{\theta}(t)  - \theta_T\right| \leq C (T-t)^{\frac{1}{12}},  \forall t  \in [0,T).$$ 
 This implies   \eqref{profile-intermidiate} with $\theta^* =  \frac{\theta_T}{ \lambda^{\frac{1}{3}}}.$   Then, item   $ (i)$ of  Theorem  \ref{theorem-existence}  follows.  We now prove  that  $u$ quenches  only  at $0$. Indeed, from the above estimate, we can derive  that  $0$ is a quenching point of  $u$.   Now, we aim at proving that  $ x \in \Omega \setminus \{0\}$ are  not quenching points  of $u$. In fact, relying on   relations \eqref{defini-bar-u} and \eqref{defini-U-x-t},   it is  enough to  prove the following Lemma:
 \begin{proposition}\label{pro-U-to-U-*}
 The solution $U$ satisfies the following statements:
 \begin{itemize}
 \item[$(i)$] For all $x \in \Omega \setminus \{0\},$ there exits  $\nu (x) > 0$ such that 
 \begin{equation}\label{exist-lim-sup}
 \limsup_{t \to T}  \sup_{|x' - x| \leq \nu (x)} U (x', t)  < + \infty.
\end{equation}  
\item[$(ii)$]   For all $x \in \Omega \setminus \{0\}$,     $\lim_{t \to T} U(x,t) $ exists. In particular,  if  we  define  $U^* (x) =  \lim_{t \to T} U(x,t),$ for all $x \in \Omega \diagdown \{0\}$, then  $u^* \in C (\bar{\Omega}  \setminus \{0\}),$ and  $U(t) $ uniformly converges  to $u^*$ on  every compact subset of  $\bar{\Omega} \setminus\{0\}$.  In particular,  we have  the following  asymptotic behavior
\begin{equation}\label{asymptotics-u^*-theorem-1}
U^* (x) \sim  \left[ \frac{9}{32}  \frac{|x|^2}{|\ln| x||} \right]^{-\frac{1}{3}}, \text{ as } x \to 0. 
\end{equation}
 \end{itemize}
 \end{proposition}
 \begin{proof}
 We consider $U$  the solution  constructed in Proposition \ref{proposition-existence-U}. The proof  will be given  in two parts:
 
 \medskip
 \textit{ -  The proof of item $(i)$:} The proof follows from  the  definition of  shrinking set  $S(t)$. Let us  consider two cases:  $|x| > \frac{\epsilon_0}{4}, x \in \Omega $ and  $|x| \leq   \frac{\epsilon_0}{4}, x \in \Omega$.
 
  + The case where $|x| >    \frac{\epsilon_0}{4}, x \in \Omega $:  Using  item $(iii)$ of  Definition  \ref{defini-shrinking-set-S-t}, we conclude   that for all $t \in [0,T)$, 
  $$    U(x,t)  \leq     U(x,0)  + \eta_0 < + \infty.$$      
  Then, \eqref{exist-lim-sup} follows.
  
  + The case  where  $|x|  \leq  \frac{\epsilon_0}{4}, x \in \Omega$:   For every $x$ in that  region, we can    find   $t_x$ close  to $T$ such that   $ |x| \in \left[ \frac{K_0}{4} \sqrt{(T-t_x) |\ln (T-t_x)|}, \epsilon_0 \right]$. Moreover,  we  derive that  $|x| \in  \left[ \frac{K_0}{4} \sqrt{(T-t) |\ln (T-t)|}, \epsilon_0 \right]$ for all $t \in [t_x, T)$. Considering $t \in [t_x, T) $ and  using  item $(ii)$ in  Definition \ref{defini-shrinking-set-S-t}, we derive the following 
  $$ U(x + \xi \sqrt{ \varrho(x)}, t) \leq \varrho^{-\frac{1}{3}} (x) \left[ \hat{\mathcal{U}}(\tau(x,t)) + \delta_0  \right], \forall |\xi| \leq  \alpha_0 \sqrt{|\ln \varrho(x)|} .$$
  This estimate  directly   implies   \eqref{exist-lim-sup}.
  
  \textit{ - The proof of item  $(ii)$:} By using parabolic regularity and  the technique given by Merle in  \cite{Mercpam92},   item $(i)$ and Lemma  \ref{lemma-parabolic-estimates},  we may derive  that there exists a function  $ U^* \in  C(\bar{\Omega} \setminus \{0\})$ such that $ U(x,t) \to U^*(x),  $ as  $t \to T$, for all $x \in \bar\Omega, x \neq 0$. Moreover,  one can prove that  the   convergence is  uniform on every  compact subset   of     $\bar{\Omega} \backslash \{ 0\}$. It remains to give the  asymptotic behavior  \eqref{asymptotics-u^*-theorem-1}.
 We consider  $x_0 \in \Omega$ such that  $|x_0|$ is  small enough. We first  introduce  the following functions: $\mathcal{U} (x_0, \xi, \tau)$ is  defined in \eqref{rescaled-function-U} and
 \begin{eqnarray}
\mathcal{V}(x_0, \xi, \tau) & = & \nabla_\xi \mathcal{U} (x_0, \xi, \tau),\label{defini-nabla-xi-mathcal-U-x-0}
\end{eqnarray}
where  $\xi \in  \varrho^{-\frac{1}{3}} (x_0)(\Omega - x_0) \subset \mathbb{R}^n$ and  $\tau \in  \left[ - \frac{t(x_0)}{\varrho(x_0)}, 1 \right)$, where $t(x_0)$ and $\varrho(x_0)$ are defined as in \eqref{defini-t(x)-} and \eqref{defini-theta}, respectively.  We aim at proving the following estimates:
\begin{equation}\label{prove-estima-mathcal-U-hat-U}
\sup_{\tau  \in [0,1), |\xi| \leq |\ln (\varrho(x_0))|^{\frac{1}{4}}  } \left| \mathcal{U}(x_0, \xi, \tau)  - \hat{\mathcal{U}} (\tau) \right| \leq  \frac{C}{ |\ln(\varrho(x_0))|^{\frac{1}{4}}},
\end{equation} 
 \begin{equation}\label{prove-sup-tau-xi-V-ln-T-t-0}
\sup_{\tau \in [0,1), |\xi| \leq 2 |\ln(\varrho (x_0))|^{\frac{1}{4}}}   |\mathcal V (x_0, \xi, \tau)| \leq  \frac{C}{ |\ln (\varrho (x_0))|^\frac{1}{4}},
\end{equation}
and 
\begin{eqnarray}
\sup_{\tau \in [\tau_0,1), |\xi| \leq \frac{1}{2} |\ln(\varrho (x_0))|^{\frac{1}{4}}} \left| \partial_\tau \mathcal{U}(x_0, \xi, \tau)\right| \leq C(x_0),   \label{prove-sup0-1)xi-partial-mathcal-U}
\end{eqnarray}
for some $\tau_0 \in (0,1)$, fixed,  and  we also  recall that $\hat{\mathcal{U}} ( \tau) $ is  introduced  in \eqref{defin-hat-mathcal-U-tau}.

\noindent
We see that  \eqref{prove-sup-tau-xi-V-ln-T-t-0} follows  from  the fact that $U \in S(t), \forall t \in [0,T) $ and item $(ii)$ of Definition \ref{defini-shrinking-set-S-t}. Thus,  we only need  to give the proofs of  \eqref{prove-estima-mathcal-U-hat-U} and  \eqref{prove-sup0-1)xi-partial-mathcal-U}.  

\medskip
\textit{ - The proof of  \eqref{prove-estima-mathcal-U-hat-U}:} We write here  the  equation of  $\mathcal{U}$  from \eqref{equa-hat-mathcal-U}
 \begin{equation}\label{equa-mathcal-U-x-0}
 \partial_\tau  \mathcal{U}  =   \Delta_\xi \mathcal{U}  - 2 \frac{ \left| \nabla_\xi \mathcal{U} \right|^2}{ \mathcal{U} +  \frac{  \lambda^{\frac{1}{3}} \varrho^{\frac{1}{3}}(x_0)  }{\tilde \theta (\tau)}} + \left(  \mathcal{U} + \frac{ \lambda^{\frac{1}{3}} \varrho^{\frac{1}{3}} (x_0) }{\tilde{ \theta} (\tau) }   \right)^4 - \frac{\tilde{\theta}'(\tau) }{ \tilde{\theta}(\tau)} \mathcal{U},
\end{equation}
where  $\tilde \theta (\tau) = \bar \theta ( \tau \varrho(x_0)   + t (x_0) ) $  is given   in \eqref{defini-tilde-theta}. From   \eqref{estima-1-u-T-t-Omega-3-4} with $t= t(x_0)$, we derive that  
 \begin{eqnarray}
 \sup_{|\xi| \leq  	 6 |\ln (\varrho(x_0))|^\frac{1}{4}} \left|  \mathcal{U}(x_0, \xi,0) -  \hat{\mathcal{U}} (0)\right| &\leq  &  \frac{C}{  |\ln (\varrho (x_0))|^\frac{1}{4}}.\label{sup-6-ln-T-t-0-x-0mathcal-U}
\end{eqnarray}    
 In addition to that, from item $(ii)$ of Definition \ref{defini-shrinking-set-S-t}, we have  for all $|\xi| \leq  6 |\ln \varrho(x_0)|^{\frac{1}{4}}$ and  $\tau \in [0,1)$:   
 \begin{eqnarray}
\mathcal{U} (x_0, \xi, \tau)  \geq \frac{1}{2} \hat{\mathcal{U}} (0), \label{mathcal-U-geq-hat-U-0}\\
\mathcal{U} (x_0, \xi, \tau) \leq  \frac{3}{2} \hat{\mathcal{U}} (1), \label{mathcal-U-leq-3-2-hat-U-0}
 \end{eqnarray} 
 provided that  $ \delta_0  \leq \frac{1}{2} \hat{ \mathcal{U}}(0)$. We now consider  $\mathbb{U}(\xi, \tau)$ as follows
 \begin{equation*}\label{defi-mathbb-U}
 \mathbb{U} (\xi, \tau)  =  \mathcal{U}(x_0, \xi, \tau) - \hat{\mathcal{U}} (\tau), \text{ where } \xi \in \varrho^{-\frac{1}{3}} (x_0)(\Omega - x_0)  \text{ and } \tau \in [0,1).
 \end{equation*}
 We then derive an equation  satisfied by $\mathbb{U}$
 \begin{equation}\label{equa-mathbb-U}
 \partial_\tau \mathbb{U}  = \Delta_\xi \mathbb{U} + G_1 + G_2 ,
 \end{equation}
 where  $G_1, G_2$ are defined as follows
 \begin{eqnarray*}
 G_1 (\xi, \tau)  & = & - 2 \frac{| \nabla \mathcal{U}|^2}{ \mathcal{U} +  \frac{  \lambda^{\frac{1}{3}} \varrho^{\frac{1}{3}}(x_0)  }{\tilde \theta (\tau)} } -  \frac{ \tilde{\theta}' (\tau)}{\tilde{\theta} (\tau) } \mathcal{U},\\
 G_2 (\xi, \tau) & = &    \left(  \mathcal{U} + \frac{ \lambda^{\frac{1}{3}} \varrho^{\frac{1}{3}} (x_0) }{\tilde{ \theta} (\tau) }   \right)^4 - \hat{\mathcal{U}}^4 (\tau).
 \end{eqnarray*} 
 Next,  we derive from the definition of $\tilde{\theta} (\tau)$, Proposition \ref{propo-bar-mu-bounded} and the fact that for all $\tau \in  \left( 0, 1\right)$,
 $$    \left| \tilde{\theta}' (\tau)  \right| \leq  C \varrho^{\frac{1}{12}} (x_0) (1 - \tau)^{-\frac{11}{12}} ,  $$
 and 
 $$ 1 \leq   \tilde{\theta} (\tau) \leq  C.$$
 Hence,  from  \eqref{prove-sup-tau-xi-V-ln-T-t-0}, \eqref{mathcal-U-geq-hat-U-0} and \eqref{mathcal-U-leq-3-2-hat-U-0}, we  deduce that for all $\tau \in [0,1),|\xi| \leq 2 |\ln \varrho(x_0)|^{\frac{1}{4}}$
 \begin{eqnarray*}
 \left|  G_1 (\xi, \tau) \right|  & = &  \left| -  2 \frac{| \nabla \mathcal{U}(x_0,\xi, \tau)|^2}{ \mathcal{U}(x_0, \xi,\tau) +  \frac{  \lambda^{\frac{1}{3}} \varrho^{\frac{1}{3}}(x_0)  }{\tilde \theta (\tau)} } -  \frac{ \tilde{\theta}' (\tau)}{\tilde{\theta} (\tau) } \left( \mathcal{U}(x_0, \xi, \tau)  \right)  \right| \\[0.3cm]
 & \leq &  \frac{C}{|\ln \varrho (x_0)|^{\frac{1}{4}}} \left( (1 - \tau )^{-\frac{11}{12}}  + 1\right).
\end{eqnarray*}
In addition to that, we derive  from  \eqref{mathcal-U-leq-3-2-hat-U-0} that     
$$   | G_2 ( \xi, \tau)| \leq C | \mathbb{U}(x_0, \xi, \tau)|  + \frac{C}{ |\ln \varrho (x_0)|^{\frac{1}{4}}}.$$
 We  now recall the cut-off function   $\chi_0,$ defined as in  \eqref{defini-chi-0} , then, we introduce  
$$  \phi_1 (\xi)  = \chi_0 \left(   \frac{|\xi|}{ \left| \ln (\varrho(x_0))  \right|^{\frac{1}{4}}} \right).$$
As a matter of fact, we have some  rough estimates  on    $\phi_1$
\begin{equation}\label{bound-phi-anh-nabla-phi}
\| \nabla_\xi \phi_1 \|_{L^\infty(\mathbb{R}^n)} \leq   \frac{C}{ |\ln (\varrho (x_0))|^{\frac{1}{4}}}  \text{ and }  \|  \Delta_\xi \phi_1  \|_{L^\infty(\mathbb{R}^n)} \leq \frac{C}{ \left| \ln(\varrho (x_0))\right|^{\frac{1}{2}}}.
\end{equation}
  Let us  define  $\mathbb{U}_1 (\xi, \tau) = \phi_1 (\xi) \mathbb{U} (\xi, \tau), $  for all $\xi \in \mathbb{R}^n$ and $ \tau \in [0,1)$.   Then, $\mathbb{U}_1$ satisfies  the following   equation
  $$ \partial_t \mathbb{U}_1 =  \Delta \mathbb{U}_1 - 2 \nabla \phi_1 \cdot \nabla  \mathbb{U} - \Delta \phi_1 \mathbb{U}  + \phi_1 G_1 (\xi, \tau) + \phi_1G_2 (\xi, \tau).  $$
  Using Duhamel's principal,  we write  an integral equation satisfied by $\mathbb{U}_1$
$$  \mathbb{U}_1 (\tau) = e^{\tau \Delta} \mathbb{U}_1 (0) + \int_0^\tau e^{(\tau - \sigma)\Delta} \left[ - 2 \nabla \phi_1 \cdot \nabla  \mathbb{U} - \Delta \phi_1 \mathbb{U}  + \phi_1 G_1  + \phi_1G_2  \right](\sigma)d \sigma.  $$
This implies  that  for all $\tau \in [0,1),$ we have
\begin{eqnarray*}
\| \mathbb{U}_1(.,\tau) \|_{L^\infty(\mathbb{R}^n)} \leq  \frac{C}{ |\ln (\varrho(x_0))|^{\frac{1}{4}}} + C \int_0^\tau   \|\mathbb{U}_1(.,\sigma)\|_{L^\infty(\mathbb{R}^n)}  d \sigma.
\end{eqnarray*}
 Thanks to Granwall's inequality, we get the  following
 $$  \| \mathbb{U}_1(.,\tau) \|_{L^\infty(\mathbb{R}^n)} \leq  \frac{C}{|\ln (\varrho(x_0))|^{\frac{1}{4}}}, \forall \tau \in \left[0,1\right), $$
 which  yields \eqref{prove-estima-mathcal-U-hat-U}. 

\medskip
Using \eqref{prove-sup0-1)xi-partial-mathcal-U}, we  can derive that    the limit $  \lim_{\tau \to 1} \mathcal{U}(x_0,0, \tau)$ exists. In addition to that, we derive from    \eqref{prove-estima-mathcal-U-hat-U}  that
$$ U^*(x_0) = \lim_{\tau \to 1} \frac{\mathcal{U}(x_0, 0, \tau)}{\varrho^{\frac{1}{3}} (x_0)}  \sim \left(  \frac{9}{8} \frac{K_0^2}{16} \varrho (x_0)\right)^{-\frac{1}{3}} \sim\left[\frac{9}{16} \frac{|x_0|^2}{|\ln|x_0||} \right]^{-\frac{1}{3}} \text{ as } x_0 \to 0.$$
This is the conclusion of   \eqref{asymptotics-u^*-theorem-1}. So, we get   the  proof in  Proposition \ref{pro-U-to-U-*} and we also get the complete conclusion of Theorem \ref{theorem-existence}. 
 \end{proof}

\section{ Reduction  to a finite  dimensional problem}\label{section-reduction}

This  section  plays a central role  in  our analysis. In fact,  it is  devoted to  the  proof of Proposition  \ref{proposition-reduction-finite}.  More precisely, this section  has  two parts:

- In  the first subsection,  we  prove \textit{ priori estimates }  on  $U$  in $P_1(t), P_2 (t)$ and $P_3 (t)$ when $U$  is trapped  in $S(t)$.

- The second subsection is  devoted to  the conclusion    of Proposition \ref{proposition-reduction-finite}. In fact,   we  use the   first subsection  to  derive  that  $U$   satisfies   almost all  the conditions  in   $S(t)$  with  strict bounds, except  for the bounds  on  $q_0(s) $  and   $  q_1(s)$, with $s = - \ln (T-t)$. This means that  in order to control  $U$  in $S(t)$, we need to  control  only  $(q_0,q_1) (s)$ in $\hat{\mathcal{V}}_A (s),$  defined in \eqref{defini-hat-mathcal-V-A}.  In addition to that, we  also prove  the outgoing transversal crossing property.  It means that    if  the solution   $U$  touches   the boundary of  $S(t_1)$  for  some   $t_1 \in (0, T)$,  then,  $U$  will  be outside   $S(t)$  for all  $t \in (t_1,t_1 + \nu)$ with $\nu$ small enough. In one word,    this   is   the   reduction   to a  finite dimensional problem:   the   control of two components  $(q_0,q_1) (s)$ in  $\hat{\mathcal{V}}_A (s)$. 
\subsection{A priori estimates}

We proceed in 3 steps:  $a, b$  and $c)$, respectively  devoted to parts $P_1 (t), P_2 (t)$ and  $P_3 (t)$.

a) We aim in  the following  Proposition at proving   a priori estimates  for  $U$ in   $P_1 (t)$:
\begin{lemma}\label{propo-estima-in-P-1}
There exists   $K_4 > 0, A_4 > 0$ such that    for all $K_0 \geq K_4, A  \geq A_4$ and  $l^* > 0$  there exists   $T_4 (K_0,  A,  l^* )$   such that for all $ \epsilon_0 >0, \alpha_0 >0,  \delta_0 > 0, \eta_0 > 0,  C_0 > 0, T \leq T_4$ and   for all $l  \in [0, l^*]$,   the following holds:  Assume that  we   have the following conditions: 
\begin{itemize}
\item[-]  We consider initial data  $ U(0) = U_{d_0,d_1} (0)$, given  in \eqref{defini-initial-data} and  $(d_0, d_1) \in \mathcal{D}_{A},$ given in Proposition  \ref{proposiiton-initial-data}   such that    $(q_0,q_1)(s_0)$ belongs  to  $\hat{\mathcal{V}}_A(s_0),$ where  $s_0 = -  \ln T$,  $\hat{\mathcal{V}}_A (s)$ is   defined  in  \eqref{defini-hat-mathcal-V-A} and $q_0, q_1$ are components  of  $q_{d_0, d_1} (s_0),$ a  transform function of  $U,$ defined  in  \eqref{defini-q}.
 \item[-] We have  $U  \in S(T,K_0, \epsilon_0, \alpha_0, A, \delta_0, C_0,\eta_0, t)$  for all $t  \in [T-e^{-\sigma}, T - e^{-(\sigma + l)}]$, for some  $\sigma \geq s_0$ and $l \in [0, l^*]$.
 \end{itemize}
Then, the following estimates hold:
\begin{itemize}
\item[$(i)$] For all  $s  \in  [ \sigma,  \sigma + l	]$, we have
\begin{eqnarray}
\left|q_0'(s)  -  q_0 (s)\right|  + \left| q'_{1,i} (s)  - \frac{1}{2} q_{1,i} (s) \right| \leq  \frac{C}{s^2}, \forall i \in  \{ 1,...,n \} \label{esti-ODE-0-1},
\end{eqnarray}
and
\begin{equation}\label{esti-ODE-2}
\left|  q'_{2,i,j}  (s)   + \frac{2}{s} q_{2,i,j} (s)    \right| \leq \frac{C A}{s^3}, \forall i,j \in \{1,...,n\},
\end{equation}
where  $q_1  = (q_{1,j})_{1 \leq i  \leq n }, q_{2}  = (q_{2,i,j})_{1 \leq i,j \leq n}$ and  $q_1,q_2$ are  defined  in \eqref{defini-R-i}. 
\item[$(ii)$] Control  of  $q_-(s) $:  For all $s \in [\sigma, \sigma + l] , y \in \mathbb{R}^n$ we have  the two   following  cases:
\begin{itemize}
\item[-] The case where  $\sigma \geq  s_0:$
\begin{eqnarray}
\left|  q_-(y,s)   \right| \leq  C \left(  A e^{-\frac{s -\sigma}{2}}  + A^2 e^{- (s - \sigma)^2}  + (s -\sigma)   \right) \frac{(1 +  |y|^3)}{ s^2}, \label{estima-Q--case-sigma-geq-s-0}
\end{eqnarray}
\item[-] The  case where  $\sigma = s_0$
\begin{equation}\label{estima-Q--case-sigma=s-0}
\left| q_-(y,s)  \right| \leq  C (1  +  (s -\sigma) ) \frac{(1 + |y|^3)}{s^2}.
\end{equation}
\end{itemize}
\item[$(iii)$] Control  of the  gradient   term of $q$: For all  $s \in [\sigma, \sigma + l] , y \in \mathbb{R}^n$, we have  the two   following  cases:
\begin{itemize}
\item[-] The case where  $\sigma \geq  s_0:$
\begin{eqnarray}
\left|  (\nabla q)_\perp (y,s)    \right| \leq  C \left(  A e^{-\frac{s -\sigma}{2}}  +  e^{- (s - \sigma)^2}  + (s -\sigma)  + \sqrt{s - \sigma}  \right) \frac{(1 +  |y|^3)}{ s^2}, \label{estima-nabla-Q--case-sigma-geq-s-0}
\end{eqnarray}
\item[-] The  case where  $\sigma = s_0$
\begin{equation}\label{estima--nabla-Q--case-sigma=s-0}
\left| (\nabla q)_\perp (y,s)  \right| \leq  C \left(1  +  (s -\sigma) + \sqrt{s -\sigma} \right) \frac{(1 + |y|^3)}{s^2}.
\end{equation}
\end{itemize}
\item[$(iii)$]  Control   of the  outside part $q_e$: For all $s \in [\sigma, \sigma + \lambda] ,$ we have  the two   following cases:
\begin{itemize}
\item[-] The case where  $\sigma \geq  s_0:$
\begin{eqnarray}
\left\|  q_e(.,s)   \right\|_{L^\infty(\mathbb{R}^n)}  \leq  C \left(  A^2 e^{-\frac{s -\sigma}{2}}  + A e^{ (s - \sigma)}  + 1 +  (s -\sigma)   \right) \frac{ 1}{ \sqrt{s}}, \label{estima-Q-e-case-sigma-geq-s-0}
\end{eqnarray}
\item[-] The  case where  $\sigma = s_0$
\begin{equation}\label{estima-Q-e-case-sigma=s-0}
\left\| q_e(.,s)  \right\|_{L^\infty(\mathbb{R}^n)} \leq  C \left(1  +  (s -\sigma) \right)\frac{1}{\sqrt s}.
\end{equation}
\end{itemize}
\end{itemize}

\end{lemma} 
\begin{proof}
 The proof of  this  proposition relies  completely on  techniques given  by Merle and  Zaag in \cite{MZnon97}. As a matter  of fact,    the  equation  \eqref{equa-Q}  is  quite       the same   as in that  paper  if we  ignore  some  perturbations which  will be  very small in our analysis.   More precisely,  thanks to  Lemmas  \ref{lemma-potential-V},  \ref{lemma-bound-B-Q}, \ref{lemma-bound-T-Q}, \ref{lemma-Bound-R},  \ref{lemma-Bound-N-Q} and     \ref{lemma-F-y-s},  we assert that   the techniques  in  \cite{MZnon97}     hold  in our case.  Hence, we  kindly refer  the reader  to  Lemma 3.2 at page 1523 in \cite{MZnon97}  for more details.
\end{proof}

This  implies a priori estimates in $P_1 (t)$  as follows:
\begin{proposition}[A priori estimates in $P_1 (t)$]\label{propo-priori-P-1-later}
There exist $K_5 \geq 1$ and  $A_5 \geq 1$ such that for all $K_0 \geq K_5, A \geq A_5, \epsilon_0 > 0, \alpha_0 > 0, \delta_0 \leq \frac{1}{2} \hat{\mathcal{U}}(0), C_0 > 0 , \eta_0 >0 $, there exists $T_5 (K_0, \epsilon_0,\alpha_0,A,\delta_0, C_0,\eta_0)$ such that  for all $T \leq T_5$, the following holds: If $U$ a nonnegative  solution  of equation \eqref{equa-U} satisfying  $U \in S(T,K_0,\epsilon_0, \alpha_0,A,\delta_0, C_0,\eta_0, t)$ for all $t \in [0,t_5]$ for some $t_5 \in [0,T)$, and  initial data $U(0) = U_{d_0,d_1}$ given in \eqref{defini-initial-data} for some $d_0,d_1 \in \mathcal{D}_A$ given in Proposition \ref{proposiiton-initial-data}, then, for all $s \in [ -\ln T,- \ln (T-t_5) ]$, we have the following:
\begin{eqnarray*}
 \forall i,j \in \{1, \cdots, n\}, \quad |q_{2,i,j}(s)| & \leq & \frac{A^2 \ln s}{2 s^2}, \label{conq_1-2} \\
\left\| \frac{q_{,-}(.,s)}{1 + |y|^3}\right\|_{L^{\infty}} &\leq &  \frac{A}{2 s^{2}}, \left\|  \frac{(\nabla q(.,s))_{\perp}}{ 1 + |y|^3}  \right\|_{L^\infty} \leq  \frac{A}{2 s^2}  \text{ and } \|q_{e}(s)\|_{L^{\infty}} \leq  \frac{A^2}{2 \sqrt s}, \label{conq-q-1--and-e}
\end{eqnarray*}
where   $q$ is a transformed function of $U$ given in \eqref{defini-q}. 
\end{proposition}
\noindent
\begin{proof}
 The proof   is a   consequence  of Lemma  \ref{propo-estima-in-P-1}.  In particular, the proof  is    the same as in  the work of Merle and Zaag  in \cite{MZdm97}. Hence, we refer the reader to  Proposition 3.7, page 	157 in that work. 
\end{proof}
b) We now show  a priori estimate    on   $U$ in  $P_2 (t)$.  We   start with   the following lemma:
\begin{lemma}[A priori estimates in the intermediate region]\label{lemma-max-t-x-0}
  There exists $K_ 6 $ and   $A_6 > 0, $ such that  for all $K_0 \geq K_6, A \geq A_6,  \delta_6 >0$,  there exists      $\alpha_6 (K_0, \delta_6) > 0, C_6 (K_0, A) >  0$ such that for all $\alpha_0 \leq  \alpha_6, C_0 > 0 $,   there exists  $\epsilon_6 (\alpha_0, A, \delta_6, C_0)$ such that  for all $\epsilon_0 \leq  \epsilon_6 $, there exists $T_6(\epsilon_0, A, \delta_6,C_0)$ and $\eta_6 (\epsilon_0, A, \delta_0, C_0) > 0$ such that for all $T \leq T_6 , \eta_ 0 \leq  \eta_6,   \delta_0 \leq  \frac{1}{2} \left( 3 + \frac{9}{8} \frac{K_0^2}{16} \right)^{-\frac{1}{3}}$,   the following  holds: if   $U \in S(T, K_0, \epsilon_0, \alpha_0, A, \delta_0, C_0, \eta_0, t)$ for all $t \in [0, t_*],$ for some $t_* \in [0,T)$, then, for all $|x| \in \left[ \frac{K_0}{4} \sqrt{(T -t_*) |\ln (T-t_*)|}, \epsilon_0\right]$, we have:
\begin{itemize}
\item[$(i)$] For all $|\xi| \leq \frac{7}{4}   \alpha_0 \sqrt{|\ln \varrho(x)|}     $ and $\tau \in \left[  \max \left( 0, - \frac{t(x)}{\varrho(x)}\right), \frac{t_*- t(x)}{ \varrho(x)} \right]  $,  the  transformed function   $\mathcal{U} (x,\xi, \tau)$ defined in \eqref{rescaled-function-U}   satisfies  the following: 
\begin{eqnarray}
 \left|  \nabla_\xi  \mathcal{U} (x, \xi, \tau)  \right|   & \leq & \frac{2  C_0}{ \sqrt{|\ln \varrho(x)|}}, \label{estima-nabka-mathcal-U-lema}\\
 \mathcal{U}(x, \xi, \tau)  & \geq &   \frac{1}{4} \left( 3   + \frac{9}{8} \frac{K_0^2}{16} \right)^{-\frac{1}{3}}, \label{estima-mathcal-U-leq-K-0-2}\\
   \left|   \mathcal{U} (x, \xi ,  \tau) \right|  & \leq  &  4. \label{mathcal-U-bound-lema}
\end{eqnarray}
\item[$(ii)$] For all $|\xi| \leq  2 \alpha_0 \sqrt{|\ln \varrho(x)|}$ and $\tau_0 =  \max\left( 0, - \frac{t(x)}{\varrho(x)}\right)$: we have
$$ \left|  \mathcal{U}( x, \xi, \tau_0)  - \hat{\mathcal{U}} (\tau_0)    \right| \leq \delta_6 \text{ and  }    \left| \nabla_\xi \mathcal{U}(x, \xi, \tau_0)   \right| \leq \frac{C_6 }{  \sqrt{|\ln \varrho(x)|}}. $$ 
\end{itemize}
\end{lemma}
\begin{proof}
We   leave the proof  to Appendix \ref{appex-lemma-7-4}.
\end{proof}
Using  the above lemma, we now give a priori estimates in  $P_2 (t).$ The following is our statement:
\begin{proposition}[A priori  estimates in  $P_2 (t)$]\label{propo-a-priori-P-2}
There exists  $K_{7} > 0  $ and  $A_7 >  0$ such that for all $K_0 \geq  K_7,   A \geq A_7$,  there exists  $  \delta_7 \leq \frac{1}{2} \hat{\mathcal{U}}(0) $ and $C_7 (K_0,A)$ such that  for all $\delta_0 \leq \delta_7, C_0 \geq C_7$ there exists $\alpha_7 (K_0, \delta_0)$ such that for all $ \alpha_0  \leq \alpha_7$,    there exist $\epsilon_7( K_0, \delta_0, C_0) > 0$  such that for all $\epsilon_0 \leq \epsilon_7$,  there exists $T_7(\epsilon_0, A, \delta_0,  C_0)> 0$  such that  for all $T \leq T_7$ the following holds:  We assume that we have $  U  \in S(T,K_0, \epsilon_0, \alpha_0, A, \delta_0, C_0,t  )$ for all $t \in [0, t_7]$ for some  $t_7 \in [0,  T)$,  then,  for all  $|x| \in \left[ \frac{K_0}{4} \sqrt{(T- t_*) |\ln (T-t_*)|}, \epsilon_0   \right],$        $|\xi| \leq  \alpha_0 \sqrt{|\ln \varrho(x)|} $  and   $  \tau \in \left[ \max  \left( - \frac{t(x)}{\varrho (x)},0\right), \frac{t_7 - t(x)}{\varrho (x)}\right]  $, we have
$$ \left| \mathcal{U}(x, \xi, \tau_*) - \hat{\mathcal{U}} (x, \xi, \tau_*)    \right| \leq  \frac{\delta_0}{2} \text{ and } \left| \nabla\mathcal{U} (x, \xi ,\tau) \right| \leq  \frac{ C_0}{2 \sqrt{| \ln \varrho(x)|}},$$
where  $\varrho (x) = T- t(x)$. 
\end{proposition}
\begin{proof}
  We leave the proof to  Appendix \ref{appen-priori-P-2}.
\end{proof}

\begin{remark}
Unlike what Merle and Zaag  did in \cite{MZnon97}, we don't require any condition  in $\nabla^2 \mathcal{U}$ in $P_2 (t)$ (see Definition \ref{defini-shrinking-set-S-t}), as we have  aldready stated in Remark \ref{remark-change-merle-zaag}. Accordingly, our  a priori estimates  in $P_2 (t)$ will be simpler than  those  of \cite{MZnon97}, as one may see  from the proof given in Appendix C.
\end{remark}

c) We now  give   a priori   estimates     on $U$ in  $P_3 (t)$: 
\begin{proposition}[A priori estimates   in $P_3$]\label{propo-priori-P-3}
Let us  consider  $K_0 > 0, \epsilon_0 > 0, \alpha_0 > 0,A > 0, \delta_0 \in [0, \frac{1}{2} \hat{\mathcal{U}}(0)], C_0 > 0, \eta_0 > 0$. Then, there exists     $T_8( \eta_0)  > 0$  such that  for all    $T \leq T_8$, the following holds: We assume that    $U$ is a nonnegative  solution of  \eqref{equa-U} on $[0, t_8] $ for some $t_8 < T	$, and  $U \in S(K_0,\epsilon_0, \alpha_0, A, \delta_0, C_0, \eta_0,t )$ for all $t \in [0,t_8]$ and  initial data  $U(0) = U_{d_0,d_1}$ given  in \eqref{defini-initial-data} with $|d_0|, |d_1| \leq 2$.  Then,  for all  $|x| \geq   \frac{\epsilon_0}{4}$   and  $ t \in (0, t_8],$
\begin{eqnarray}
\left| U (x,t)  - U(x,0)  \right|  \leq  \frac{\eta_0}{2}, \label{U-U-0-control-P-3}\\
\left| \nabla U (x,t)  -  \nabla e^{t \Delta} U(x,0)  \right|   \leq \frac{\eta_
0}{2}.\label{nabla U-nabla-semi-intial}
\end{eqnarray}
\end{proposition}  
\begin{remark}
As  we  have mentioned in Remark  \ref{remark-change-merle-zaag}, we  draw the attention of the  reader to the change  we have  made with respect to the work of Merle and Zaag  in \cite{MZnon97}: We compare $\nabla U(t)$ to $\nabla e^{t \Delta}U(0)$ and  not to $\nabla U (0)$ in \cite{MZnon97} and this   is crucial, since we are working  on a bounded domain.  
\end{remark}

Following the remark, we have just stated, we give  in the following a crucial  parabolic estimate for the free Dirichlet heat semi-group in $\Omega$: 
\begin{lemma}[A   parabolic  regularity  on the  linear  problem]\label{lemma-regular)linear} Let us consider  initial data  $U_{d_0,d_1},$ given  in \eqref{defini-initial-data}, for some  $|d_0|, |d_1| \leq 2$.  If  we define 
$$    L(t) =  e^{t \Delta} U_{d_0, d_1}, t \in (0,T]. $$
Then,  $L(t)  \in C(\bar \Omega \times [0,T])   \cap  C^\infty ( \Omega \times (0,T] ) $. Moreover, the following  holds
\begin{equation}\label{norm-nabla-e-t-delta-U-0}
\|  \nabla_x L (t)\|_{L^\infty \left(  |x| \geq \frac{\epsilon_0}{8}, x \in \Omega  \right)} \leq C(\epsilon_0), \forall [0,T],
\end{equation}
where  $\epsilon_0$ introduced  in the definition of $U_{d_0, d_1}$.
\end{lemma}
\begin{proof}
See  Appendix  \ref{appex-semi-group}
\end{proof}

\begin{proof}[The proof of Proposition \ref{propo-priori-P-3}]
  We   rewrite the equation satisfied by $U$ as follows
  $$   \partial_t U = \Delta U +  G(U),$$
  where 
  $$G(U) =   -2 \frac{|\nabla U|^2}{ U + \frac{\lambda^{\frac{1}{3}}}{ \bar\theta (t)}}   +  \left(  U + \frac{\lambda^{\frac{1}{3}}}{\bar \theta(t)}   \right)^4 - \frac{\bar \theta'(t)}{\bar \theta (t)} U.  $$ 
We remark that in order to get the conclusion, it is  enough to prove that for all $ x \in \Omega, |x| \geq   \frac{\epsilon_0}{4}$  and $t \in (0,t_8],$ we have the following estimates
  \begin{eqnarray}
\left| U_1 (x,t)  - U_1(x,0)  \right|  \leq  \frac{\eta_0}{2}, \label{U-1-U-0-control-P-3}\\
\left| \nabla U_1 (x,t)  -  \nabla e^{t \Delta}U_1(x,0)  \right|   \leq \frac{\eta_0}{2},\label{nabla U-1-nabla-semi-intial}
\end{eqnarray}
 where  $ U_1 (x,t) = \exp\left( \displaystyle \int_0^t \frac{ \bar \theta' (s)}{ \bar \theta (s)} ds \right) U(x,t).$ Using  the equation  satisfied  by  $U$, we may derive an  equation satisfied by $U_1$ as follows:
\begin{equation}\label{equa-U-1-P-3}
\partial_t U_1 =   \Delta U_1 + G_1,
\end{equation} 
 where $G_1  (t)=  \exp\left( \displaystyle \int_0^t \frac{ \bar \theta' (s)}{ \bar \theta (s)} ds \right) \left[   -2 \frac{|\nabla U|^2}{ U + \frac{\lambda^{\frac{1}{3}}}{ \bar\theta (t)}}   +  \left(  U + \frac{\lambda^{\frac{1}{3}}}{\bar \theta(t)}   \right)^4 \right]$.
	In particular, from the fact  that  $U \in S(t)$ and Proposition  \ref{propo-bar-mu-bounded}, we can derive the following
	$$  \left| \exp \left( \pm  \int_0^t \frac{\bar \theta' (s)}{ \bar \theta (s)} \right) ds \right|  \leq 2.$$
Moreover, from item $(iii)$   of Definition of  \ref{defini-shrinking-set-S-t}   and 	 Lemma   \ref{lemma-regular)linear},        we derive  the following:
$$ |G_1 (x,t)| \leq C( K_0, \epsilon_0,  \eta_0) , \forall |x| \geq \frac{\epsilon_0}{8} \text{ and } \forall t\in (0,t_8].$$
 In the following, we first prove  \eqref{U-1-U-0-control-P-3} then  \eqref{nabla U-1-nabla-semi-intial}.

\medskip
\textit{ + The proof of \eqref{U-1-U-0-control-P-3}:} We  consider   a cut-off function $\chi_2 \in C^\infty_0 (\bar \Omega)$ such that  $\chi_2 = 1$ for all $|x| \geq \frac{\epsilon_0}{6}, x \in \bar \Omega$ and  $\chi_2 = 0 $ for all $|x| \leq \frac{\epsilon_0}{8}$ 	and  $|\nabla \chi_2| + |\Delta \chi_2| \leq  C(\epsilon_0).$    If we  define $U_2 = U_1 \chi_2, $ then   $U_2$ satisfies the following
$$  \partial_t U_2 =  \Delta U_2 + G_2,$$
where 
$$ G_2 (U) = - 2 \nabla U_1 \cdot \nabla \chi_2  - \Delta \chi_2 U_1  - \chi_2 G_1.$$
Using  the estimate   of $G_1  $  and the following fact
$$  |\nabla U_1(x,t)| + |U_1(x,t)|  \leq C(K_0, \epsilon_0, \eta_0), \forall |x| \geq  \frac{\epsilon_0}{8} \text{ and  } t\in [0,t_8],$$
which is a consequence of the fact that $U \in S(t)$ (particularly items $(i)$ and $(iii)$ in Definition \ref{defini-shrinking-set-S-t}), we conclude the following 
  $$ \left\|    G_2 (x, t) \right\|_{L^\infty (\Omega)} \leq  C (K_0,\epsilon_0, C_0,\eta_0), \forall |x| \geq \frac{\epsilon_0}{8} \text{ and }  \forall t \in [0, t_8].$$
We now use a  Duhamel formula to   write  $U_2$  as follows
\begin{equation}\label{duhamel-U-2-P-3}
U_2 (t)   =  e^{t \Delta } U_2(0)  +   \int_{0}^t e^{(t-\tau)\Delta} \left(G_2 (U(\tau)) \right)  d\tau,
\end{equation}  
  where  $e^{t \Delta }$ stands for  the  Dirichlet heat semi-group on $\Omega$ (see more in Appendix \ref{appex-semi-group}). In particular, we have for all $U_0 \in L^\infty (\Omega),$ 
  \begin{eqnarray*}
   \left\| e^{t \Delta }U_0\right\|_{L^\infty (\Omega)} \leq  \|U_0 \|_{L^\infty (\Omega)}.
  \end{eqnarray*}
 Therefore,
  \begin{eqnarray*}
  \left| U_2 (t) - U_2(0)   \right| &\leq &  \left| U_2 (t) - e^{t \Delta } U_2 (0) \right| + \left|e^{t \Delta }U_2 (0) - U_2 (0)\right|\\
  & \leq  & \left| \int_0^t  e^{(t-s) \Delta }G_2 (s) ds    \right| + \left|  e^{t \Delta } U_2(0) - U_2 (0)  \right|\\
  & \leq  & C(K_0, \epsilon_0, C_0, \eta_0) T + \left\|  e^{t \Delta } (U_2 (0)) - U_2 (0)\right\|_{L^\infty(\Omega)}.
  \end{eqnarray*}
  In addition to that, because $ U_2(0)$ is smooth and  has a compact  support  in $\Omega$,  we can prove that
  $$\left\|  e^{t \Delta } (U_2 (0)) - U_2(0) \right\|_{L^\infty(\Omega)} \to 0 \text{ as } t \to 0,$$ 
which  yields the fact that
$$   \left\| U_2 (t) - U_2(0)   \right\|_{L^\infty (\Omega)} \leq  \frac{\eta_0}{2},  $$ 
  provided  that  $T \leq  T_{8,1}(K_0, \epsilon_0, C_0, \eta_0).$  This  concludes  the proof of \eqref{U-1-U-0-control-P-3}.
  \medskip
  
  \textit{ + The  proof  of \eqref{nabla U-1-nabla-semi-intial}:} We derive from \eqref{duhamel-U-2-P-3}  the following fact:
  $$
  \nabla U_2 (t) = \nabla e^{t \Delta} U_2(0) + \int_0^t \nabla e^{(t-\tau)\Delta } G_2(\tau) d\tau.$$
  This implies that
  \begin{eqnarray*}
  \left| \nabla U_2 (t)  - \nabla e^{t\Delta} U_1(0) \right| \leq \left| \nabla e^{t \Delta} U_2(0)  - \nabla e^{t\Delta} U_1(0)  \right| + \left|  \int_0^t \nabla e^{(t-\tau)\Delta } G_2(\tau) d\tau  \right|.
  \end{eqnarray*}
 Using  \eqref{defi-semi-group} and  Lemma  \eqref{Green-function} below, we derive that
 $$ \left|  \int_0^t \nabla e^{(t-\tau)\Delta } G_2(\tau) d\tau  \right| \leq  C(K_0,\epsilon_0,C_0, \eta_0) \int_{0}^t \frac{1}{\sqrt{t-\tau}} d\tau \leq C(K_0,\epsilon_0,C_0, \eta_0)  \sqrt{T}.$$
In order to  finish  the proof, it is enough to   prove that for all   $|x| \geq \frac{\epsilon_0}{4}$, we have
 \begin{equation}\label{would-like-nabla-2-nabla-1}
 \left| \nabla e^{t \Delta} U_2(0)  - \nabla e^{t\Delta} U_1(0)  \right| \leq \frac{\eta_0}{4},
 \end{equation}
 provided that $T \leq T_{8,2}$. Indeed,  using  the definition of the Dirichlet heat semi-group and  Lemma \ref{Green-function} below,  we may write    the following: 
\begin{eqnarray*}
\left| \nabla e^{t \Delta} U_2(0)  - \nabla e^{t\Delta} U_1(0)  \right| &=&  \left| \int_\Omega \nabla_x G(x,y,t,0) (1 - \chi_2(y)) U_{d_0,d_1} (y) dy \right|\\
& \leq & C \int_{|y| \leq  \frac{\epsilon_0}{6}} \frac{\exp \left( - \frac{|x-y|^2}{t}\right)}{t^{\frac{n+1}{2}}} |U_{d_0,d_1}(y)| dy\\
& \leq &C \int_{|y| \leq  \frac{\epsilon_0}{6}} \exp \left( - \frac{|x-y|^2}{t}\right)  \left( \frac{|x-y|}{\sqrt t} \right)^{n+2} \frac{\sqrt{t}}{|x-y|^{n+2}} |U_{d_0,d_1}(y)| dy\\
& \leq  &C (\epsilon) \sqrt{t} \int_{|y| \leq  \frac{\epsilon_0}{6}} |U_{d_0,d_1} (y)| dy\\
&\leq & C(\epsilon_0) \sqrt{t} \|U_{d_0,d_1}\|_{L^1 (\Omega)} \leq C(\epsilon_0) \sqrt{T}.
\end{eqnarray*} 
 This yields \eqref{would-like-nabla-2-nabla-1}, provided that $T \leq T_{8,3 } (\epsilon_0)$. In particular, from the defintions of $U_2 $ and $ U_1$, we can derive  \eqref{nabla U-1-nabla-semi-intial}.   Finally, we get the conclusion of  Proposition  \ref{propo-priori-P-3}.
\end{proof}
\subsection{  The conclusion of  the proof of  Proposition \ref{proposition-reduction-finite} }
It this part,   we   aim at  giving a complete         proof to   Proposition \ref{proposition-reduction-finite}:

\begin{proof}[The proof of Proposition \ref{proposition-reduction-finite} ]
   We first   choose    the parameters $K_0,\epsilon_0 > 0,  \alpha_0 > 0, A  >0,    \delta_0> 0, \delta_1 > 0, C_0> 0,  \eta_0> 0$  and  $ T > 0$  such that   Propositions \ref{proposiiton-initial-data}, \ref{propo-priori-P-1-later}, \ref{propo-a-priori-P-2} and  \ref{propo-priori-P-3} hold.    In particular,  the constant $T$ will be fixed small later. Then,  the conclusion of  the proof follows as we will show in the following.   We   now  consider   $U,$ a solution of  equation  \eqref{equa-U}, with  initial data  $U_{d_0,d_1} (0)$,  defined  in   Definition \ref{defini-initial-data} and  satisfying the following:  
$$U \in S(T,K_0, \alpha_0, \epsilon_0,  A, \delta_0, C_0, \eta_0, t)= S(t), $$
 for all  $ t \in [0, t_*]$ for some $t_* \in (0, T)$ and 
$$ u \in \partial  S(t_*).$$ 
 \medskip
$(i)$ Using  Propositions  \ref{propo-priori-P-1-later}, \ref{propo-a-priori-P-2} and  \ref{propo-priori-P-3}, we can derive that 
\begin{equation}\label{q-1-2-V-A-s-*}
 (q_1,q_2) (s_*) \in  \partial  \hat V_A (s_*),
\end{equation}
where $s_* = \ln (T-t_*)$. 

\medskip
 $(ii)$ Using item  $(i)$, we  derive that  either  
$$ \left|  q_0(s_*)  \right|  = \frac{A}{s^2_*},$$
or there exists $j_0 \in \{ 1,...,n\}$ such that
 $$  \left| q_{1,j_0} (s_*) \right| = \frac{A}{s^2_*}.$$
Then, without loss of generality, we can suppose that the first case occurs, because the argument is the same in   other cases. Hence, using \eqref{esti-ODE-0-1} in Lemma \ref{propo-estima-in-P-1}, we see that 
$$ \left| q'_0 (s) - q_0 (s) \right| \leq \frac{C}{s^2}.$$
Therefore,  we obtain  that   the sign  of  $q_0' (s_*)  $ is opposite  to the sign of  
$$ \frac{d }{ds } \left( \epsilon_0 \frac{A}{s^2}\right)(s_*),$$
 provided that   $A \geq 2 C$,    where  $\ep_0 = \pm 1$ and  $q_0 (s_*) = \epsilon_0 \frac{A}{s_*^2}$. This means that the flow of  $q_0 $ is transverse outgoing on the bounds of  the shrinking set
$$   -\frac{A}{s^2} \leq  q_0 (s) \leq \frac{A}{s^2}.$$ 
It follows   then that  $(q_0,q_1)(	s)$   leaves $\hat{\mathcal{V}}(s)$  at $ s_*$. Thus, we conclude item $(ii)$. Finally, we get the conclusion of Proposition \ref{proposition-reduction-finite} 
\end{proof}

\appendix

\section{ Preparation  of  initial data }\label{preperation-initial-data}
In this   section,  we     give   the proof of   Proposition \ref{proposiiton-initial-data}. More precisely, we aim at  proving  the following lemma    which    directly implies Proposition \ref{proposiiton-initial-data}:
\begin{lemma}\label{lemma-appendix-initial-data} There exists $K_2 > 0$ such that   for all $K_0 \ge K_2, \delta_2 >  0,$  there  exist $\alpha_2 (K_0, \delta_2) > 0 ,  C_2   > 0 $  such that      for  all  $ \alpha_0 \in (0, \alpha_2] $ there exists $\epsilon_2  (K_0, \delta_2,  \alpha_0) > 0 $ such that     for all $ \epsilon_0 \in (0, \epsilon_2]$ and  $ A \geq 1$,  there exists  $T_2 (K_0,  \delta_2, \epsilon_0, A,  C_2) > 0 $ such that  for all $T \in  (0, T_2],$  there exists     a subset   $\mathcal{D}_{A}  \subset [-2,2]^{1 +n}  $ such that  the  following  properties hold: Asumme  that  initial data $U_{d_0, d_1} (0)$ is  given as  in \eqref{defini-initial-data}, then: 

\medskip
$A)$  For all  $(d_0, d_1) \in  \mathcal{D}_{A}$, we have  initial data $U(0)= U_{d_0, d_1}(0) \in S(T, K_0,  \epsilon_0, \alpha_0,  A, \delta_2,  0, C_2, 0,0)$. In particular,  we have   the following:
\begin{itemize}
\item[$(i)$]  Estimates  in  $P_1 (0)$:  we have  the transformed function $q(s_0)$ of $U_{d_0, d_1} (0),$ trapped in $V_{K_0, A} (s_0),$ where  $s_0 = - \ln T$   and  we have also the  following estimates:
$$ \left|  q_0 (s_0)  -  \frac{A d_0}{s_0^2} \right|   + \left| q_{1,j} (s_0) - \frac{Ad_{1,j}}{s_0^2}\right|  \leq  C e^{-s_0}, \text{ for all } j  \in \{ 1,...,n\}, $$
$$  \left| q_{2,i,j}  (s_0) \right|  \leq \frac{\ln s_0}{ s_0^2},  \text{ for all } i,j \in \{ 1,...,n\},$$
$$   \left|  q_- (y,s_0) \right| \leq \frac{1}{s_0^2} (1  + |y|^3), \quad  \left|( \nabla_y q)_\perp (y,s_0) \right|   \leq \frac{1}{s_0^2} (1+  |y|^3), \text{ for all } y \in \mathbb{R}^n,   $$
and 
$$ q_e (s_0)    \equiv  0, $$
where  the components of  $q$ are defined  in  \eqref{represent-with-perp}.
\item[$(ii)$] Estimates in $P_2(0)$: For all   $|x| \in \left[ \frac{K_0}{4}  \sqrt{T|\ln T|}, \epsilon_0 \right]  $ and   and  $|\xi| \leq \alpha_0  \sqrt{|\ln \varrho(x)|}$, we have
$$   \left|  \mathcal{U} (x, \xi,  \tau_0(x) )   - \hat{\mathcal{U}} (\tau_0(x))\right| \leq  \delta_2, \text{ and }    \left| \nabla_\xi \mathcal{U} (x, \xi, \tau_0 (x)) \right| \leq  \frac{C_2}{ \sqrt{|\ln \varrho(x)|} },$$
where $\tau_0(x) = - \frac{t(x)}{\varrho(x)}$ and  $\mathcal{U}, \hat{\mathcal{U}}, t(x), \varrho(x)$ are given in  \eqref{rescaled-function-U},    \eqref{defini-t(x)-},  \eqref{defini-theta} and \eqref{defin-hat-mathcal-U-tau}.
\end{itemize}   

\medskip
$B)$ We have the  following 
$$   (d_0, d_1)  \in  \mathcal{D}_A    \text{ if and only if }  (q_0, q_1) (s_0) \in  \hat{\mathcal{V}}_A (s_0)    $$
and 
$$  (d_0, d_1) \in \partial  \mathcal{D}_A  \text{ if and  only if  }   (q_0, q_1) (s_0) \in \partial \hat{\mathcal{V}}_A (s_0),$$
where  $\hat{\mathcal{V}}_A(s)$ given in  \eqref{defini-hat-mathcal-V-A}. 
\end{lemma}
\begin{proof}
 We see that         part $B)$  directly follows  from item $(i)$ of   part  $A)$.   In addition to that,      our definition   is almost     the same as  in  \cite{TZpre15} (see also  the work of  Ghoul, Nguyen and Zaag \cite{GNZpre16a}, the  works of Merle and Zaag     \cite{MZnon97}, \cite{MZdm97}).    So, we  kindly refer  the reader   to see the  proofs of the existence  of  the set $\mathcal{D}_{A}$,   item $i$ in $A)$ and  part $B)$   in  Proposition 4.5 in   \cite{TZpre15}. Here we only give    the proof of item $(ii)$ in part $A)$.  We now consider  $T > 0,  K_0 > 0, \epsilon_0 > 0, \alpha_0 > 0, \delta_2 > 0, C_2 > 0, \eta_0 >0$.  We aim at  proving that  if these constants  are suitably  chosen, then for all $x \in  \left[   \frac{K_0}{4} \sqrt{T |\ln T |}, \epsilon_0   \right]$ and $|\xi| \leq  2 \alpha_0 \sqrt{|\ln \varrho(x)|},$ where  $\varrho(x) $ given in \eqref{defini-t(x)-}, we have the following
 $$ \left|  \mathcal{U}(x, \xi, \tau_0(x)) - \hat{\mathcal{U}}(\tau_0(x)) \right| \leq \delta_2, \quad \left| \nabla_\xi \mathcal{U}( x, \xi, \tau_0(x))  \right|  \leq  \frac{C_2}{ \sqrt{|\ln \varrho(x)|}}.$$
  We observe  from  the  definition  of $t(x)$ given in   \eqref{defini-t(x)-} that  if   $\alpha_0 \leq \alpha_{2,1} $ and $\epsilon_0 \leq  \epsilon_{2,1}$, then,  for all $x \in \left[ \frac{K_0}{4} \sqrt{T \ln T},  \epsilon_0\right]$ and   $|\xi| \leq  2 \alpha_0 \sqrt{|\ln \varrho(x)|} $,  we have
 $$\left|  \xi  \sqrt{\varrho(x)} \right| \leq \frac{|x|}{2},$$ 
 which yields
 \begin{eqnarray}
 \frac{r_0}{2} \leq \frac{|x|}{2} \leq  \left|  x + \xi \sqrt{T(x)} \right| \leq \frac{3}{2} |x|, \text{ with } r_0 = \frac{K_0}{4} \sqrt{T|\ln T|}.\label{|x|-leq-|x+xi|}
 \end{eqnarray}
 Hence,  for all  $x \in \left[ \frac{K_0}{4} \sqrt{T |\ln T|}, \epsilon_0 \right],$ we have
 $$ \chi \left(  16 (x + \xi \sqrt{ \varrho(x)}) \sqrt{T}, - \ln T   \right)  \chi_1 ( x + \xi \sqrt{\varrho(x)})  = 0, $$
 where  $\chi$ and  $\chi_1$  are defined    in  \eqref{defini-chi-y-s} and  \eqref{defini-chi-1}, respectively.  Therefore, from  \eqref{defini-initial-data} and  the definition of  $\mathcal{U}$ in \eqref{rescaled-function-U}, we may derive that for all    $x \in \left[ \frac{K_0}{4} \sqrt{T |\ln T|}, \epsilon_0 \right] $ and  $|\xi| \leq 2 \alpha_0 \sqrt{|\ln \varrho(x)|},$ 
 \begin{eqnarray*}
 \mathcal{U}(x, \xi, \tau_0) =  (I) \chi_1 \left( x + \xi \sqrt{\varrho(x)}\right) + (II) \left(1 - \chi_1(x + \xi \sqrt{\varrho(x)}) \right),
 \end{eqnarray*}
 where 
 \begin{eqnarray*}
 (I) & = &  \left( \frac{\varrho(x)}{T} \right)^{\frac{1}{3}} \left( 3 +  \frac{9}{8} \frac{|x + \xi \sqrt{\varrho(x)}|^2}{ T |\ln T|}\right)^{-\frac{1}{3}},\\
 & \text{ and }&\\
(II) & = & \varrho^{\frac{1}{3}}(x) H^*( x + \xi \sqrt{\varrho(x)}),  
 \end{eqnarray*}
 with $H^* (x)$ given in  \eqref{defini-H-epsilon-0}.
In addition to that, from  the definition of  $\varrho (x)$, given in \eqref{defini-theta},  we obtain   the following  asymptotics 
\begin{eqnarray}
\ln \varrho(x)  \sim  2 \ln |x|  \text{ and }  \varrho(x) \sim  \frac{8}{K_0^2}  \frac{|x|^2}{|\ln|x||} \text{ as } |x| \to 0. \label{asymptotic-T(x)-1}
\end{eqnarray}
 Besides that,     we  introduce  $r_0 =  \frac{K_0}{4} \sqrt{T |\ln T|}$ and  $R_0 = \sqrt T |\ln T|.$  Then, the following holds 
 \begin{eqnarray}
\varrho(r_0) \sim   T,  \text{ and } \varrho(R_0) \sim \frac{16}{K_0^2} T |\ln T| \text{ and } \varrho(2 R_0)  \sim \frac{64}{K_0^2}  T |\ln T | \text{ as } T \to 0 . \label{asymptotics-T-R-0}
 \end{eqnarray}
 We   aim in the following  at  giving some estimates on $\mathcal{U}(x, \xi, \tau_0 (x))$ and $ \nabla_\xi  \mathcal{U} (x, \xi, \tau_0 (x))$.
 
  \textit{-  Estimate on $\mathcal{U}$:}  From the  definition of  the cut-off function $\chi_1 $ given  in \eqref{defini-chi-1}, it is enough to prove that  for all $|x| \in \left[   r_0, (2 + \frac{1}{100} )  R_0   \right]$ and   	$|\xi| \leq  2 \alpha_0  \sqrt{|\ln \varrho(x)|}$, we have
 \begin{eqnarray}
 \left|  I_1   - \hat{\mathcal{U}} (\tau_0)\right| \leq  \frac{\delta_2}{2}, \label{prove-I-1-delta-1-2}
\end{eqnarray}  
on one hand  and also  that  for all $ |x| \in \left[ \frac{99}{100}  R_0, \epsilon_0 \right] $  and  $|\xi| \leq  2 \alpha_0  \sqrt{|\ln \varrho(x)|},$ we have
 \begin{eqnarray}
 \left|  I_2  -  \hat{ \mathcal{U}} (\tau_0)   \right| \leq  \frac{\delta_2}{2},\label{prove-2-delta-1-2}
 \end{eqnarray}
on the other hand.  Indeed,  let us  start  with the proof of  \eqref{prove-I-1-delta-1-2}: We consider $|x| \in \left[   r_0, (2 + \frac{1}{100} )  R_0   \right]$ and  	$|\xi| \leq  2 \alpha_0  \sqrt{|\ln \varrho(x)|}$. Then,  we write  the following: 
\begin{eqnarray*}
 \left|  I_1  -  \hat{\mathcal{U}}(\tau_0(x))   \right| & = &   \left| \left(  3 \frac{T}{\varrho(x)} + \frac{9}{8}   \frac{| x + \xi \sqrt{\varrho(x)}|^2}{ \varrho(x) |\ln T|}\right)^{-\frac{1}{3}}   -\left(  3 \frac{T}{\varrho(x)} + \frac{9}{8}   \frac{K_0^2}{16} \right)^{-\frac{1}{3}}    \right|\\
 & = & \left| \left(  3 \frac{T}{\varrho(x)} +  \frac{9}{8}   \frac{K_0^2}{16} + \frac{9}{8}  \left[  \frac{| x + \xi \sqrt{\varrho(x)}|^2}{ \varrho(x) |\ln T|} - \frac{K_0^2}{16}\right]\right)^{-\frac{1}{3}}   -\left(  3 \frac{T}{\varrho(x)} + \frac{9}{8}   \frac{K_0^2}{16} \right)^{-\frac{1}{3}}    \right|.
\end{eqnarray*}
In addition to that, we have
\begin{eqnarray*}
  \frac{| x + \xi \sqrt{\varrho(x)}|^2}{ \varrho(x) |\ln T|} - \frac{K_0^2}{16} &=&  \frac{|x|^2}{\varrho (x) |\ln T|} \left( 1 + 2 \frac{x \cdot \xi}{|x|^2} \sqrt{\varrho (x)} + \frac{|\xi|^2 \varrho (x)}{|x|^2} \right) - \frac{K_0^2}{16}\\
  & =&  \frac{K_0^2}{16} \frac{|\ln \varrho (x)|}{ |\ln T|} \left( 1 + 2 \frac{x \cdot \xi}{|x|^2} \sqrt{\varrho (x)} + \frac{|\xi|^2 \varrho (x)}{|x|^2} \right) - \frac{K_0^2}{16}.
\end{eqnarray*}
Besides that,   we also have  the following: 
\begin{eqnarray*}
\left| \frac{x \cdot \xi}{|x|^2} |\sqrt{\varrho (x)}|\right| \leq 4\alpha_0,\\
\left|  \frac{|\xi|^2}{|x|^2} \varrho (x) \right| \leq 4 \alpha_0^2.
\end{eqnarray*}
Moreover,  for all $|x| \in  \left[r_0, \left(2 + \frac{1}{100} \right)R_0 \right]$,  we derive  from \eqref{asymptotics-T-R-0} that 
$$   \frac{|\ln \varrho (x)|}{ |\ln T|} \sim 1, \text{ as }  T \to 0.$$
So, the following holds
$$ \left| \frac{| x + \xi \sqrt{\varrho(x)}|^2}{ \varrho(x) |\ln T|} - \frac{K_0^2}{16}  \right| \to 0,  $$
as  $(\alpha_0, T) \to (0,0)$. From this fact, we can derive that  if $  T \leq T_{2,1}(K_0, \delta_2), \alpha_0 \leq \alpha_{2,2} (K_0,\delta_2)$, we have 
\begin{eqnarray*}
 \left|  I_1  -  \hat{\mathcal{U}}(\tau_0(x))   \right| & = &   \left| \left(  3 \frac{T}{\varrho(x)} + \frac{9}{8}   \frac{| x + \xi \sqrt{\varrho(x)}|^2}{ \varrho(x) |\ln T|}\right)^{-\frac{1}{3}}   -\left(  3 \frac{T}{\varrho(x)} + \frac{9}{8}   \frac{K_0^2}{16} \right)^{-\frac{1}{3}}    \right|\\
 & \leq  & C(K_0)  \left|      \frac{| x + \xi \sqrt{T(x)}|^2}{ T(x) |\ln T|}   -  \frac{K_0^2}{16} \right| \leq \frac{\delta_1}{2}.
 \end{eqnarray*}
This concludes     the proof of \eqref{prove-I-1-delta-1-2}. 

\noindent
 We now  aim at  proving \eqref{prove-2-delta-1-2}. We  consider  $ |x| \in \left[ \frac{99}{100}  R_0, \epsilon_0 \right] $  and  $|\xi| \leq  2 \alpha_0  \sqrt{|\ln \varrho(x)|}.$     Using  the definition of  $(II),$  we   write  as  follows  
\begin{eqnarray*}
\left| (II) -   \hat{\mathcal{U}}(\tau_0(x)) \right|  & = & \left|  \left( \frac{9}{16} \frac{\left|  x + \xi \sqrt{\varrho (x)} \right|^2}{  \varrho (x)|\ln |x + \xi \sqrt{\varrho(x)}||}  \right)^{-\frac{1}{3}}  -  \left( 3 \frac{T}{\varrho(x)} + \frac{9}{8} \frac{K_0^2}{16} \right)^{-\frac{1}{3}}  \right| \\
& = &  \left|  \left(\frac{9}{8} \frac{K_0^2}{16} +  \frac{9}{16} \left( \frac{\left|   x+ \xi \sqrt{\varrho (x)} \right|^2}{ \varrho (x) |\ln |x + \xi \sqrt{\varrho(x)}||} -\frac{K_0^2}{8}\right) \right)^{-\frac{1}{3}}  -  \left(  \frac{9}{8} \frac{K_0^2}{16}  + 3 \frac{T}{\varrho(x)} \right)^{-\frac{1}{3}}  \right| .
\end{eqnarray*}
 Besides that,  the function $\varrho(x)$ is  radial in  $x$,   and increasing in   $|x| $ when  $|x|$ is    small enough. Then,    for  all $\epsilon_0 \leq \epsilon_{2,1}$  and $|x| \in \left[  \frac{99}{100}
 R_0 , \epsilon_0 \right]$, we have 
 \begin{eqnarray}
\left|\frac{T}{\varrho(x)} \right| \leq \left|  \frac{T}{ \varrho \left(\frac{99}{100}R_0\right)} \right| \leq C(K_0) |\ln T|^{-1} \to 0 \text{ as } T \to 0.   \label{T-T(T-0)-to-0} 
 \end{eqnarray}
In addition to that,  we have 
\begin{eqnarray*}
 \frac{\left|  x + \xi \sqrt{\varrho (x)} \right|^2}{ \varrho (x)|\ln |x + \xi \sqrt{\varrho(x)}||} - \frac{K_0^2}{8} &=& \frac{1}{\varrho (x) |\ln |x + \xi \sqrt{\varrho(x)}||  } \left[    |x|^2 + 2 x \cdot \xi \sqrt{\varrho (x)}  +  |\xi|^2 \varrho (x)      \right] - \frac{K_0^2}{8} \\
 & = & \frac{K_0^2}{16} \left[     \frac{| \ln \varrho (x)|}{ | \ln |x + \xi \sqrt{\varrho (x)}||}  -2  +  4\alpha_0    \frac{| \ln \varrho (x)|}{ | \ln |x + \xi \sqrt{\varrho (x)}||}  \right. \\
 & + &  \left.   4\alpha_0^2  \frac{| \ln \varrho (x)|}{ | \ln |x + \xi \sqrt{\varrho (x)}||}    \right].
\end{eqnarray*}
In particular, we have the following fact
\begin{eqnarray*}
\ln \varrho (x)     &\sim  &     2 \ln |x|, \text{ as } |x| \sim 0,\\  
\frac{1}{ |\ln | x + \xi \sqrt{\varrho (x)} ||}  &\sim & \frac{1}{|\ln |x||} , \text{ as } \alpha_0 \to 0. 
\end{eqnarray*} 
This yields
\begin{eqnarray}
\left| \frac{\left|  x + \xi \sqrt{\varrho (x)} \right|^2}{ \varrho (x)|\ln |x + \xi \sqrt{\varrho(x)}||} - \frac{K_0^2}{8}\right| \to 0 \text{ as } (\epsilon_0, \alpha_0) \to (0,0).\label{estima-x-xi-K-0-2-to-0-alpha-epsilon-0}
\end{eqnarray}
From   \eqref{T-T(T-0)-to-0} and \eqref{estima-x-xi-K-0-2-to-0-alpha-epsilon-0}, we derive that 
\begin{eqnarray*}
\left|   (II) - \hat{\mathcal{U}} (\tau_0 (x))  \right|  \leq C(K_0) \left[  \left| \frac{\left|  x + \xi \sqrt{\varrho (x)} \right|^2}{ \varrho (x)|\ln |x + \xi \sqrt{\varrho(x)}||} - \frac{K_0^2}{8}\right| +    \frac{T}{\varrho (x)} \right] \leq \frac{\delta_2}{2},
\end{eqnarray*}
provided that $\alpha \leq \alpha_{2,3} (K_0, \delta_2), \epsilon_0 \leq \alpha_{2,2}(K_0, \delta_2, \alpha_0)$ and  $T \leq T_{2,3}$.  Thus, $\eqref{prove-2-delta-1-2}$ holds.   Finally, we get the conclusion that for all $|x| \in \left[ \frac{K_0}{4} \sqrt{T |\ln T|},    \epsilon_0\right]$ and  $|\xi| \leq 2 \alpha_0 \sqrt{|\ln \varrho (x)|}$, we have 
$$ \left| \mathcal{U} (x, \xi,  \tau_0 (x))  - \hat{\mathcal{U}} (\tau_0 (x))  \right| \leq \delta_2.$$

 \textit{ - Estimate on $ \partial_\xi  \mathcal{U}$: } From the definition   of  $\mathcal{U} (x,\xi, \tau_0(x)) =  \mathcal{U} \left(x, \xi, - \frac{t(x)}{\varrho(x)}\right)$ given in \eqref{rescaled-function-U} and  the expression \eqref{defini-initial-data} of initial data,  we  decompose  $\nabla_\xi \mathcal{U}$  as follows
 $$ \partial_\xi \mathcal{U} (x, \xi, \tau_0(x)) =  B_1 + B_2 + B_3,  $$
  where 
  \begin{eqnarray}
B_1  & = &     \left[  - \frac{3}{4}  \frac{\varrho^{\frac{5}{6}}(x)}{T^\frac{4}{3} |\ln T |}  (x + \xi \sqrt{\varrho(x)})\left(3 + \frac{9}{8} \frac{\left| x + \xi \sqrt{\varrho(x)}\right|^2}{T |\ln T|} \right)^{-\frac{4}{3}}  \right] \chi_1 ( x + \xi \sqrt{\varrho(x)} ),\\
B_2  & = &    \varrho^{\frac{5}{6}} (x)  \nabla H^* ( x + \xi \sqrt{\varrho(x)}) \left( 1 - \chi_1 (x + \xi \sqrt{\varrho(x)}) \right),\\
B_3    & = &    \left[   \left( \frac{\varrho(x)}{T}\right)^{\frac{1}{3}}  \left( 3 + \frac{9}{8} \frac{| x + \xi \sqrt{\varrho(x)}|^2}{ T | \ln T|} \right)^{-\frac{1}{3}}  + \frac{ 3^{-\frac{1}{3}} n}{|\ln T|} \frac{\varrho^{\frac{1}{3}} (x)}{ T^{\frac{1}{3}}}   - \varrho^{\frac{1}{3}} (x) H^* (x + \xi \sqrt{\varrho(x)})\right]\\
& \times & \sqrt{\varrho(x) } \nabla \chi_1 (x + \xi \sqrt{\varrho(x)}) \nonumber.
  \end{eqnarray}
It is  enough  to prove  the  following estimates: 
  
  - Estimate of $B_1:$ For all $|x| \in \left[ r_0;  (2 + \frac{1}{100}) R_0\right]$ and  $|\xi| \leq 2 \alpha_0 \sqrt{|\ln\varrho(x)|}$ we have
\begin{eqnarray}
\left| B_1\right| \leq \frac{C(K_0)}{ \sqrt{|\ln \varrho(x)|}}. \label{estimate-B-1}
\end{eqnarray}  

- Estimate  of $B_2:$ For all $|x| \in \left[ \frac{99}{100} R_0, \epsilon_0 \right] $  and  $|\xi| \leq  2 \alpha_0 \sqrt{| \ln \varrho(x)|},$ we have
\begin{equation}\label{estimate-B-2}
\left| B_2\right| \leq \frac{C(K_0)}{ \sqrt{|\ln \varrho(x)|}}. 
\end{equation}

- Estimate of   $B_3:$  For all $|x| \in \left[   \frac{99}{100} R_0 ,  (2 + \frac{1}{100}) R_0\right]$ and   $|\xi| \leq  2  \alpha_0 \sqrt{|\ln \varrho(x)|},$ we have 
\begin{equation}\label{estimate-B-3}
|B_3| \leq \frac{C(K_0)}{ \sqrt{|\ln \varrho(x)|}}.
\end{equation}
We  now  start  the proof:

\textit{- Estimate of  $B_1$:}  We have the  fact that  for all $|z| \geq 1$  
$$ \left( 3 + \frac{9}{8} |z|^2 \right)^{-\frac{4}{3}}  \leq  C |z|^{- \frac{8}{3}}.$$
  Then, 
  \begin{eqnarray*}
  \left| B_1 \right| & \leq &  C \frac{\varrho^{\frac{5}{6}} (x)}{T^{\frac{4}{3}} |\ln T|} \frac{T^{\frac{4}{3}} |\ln T|^{\frac{4}{3}}}{ |  x + \xi \sqrt{\varrho(x)}    |^{\frac{5}{3}}}\\
  & \leq  & C   \frac{\varrho^{\frac{5}{6}} (x)|\ln T|^{\frac{1}{3}} }{ |x + \xi \sqrt{\varrho(x)}|^{\frac{5}{3}}}.  
  \end{eqnarray*}
 Using \eqref{|x|-leq-|x+xi|},  we   obtain the following:
$$  \left|   B_1 \right|  \leq  C \frac{\varrho^{\frac{5}{6}} (x)  |\ln T|^{\frac{1}{3}} }{ |x|^{\frac{5}{3}}}.$$
In addition to that,   for all $ |x| \in \left[ r_0, (2 + \frac{1}{100}) R_0 \right]$, we have

$$ |\ln \varrho(x)| \sim  |\ln T|,  \text{ as }  T  \to 0.$$
Then,  we have
$$  |B_1| \leq   \frac{C}{K_0^2} \frac{  \varrho^{\frac{5}{6}} (x) |\ln T|^{\frac{1}{3}} }{ \varrho^{\frac{5}{6}}(x) |\ln \varrho(x)|^{\frac{5}{6}}}   \leq  \frac{ C}{\sqrt{|\ln \varrho(x)|}}, $$
provided that $K_0 \geq K_{2,3}, T \leq T_{2,4} $.  This yields \eqref{estimate-B-1}.

\textit{- Estimate of  $B_2$:} From  the definition of  $H^*(x),$ when  $|x|  \leq \epsilon_0$, $\epsilon_0$  small enough,  we have 
$$ H^* (x)     =   \left[  \frac{9}{16} \frac{|x|^2}{|\ln |x||} \right]^{-\frac{1}{3}}     .$$
This implies 
$$    \left|  \nabla  H^* (x)\right|    \leq  C \frac{|\ln |x||^{\frac{1}{3}}}{|x|^{\frac{5}{3}}}. $$
Hence,  
\begin{eqnarray*}
|B_2| \leq   C \frac{\varrho^{\frac{5}{6}} (x) |\ln |x|| ^{\frac{1}{3}}}{ |x|^{\frac{5}{3}}} \leq  C \frac{|\ln |x||^{\frac{1}{3}}}{ |\ln \varrho(x)|^{\frac{1}{3}}} \frac{1}{ \sqrt{|\ln \varrho(x)|}},
\end{eqnarray*}
on one hand. On the other hand, we have the following 
$$ |\ln \varrho(x)| \sim 2  |\ln|x||, \text{ as } x \to 0. $$
Thus, \eqref{estimate-B-2} holds provided that  $\epsilon_0 \leq \epsilon_{2,4} (K_0)$.

 \textit{ - Estimate of  $B_3$:}  We  first  use  the definition of $\chi_1$ in \eqref{defini-chi-1} to write
 $$    \left|    \nabla_x  \chi _1 (x)\right|  \leq  \frac{C}{  \sqrt{T}  |\ln T|}.      $$
We now consider    $|x| \in   \left[   \frac{99}{100}  R_0, (2 + \frac{1}{100}) R_0 \right]$ and  $|\xi| \leq 2 \alpha_0 \sqrt{|\ln \varrho(x)|}$. We     define  
$$ B_4 =   \left( 3 T + \frac{9}{8} \frac{| x + \xi \sqrt{\varrho(x)}|^2}{ | \ln T|} \right)^{-\frac{1}{3}}  + \frac{ 3^{-\frac{1}{3}} n}{T^{\frac{1}{3}}|\ln T|}   -  H^* (x + \xi \sqrt{\varrho(x)}).   $$
Then, 
$$ B_3   =  B_4 \varrho^{\frac{5}{6}}(x) \nabla  \chi_1 \left(  x + \xi\sqrt{\varrho(x)}  \right).$$
We now aim at   giving some estimates on   $ B_4$ as follows:  Using  the fact that $|x| \in   \left[   \frac{99}{100}  R_0, (2 + \frac{1}{100}) R_0 \right]$, $|\xi| \leq 2 \alpha_0 \sqrt{|\ln \varrho(x)|}$ and  \eqref{|x|-leq-|x+xi|}, we can derive that  
\begin{eqnarray*}
\frac{1}{C}   T |\ln T|  \leq    \frac{|x + \xi  \sqrt{\varrho (x)}|^2}{|\ln T|} \leq C T |\ln T |,\\
\frac{1}{C}   T |\ln T|  \leq    \frac{| x + \xi  \sqrt{\varrho (x)}|^2}{ | \ln | x + \xi \sqrt{\varrho (x)}||}  \leq C T |\ln T|.
\end{eqnarray*}
This implies  that 
$$  | B_4| \leq C  (T |\ln T|)^{-\frac{1}{3}}. $$
Hence, we estimate  $B_3$: 
\begin{eqnarray*}
\left|   B_3 \right| & \leq &  C    (T |\ln T|)^{-\frac{1}{3}} \varrho^{\frac{5}{6}}(x)  \nabla_x \chi_1 \left(  x + \xi \sqrt{\varrho(x)} \right) \\
& \leq  & C \frac{\varrho^{\frac{5}{6}}(x)}{ T^{\frac{5}{6}}} \frac{1}{ | \ln T|^{\frac{4}{3}}}
\end{eqnarray*}
 In addition to that,  for all  $ |x| \in \left[  \frac{99}{100}  R_0 ,( 2 + \frac{1}{100}) R_0  \right] $, we use \eqref{asymptotics-T-R-0} to deduce that
 $$ |\varrho(x)| \leq   C   T | \ln  T|,  $$
 and we also have  the following fact 
\begin{eqnarray*}
|\ln \varrho(x)   | &\sim &  |\ln T|, \text{ as } T \to 0.
 \end{eqnarray*}
So, we conclude that 
$$ | B_3| \leq   \frac{C}{ \sqrt{|\ln \varrho (x)|}}  $$
provided that  $K_0 \geq K_{2,4},    \epsilon_0 \leq \epsilon_{2,5} (K_0, \alpha_0 )$ and  $T \leq T_{2,5} (K_0)$. Thus, we get the conclusion of    \eqref{estimate-B-3}. Finally, the conclusion  of     Lemma  \ref{lemma-appendix-initial-data}   follows.
\end{proof}

\section{A priori estimates  in the intermediate region }\label{appex-lemma-7-4}

In this  section, we aim at giving the proof of Lemma \ref{lemma-max-t-x-0}.   Because  our  definitions  are the same as  in \cite{MZnon97},   estimates        in this  Proposition follow in  the same  way as   in that work.  Hence, we kindly refer 	 the	 reader  to Lemma  2.6 in  page 1515 in that work for the proof of  \eqref{estima-nabka-mathcal-U-lema} and item $(ii)$.  It happens that,  although the authors in \cite{MZnon97} gave a statement which is similar to \eqref{estima-mathcal-U-leq-K-0-2}, they did not  gave the proof.  For that reason, we   give here the proof  of \eqref{estima-mathcal-U-leq-K-0-2} and  \eqref{mathcal-U-bound-lema}. 

\textit{ -  The proof of \eqref{estima-mathcal-U-leq-K-0-2}:}   We consider  $ |x| \in \left[ \frac{K_0}{4} \sqrt{(T-t_*)|\ln (T-t_*)|}, \epsilon_0 \right], |\xi| \leq \frac{7}{4} \alpha_0 \sqrt{|\ln  \varrho (x)|}$  and $\tau \in \left[  \max \left( 0, - \frac{t(x)}{\varrho(x)}\right), \frac{t_*- t(x)}{ \varrho(x)} \right] $.  As a matter of fact, there exists $t \in [0, t_*]$ such that
$$  \tau = \frac{t - t(x)}{\varrho (x)}.$$
Let us define 
$$ X = x + \xi \sqrt{\varrho (x)}.$$
 We aim at considering the    three following  cases:
 
 + The case where $|X| \leq  \frac{K_0}{4} \sqrt{(T-t) |\ln (T-t)|}$. We write  
 $$\mathcal{U} (x,\xi, \tau) =  \varrho^{\frac{1}{3}} (x) U(X, t).$$
 We have the fact that $X \in P_1 (t)$. Then, using item  $(i)$  in Definition  \ref{defini-shrinking-set-S-t} together with item $(i)$  in Lemma \ref{lemma-properties-V-A-s},  we get
 \begin{eqnarray*}
 (T-t)^{-\frac{1}{3}} U(X,t)  &=& W (Y,s)  , \text{ where }  Y =  \frac{X}{\sqrt{T-t}}, s = - \ln (T-t)\\
 &  \geq &  \left( 3 + \frac{|X|^2}{(T-t) |\ln (T-t)|}\right)^{-\frac{1}{3}} - \frac{CA^2}{\sqrt{s}}\\
 & \geq & \frac{1}{2} \left( 3 + \frac{9}{8} \frac{K_0^2}{16}\right)^{-\frac{1}{3}},
\end{eqnarray*}  
provided that $T \leq T_{6,1} (K_0,A)$.  This yields
\begin{eqnarray*}
\mathcal{U}(x, \xi, \tau) \geq  \left(\frac{\varrho (x)}{ T-t}\right)^{-\frac{1}{3}} \frac{1}{2} \left( 3 + \frac{9}{8}    \frac{K_0^2}{16}\right)^{- \frac{1}{3}}.
\end{eqnarray*}
In addition to that, the function $|x|  \mapsto \varrho (x)$  is increasing when $|x| $ is small enough. This implies that
$$ \varrho (x) \leq \varrho \left( \frac{K_0}{4 (1 - \frac{7}{4} \alpha_0)}  \sqrt{(T-t)|\ln (T-t)|} \right).$$
From  \eqref{defini-t(x)-},  \eqref{defini-theta}  and   \eqref{asymptotic-T(x)-1}, we derive that
 $$ \varrho \left( \frac{K_0}{4 (1 - \frac{7}{4} \alpha_0)}  \sqrt{(T-t)|\ln (T-t)|} \right) \sim  \frac{8}{K_0^2} \frac{K_0^2}{16 (1 - \frac{7}{4} \alpha_0)^2} \frac{2 (T-t) |\ln(T-t)|}{|\ln (T-t)|} =  \frac{(T-t)}{ (1 - \frac{7}{4} \alpha_0)^2},$$
 as $T \to 0$.   Hence, we have
$$ \frac{\varrho (x)}{T-t} \leq  4,$$
provided that $\alpha_0 \leq  \frac{2}{7},  T\leq  T_{6,2}$. Finally, we get
$$ \mathcal{U} (x,\xi, \tau) \geq  \frac{1}{4} \left(  3  + \frac{8}{9} \frac{K_0^2}{16}  \right)^{-\frac{1}{3}}.$$
 + The case where  $|X| \in \left[ \frac{K_0}{4} \sqrt{(T-t) |\ln (T-t)|}, \epsilon_0   \right]$. In other words, we have  $X \in P_2 (t)$. We write as follows  
$$ \mathcal{U} (x, \xi, \tau) = \varrho (x) U (X, t).$$
In addition to that, using  item $(ii)$ in the definition of $S(t)$ (see Definition \ref{defini-shrinking-set-S-t}), we get the following:
$$  U(X,t)  =  \varrho^{-\frac{1}{3}} (X) \mathcal{U} (X, 0,  \frac{t-t (X)}{\varrho (X)}) \geq  \varrho^{-\frac{1}{3}} (X)  \frac{1}{2} \left( 3 + \frac{9}{8} \frac{K_0^2}{16}\right)^{-\frac{1}{3}},$$
provided that $\delta_0 \leq \frac{1}{2} \left( 3 + \frac{9}{8} \frac{K_0^2}{16}\right)^{-\frac{1}{3}}.$  In particular, using  the fact that
\begin{equation}\label{fact-ralation-x-X}
 ( 1 - \frac{7}{4} \alpha_0) |x| \leq |X| \leq  (1 + \frac{7}{4} \alpha_0) |x|.
\end{equation}
Then, we get  
$$  \left( \frac{\varrho (x)}{ \varrho (X)}\right)^{\frac{1}{3}}  \geq \frac{1}{2}, $$
provided that  $ \alpha_0 \leq  \alpha_{7,2} (K_0)$ and $|x| \leq \epsilon_{7,2}(K_0, \alpha_0)$. This  yields that 
$$  \mathcal{U}(x, \xi, \tau) \geq  \frac{1}{4} \left( 3 + \frac{9}{8} \frac{K_0^2}{16}\right)^{-\frac{1}{3}}.$$

+ The case where $|X| \geq  \epsilon_0$. This means  $X \in P_3 (t).$  We first have the following fact  
$$  \mathcal{U} (x, \xi, \tau)  =  \varrho^{\frac{1}{3}}(x)  U(X, t)  \geq   \frac{1}{2}  \varrho^{\frac{1}{3}}(x) U(X,0),$$
provided that $  \eta_0 \leq  \frac{1}{2}$ and $\epsilon_0 \leq \epsilon_{6,3}$. We remark also that  $|X| \leq  (1 + \frac{7}{4} \alpha_0) |x| \leq  (1 + \frac{7}{4} \alpha_0) \epsilon_0 \leq \frac{3}{2} \epsilon_0 $. Then, 
$$  U(X,0) =  \left[  \frac{9}{16} \frac{|X|^2}{|\ln |X||} \right]^{-\frac{1}{3}}.$$
Moreover, using  \eqref{asymptotic-T(x)-1} and \eqref{fact-ralation-x-X}, we get
$$ \varrho^{\frac{1}{3}} (x) U(X,0) \geq  \frac{1}{\sqrt[3]{2}} \left[ \frac{9}{8} \frac{K_0^2}{16}\right]^{-\frac{1}{3}} \geq  \frac{1}{2}  \left(3 + \frac{9}{8} \frac{K_0^2}{16} \right)^{-\frac{1}{3}},$$
provided that $\alpha_0 \leq \alpha_{6,4}, \epsilon_0 \leq \epsilon_{6,3}$.

As a matter of fact, we obtain the following 
$$ \mathcal{U} (x, \xi, \tau) \geq \frac{1}{4}   \left(3 + \frac{9}{8} \frac{K_0^2}{16} \right)^{-\frac{1}{3}}.  $$
This completely concludes  the proof   of  \eqref{estima-mathcal-U-leq-K-0-2}.

\medskip
\textit{ -  The proof of  \eqref{mathcal-U-bound-lema}:} The idea of the proof  is  similar to the first one.    We also consider  three cases

+ The case   where  $|X| \leq  \frac{K_0}{4} \sqrt{(T-t)|\ln (T-t)|}$.  This implies that $X \in P_1 (t).$ We write here
 $$ \mathcal{U}(x, \xi, \tau )  = \varrho^{\frac{1}{3}} (x) U (X, t).$$
Using item  $(i)$ in the definition of $S(t)$(see Definition \ref{defini-shrinking-set-S-t}),   together with item $(i)$  in Lemma \ref{lemma-properties-V-A-s}, we derive that 
$$ | U(X,t)   | \leq (T-t)^{-\frac{1}{3}} \left[  \left( 3 + \frac{9}{8} \frac{|X|^2}{(T-t) |\ln (T-t)|} \right)^{-\frac{1}{3}}  + \frac{CA^2}{\sqrt{|\ln (T-t)|}}  \right]  \leq 2 (T - t)^{-\frac{1}{3}},$$
provided that  $T \leq T_{6,5}$. In addition to that, from the following fact
$$  \frac{K_0}{4} \sqrt{ \varrho (X) |\ln \varrho (X)|} = |X|  \leq  \frac{K_0}{4} \sqrt{(T-t) |\ln (T-t)|},$$
this  yields that
$$   \varrho (X) \leq T-t.$$
Then, 
$$ \mathcal{U} (x,\xi, \tau) \leq  2 \left( \frac{\varrho (x)}{\varrho (X)} \right)^{\frac{1}{3}}.$$
On the other hand, using  \eqref{fact-ralation-x-X}, we can derive
\begin{eqnarray}\label{bound-varrho-x-X-2}
\frac{\varrho (x)}{ \varrho (X)} \leq 2,
\end{eqnarray}
provided that $\alpha_0 \leq  \alpha_{6,4}$. This also yields  that
$$ \mathcal{U} (x, \xi,\tau)  \leq  4.$$

+ The case where  $|X| \in  \left[  \frac{K_0}{4} \sqrt{(T-t) |\ln (T-t)|}, \epsilon_0\right]$. This means $X \in  P_2 (t)$. We write  
\begin{eqnarray*}
 \mathcal{U} (x, \xi,\tau) = \varrho^{\frac{1}{3}}(x)   \varrho^{-\frac{1}{3}} (X)  \mathcal{U} (X, 0, \frac{t - t(X)}{\varrho (X)}) .
\end{eqnarray*}
Hence, we derive from item $(ii)$ of Definition  \ref{defini-shrinking-set-S-t}, the fact that  $U \in S(t)$ and  \eqref{fact-ralation-x-X}  that
$$ \mathcal{U} (x, \xi, \tau)   \leq \left( \frac{\varrho (x)}{ \varrho (X)}\right)^{\frac{1}{3}} \mathcal{U} (X, 0, \frac{t-t(X)}{\varrho (X)})  \leq  4,$$
provided that  $K_0 \geq K_{6,2}, \alpha_0 \leq  \alpha_{6,4} (K_0), \delta_0 \leq \delta_{6,1}$.

+  The case where  $|X| \geq  \epsilon_0$. The result  follows  from  item $(iii)$ of Definition  \ref{defini-shrinking-set-S-t}. 

Hence, \eqref{mathcal-U-bound-lema} follows. Finally, we  get the conclusion of Lemma  \ref{lemma-max-t-x-0}.
\section{A priori estimate on $P_2 (t)$}\label{appen-priori-P-2}

In this  section, we aim at giving the  proof of  Proposition \ref{propo-a-priori-P-2}
\begin{proof}[The proof of Proposition \ref{propo-a-priori-P-2} ]
We  first choose  parameters $K_0, \epsilon_0,  \alpha_0, A, \delta_0, C_0, \eta_0, \delta_6$ such that  Lemma  \ref{lemma-max-t-x-0} holds. Then,  items $(i)$ and  $(ii)$ in that Lemma hold.  We would like to prove that:  for all  
$$|x| \in  \left[ \frac{K_0}{4} \sqrt{(T-t_7) |\ln(T-t_7)|}, \epsilon_0 \right], |\xi| \leq \alpha_0 \sqrt{|\ln \varrho(x)|},  \tau \in \left[ \max \left( 0, - \frac{t(x)}{\varrho (x)} \right), \frac{t_7 - t(x)}{\varrho(x)}  \right] = [\tau_0, \tau_7],$$ the following holds
\begin{eqnarray}
\left|   \mathcal{U} (x, \xi, \tau)  - \hat{\mathcal{U}} (\tau) \right| &\leq & \frac{\delta_0}{2},\label{prove-estima-apen-mahcal-u}\\
\left|  \nabla_\xi  \mathcal{U} (x, \xi, \tau)\right|  &\leq &   \frac{ C_0}{ 2\sqrt{|\ln \varrho(x)|}}\label{prove-estima-apen-nabla-mahcal-u}.
\end{eqnarray}
We first recall  equation  \eqref{equa-xi}
$$  \partial_\tau  \mathcal{U}  =   \Delta_\xi \mathcal{U}   -  2 \frac{\left| \nabla \mathcal{U}  \right|^2  }{  \mathcal{U}  +  \frac{\lambda^{\frac{1}{3}}  \varrho^{\frac{1}{3}}(x)}{\tilde \theta ( \tau )}  } +  \left(\mathcal{U}  +  \frac{\lambda^{\frac{1}{3}}\varrho^{\frac{1}{3} }(x)}{ \tilde \theta ( \tau ) } \right)^4 -\frac{\tilde \theta_\tau' (\tau)}{\tilde \theta (\tau)}  \mathcal{U}. $$
\textit{ - The proof of \eqref{prove-estima-apen-mahcal-u}:}  We  first  introduce the following function
$$ \mathcal{Z} ( \xi, \tau ) =  \mathcal{U} ( x, \xi, \tau)  - \hat{\mathcal{U}} (\tau).$$
Using  \eqref{equa-xi}, we write the following equation
$$ \partial_\tau \mathcal{Z} = \Delta \mathcal{Z}   +   \left(\mathcal{U}  +  \frac{\tilde \theta ( \tau ) \varrho^{\frac{1}{3} }(x) }{\lambda^{\frac{1}{3}}} \right)^4  - \hat{\mathcal{U}}^4 (\tau)  + G( \xi, \tau),    $$
where
$$ G(\xi, \tau ) = -  2 \frac{\left| \nabla \mathcal{U}  \right|^2  }{  \mathcal{U}  +  \frac{\lambda^{\frac{1}{3}}  \varrho^{\frac{1}{3}}(x)}{\tilde \theta ( \tau )}  }  -\frac{\tilde \theta_\tau' (\tau)}{\tilde \theta (\tau)}  \mathcal{U}. $$
Using Proposition \ref{propo-bar-mu-bounded} and  the definition of  $\tilde{\theta} (\tau)$ in \eqref{defini-tilde-theta}, we derive that
\begin{equation}\label{estima-theta-'-5-6-appex}
\left| \tilde{\theta}' (\tau)\right| \leq C  \varrho^{\frac{1}{12}} (x) (1-\tau)^{-\frac{11}{12}} . 
\end{equation}
Hence,  from  Lemma   \ref{propo-a-priori-P-2}, we  derive the following: for all $|\xi| \leq  \frac{7}{4} \alpha_0  \sqrt{|\ln \varrho(x)|} $ and  $\tau \in [\tau_0, \tau_7]$,
\begin{eqnarray*}
\left| G( \xi, \tau) \right| \leq  \frac{C}{ |\ln \varrho(x)|^{\frac{1}{2}}} \left( (1 - \tau)^{-\frac{11}{12}} + 1 \right),
\end{eqnarray*}
provided that $|x| \leq \epsilon_{7,2} (K_0, \delta_0)$.
In particular, 
$$  \left|   \left(\mathcal{U}  +  \frac{\tilde \theta ( \tau ) \varrho^{\frac{1}{3} }(x) }{\lambda^{\frac{1}{3}}} \right)^4  - \hat{\mathcal{U}}^4 (\tau)\right| \leq C \left(  |\mathcal{Z} | + \varrho^{\frac{1}{3}}(x)\right). $$
We here define   $ \chi_1 (\xi)=  \chi_0 \left( \frac{|\xi|}{\sqrt{|\ln \varrho(x)|}} \right), $ where $\chi_0 \in C^\infty_0 (\mathbb{R}), \chi_0 (x) = 1,  \forall |x| \leq \frac{5}{4}$,$ \chi_0(x)= 0, \forall |x| \geq \frac{7}{4}$,  and   $ 0 \leq \chi_0 \leq 1$.   As a matter of fact, we have the following estimates
 \begin{eqnarray}
\left|  \nabla \chi_1  \right|  \leq  \frac{C}{\sqrt{|\ln \varrho(x)|}} \text{ and } \left| \nabla^2 \chi_1 \right| \leq  \frac{C}{ |\ln \varrho(x)|}.\label{chi-1-appendix-P-2}
 \end{eqnarray}
 Introducing 
$$ \mathcal{Z}_1 (\xi, \tau)  = \chi_2 (\xi) \mathcal{Z}  (\xi, \tau),$$
we then    write an equation satisfied  by    $\mathcal{Z}_1$
$$ \partial_\tau   \mathcal{Z}_1 = \Delta \mathcal{Z}_1 + G_1( \xi, \tau), $$
where  $G_1$  satisfies  the following:  for all $ |\xi| \leq  \frac{7}{4} \alpha_0 \sqrt{|\ln \varrho(x)|}$  
$$ | G_1(x, \xi, \tau)| \leq  C( |\mathcal{Z}_1| + \frac{1}{ |\ln \varrho(x)|^{\frac{1}{2}}} \left( (1 - \tau)^{-\frac{11}{12}}  + 1 \right),  $$
Using Duhamel's principal, we derive the following
\begin{eqnarray*}
\| \mathcal{Z}_1(\tau)\|_{L^\infty}  & \leq & \left( \delta_6 + \frac{C}{|\ln \varrho(x)|^{\frac{1}{2}}} \right) + C \int_{\tau_0}^\tau \|\mathcal{Z}_1(s)\|_{L^\infty}  ds \\
& \leq & 2 \delta_6 + C \int_0^\tau \|\mathcal{Z}_1 (s)\|_{L^\infty}  ds.
\end{eqnarray*}
Using Gronwall's inequality, we  get the following
$$ \| \mathcal{Z}_1\|_{L^\infty} (\mathbb{R}^n)    \leq 2 C \delta_6.$$
In particular, if we choose $ C_0 \geq 4C \delta_6 $,  then \eqref{prove-estima-apen-mahcal-u} follows.

\textit{- The proof of  \eqref{prove-estima-apen-nabla-mahcal-u}:}  
 We  rely on  the idea as for    the proof of \eqref{prove-estima-apen-mahcal-u}. We consider  $ \mathcal{Z}_2 (\xi, \tau) = \chi_1 \mathcal{U}(x, \xi, \tau) \exp\left( \int_{\tau_{0}}^{\tau } \frac{\tilde \theta' (s)}{ \tilde \theta (s)}  ds\right)     $, where  $\chi_1$ given in the proof of   \eqref{prove-estima-apen-mahcal-u}.  Then,  we can derive an equation satisfied by $\mathcal{Z}_2$ as follows
\begin{eqnarray}
\partial_ \tau \mathcal{Z}_2 = \Delta \mathcal{Z}_2 +   \chi_1 \mathcal{U}^4  \exp\left( \int_{\tau_0}^{\tau } \frac{\tilde \theta' (s)}{ \tilde \theta (s)}  ds\right)      + G_2 (\xi, \tau),
\end{eqnarray}
where   $G_2 $ defined  by 
\begin{eqnarray*}
G_2 (\xi, \tau)  & = &  \exp \left(  \int_{\tau_0}^{\tau } \frac{\tilde \theta' (s)}{ \tilde \theta (s)}  ds  \right) \left[ - 2 \nabla \chi_1 \cdot \nabla \mathcal{U} - \Delta \chi_1 U  - \frac{\chi_1|\nabla \mathcal{U}|^2}{ \mathcal{U} + \frac{\lambda^{\frac{1}{3}} \varrho^{\frac{1}{3}} (x)}{ \tilde \theta (\tau)}}  \right.\\
& + &  \left.  \chi_1  \left(  \mathcal{U} + \frac{\lambda^{\frac{1}{3}} \varrho^{\frac{1}{3}} (x)}{ \tilde \theta (\tau)} \right)^4  - \chi_1 \mathcal{U}^4  \right].
\end{eqnarray*}
In particular, from \eqref{estima-theta-'-5-6-appex}, we can get the following fact
\begin{equation}
\left| \exp \left( \pm  \int_{\tau_0}^{\tau } \frac{\tilde \theta' (s)}{ \tilde \theta (s)}  ds  \right)   \right| \leq 2, \forall \tau \in [\tau_0, \tau_7],
\end{equation}
as $|x| \leq  \epsilon_{8,1}$. Then,  using the results in Lemma  \ref{lemma-max-t-x-0}, we can deduce  the following
$$ \|  G_2 (.,\tau) \|_{L^\infty}  \leq   \frac{C}{|\ln \varrho (x)|}, \forall \tau \in [\tau_0, \tau_7],$$
provided that $|x| \leq  \epsilon_{8,2} (K_0)$. We write  $\mathcal{Z}_2 $ in the following  integral equation
\begin{equation}\label{Duhamel-Z-2}
 \mathcal{Z}_2(\tau) =     e^{(\tau - \tau_0) \Delta} \mathcal{Z}_2 (\tau_0) + \int_{\tau_0}^\tau  e^{ (\tau - s) \Delta} \left[   \chi_1 \mathcal{U}^4(\sigma)  \exp\left( \int_{\tau_0}^{s } \frac{\tilde \theta' ( \sigma)}{ \tilde \theta (\sigma)}  d \sigma\right)  +    G_2(s) \right] ds.
\end{equation}
We  now aim at  proving the following estimates:
\begin{eqnarray}
\left\|   \nabla   e^{(\tau - \tau_0) \Delta} \mathcal{Z}_2 (\tau_0)    \right\|_{L^\infty (\mathbb{R}^n)} & \leq & \frac{C_6 + C}{\sqrt{|\ln \varrho (x)|}},\label{prove-nabla-macal-Z-2}\\
\left\|  \nabla   e^{(\tau -  s) \Delta}  \left(\chi_1 \mathcal{U}^4 (s)  \exp\left( \int_{0}^{s } \frac{\tilde \theta '( \sigma)}{ \tilde \theta (\sigma)}  d \sigma\right)  \right)    \right\|_{L^\infty (\mathbb{R}^n)}   & \leq &  C \| \nabla   \mathcal{Z}_2\|_{L^\infty (\mathbb{R}^n)}  + \frac{C}{\sqrt{|\ln \varrho (x)|}}   \label{prove-lea- nabla-macal-Z-2}.
\end{eqnarray}

+ \textit{  The proof of \eqref{prove-nabla-macal-Z-2}:}  We     write    $ e^{(\tau - \tau_0) \Delta} \mathcal{Z}_2 (\tau_0)   $    as follows
   $$    e^{(\tau - \tau_0) \Delta} \mathcal{Z}_2 (\tau_0)  (\xi, \tau) = \int_{\mathbb{R}^n}    \frac{ e^{\left( - \frac{| \xi - \xi'  |^2}{4 (\tau - s)}\right)} }{ (4 \pi (\tau - s))^{\frac{n}{2}}}   \chi_1 ( \xi')  \mathcal{ U} (x,\xi',\tau_0(x))   \exp\left( \int_{0}^{s } \frac{\tilde \theta '( \sigma)}{ \tilde \theta (\sigma)}  d \sigma\right) d \xi '.$$
 This  yields             
\begin{eqnarray*}
\left| \nabla_{\xi}  e^{(\tau - \tau_0) \Delta} \mathcal{Z}_2 (\tau_0)  (\xi, \tau)   \right| & = &  \left|    \int_{\mathbb{R}^n}     \frac{ \nabla_\xi e^{\left( - \frac{| \xi - \xi'  |^2}{4 (\tau - s)}\right)} }{ (4 \pi (\tau - s))^{\frac{n}{2}}}   \chi_1 ( \xi')  \mathcal{ U} (x, \xi', \tau_0)   \exp\left( \int_{\tau_0}^{s } \frac{\tilde \theta '( \sigma)}{ \tilde \theta (\sigma)}  d \sigma\right) d \xi '    \right| \\
& = &   \left|    \int_{\mathbb{R}^n}     \frac{ \nabla_{\xi'} e^{\left( - \frac{| \xi - \xi'  |^2}{4 (\tau - s)}\right)} }{ (4 \pi (\tau - s))^{\frac{n}{2}}}   \chi_1 ( \xi')  \mathcal{ U} (x, \xi',\tau_0)   \exp\left( \int_{\tau_0}^{s } \frac{\tilde \theta '( \sigma)}{ \tilde \theta (\sigma)}  d \sigma\right) d \xi '    \right|  \\
&  \leq & \left|    \int_{\mathbb{R}^n}     \frac{ e^{\left( - \frac{| \xi - \xi'  |^2}{4 (\tau - s)}\right)} }{ (4 \pi (\tau - s))^{\frac{n}{2}}}  \nabla_{\xi'} \chi_1 ( \xi')  \mathcal{ U} (x,\xi', \tau_0)   \exp\left( \int_{\tau_0}^{s } \frac{\tilde \theta '( \sigma)}{ \tilde \theta (\sigma)}  d \sigma\right) d \xi '    \right|  \\
& + &    \left|    \int_{\mathbb{R}^n}     \frac{ e^{\left( - \frac{| \xi - \xi'  |^2}{4 (\tau - s)}\right)} }{ (4 \pi (\tau - s))^{\frac{n}{2}}}   \chi_1 ( \xi') \nabla_{\xi'}  \mathcal{ U} (x,\xi', \tau_0)   \exp\left( \int_{\tau_0}^{s } \frac{\tilde \theta '( \sigma)}{ \tilde \theta (\sigma)}  d \sigma\right) d \xi '    \right|.
\end{eqnarray*}
Thus, using the above  estimate,  the result  of  item $(ii)$ in  Lemma  \ref{lemma-max-t-x-0} and  \eqref{chi-1-appendix-P-2},  we can  conclude \eqref{prove-nabla-macal-Z-2}.

+ \textit{ The proof of \eqref{prove-lea- nabla-macal-Z-2}:} 
\begin{eqnarray*}
& &\left|  \nabla   e^{(\tau -  s) \Delta}  \left(\chi_1 \mathcal{U}^4 (s)  \exp\left( \int_{\tau_0}^{s } \frac{\tilde \theta '( \sigma)}{ \tilde \theta (\sigma)}  d \sigma\right)  \right) (\xi, s)  \right| \\
& = & \left|    \int_{\mathbb{R}^n}     \frac{ \nabla_\xi e^{\left( - \frac{| \xi - \xi'  |^2}{4 (\tau - s)}\right)} }{ (4 \pi (\tau - s))^{\frac{n}{2}}}   \chi_1 ( \xi')  \mathcal{ U}^4 (x, \xi', \tau)   \exp\left( \int_{\tau_0}^{s } \frac{\tilde \theta '( \sigma)}{ \tilde \theta (\sigma)}  d \sigma\right) d \xi '    \right|\\
& = & \left|    \int_{\mathbb{R}^n}     \frac{ \nabla_{\xi'} e^{\left( - \frac{| \xi - \xi'  |^2}{4 (\tau - s)}\right)} }{ (4 \pi (\tau - s))^{\frac{n}{2}}}   \chi_1 ( \xi')  \mathcal{ U}^4 (x, \xi', \tau)   \exp\left( \int_{\tau_0}^{s } \frac{\tilde \theta '( \sigma)}{ \tilde \theta (\sigma)}  d \sigma\right) d \xi '    \right| \\
& \leq &  \left|    \int_{\mathbb{R}^n}     \frac{  e^{\left( - \frac{| \xi - \xi'  |^2}{4 (\tau - s)}\right)} }{ (4 \pi (\tau - s))^{\frac{n}{2}}} \nabla_{\xi'}  \chi_1 ( \xi')  \mathcal{ U}^4 (x, \xi', \tau)   \exp\left( \int_{\tau_0}^{s } \frac{\tilde \theta '( \sigma)}{ \tilde \theta (\sigma)}  d \sigma\right) d \xi '    \right| \\
& +&   \left|    \int_{\mathbb{R}^n}     \frac{  e^{\left( - \frac{| \xi - \xi'  |^2}{4 (\tau - s)}\right)} }{ (4 \pi (\tau - s))^{\frac{n}{2}}}   \chi_1 ( \xi')  4 \mathcal{ U}^3  (x, \xi', \tau)  \nabla_{\xi'} \mathcal{U}(x, \xi', s)  \exp\left( \int_{\tau_0}^{s } \frac{\tilde \theta '( \sigma)}{ \tilde \theta (\sigma)}  d \sigma\right) d \xi '    \right|.
\end{eqnarray*}
In particular, we have the following fact
\begin{eqnarray*}
\nabla_{\xi'} (\mathcal{Z}_2) (\xi', s) & = & \nabla_{\xi' } \left( \chi_1 (\xi') \mathcal{U}(x, \xi', s) \exp\left( \int_{\tau_0}^{s } \frac{\tilde \theta '( \sigma)}{ \tilde \theta (\sigma)}  d \sigma\right)   \right)\\
& = &   \nabla_{\xi'} \chi(\xi') \mathcal{U}(x, \xi',s)     \exp\left( \int_{\tau_0}^{s } \frac{\tilde \theta '( \sigma)}{ \tilde \theta (\sigma)}  d \sigma\right) +  \chi(\xi')\nabla_{\xi'}   \mathcal{U}(x, \xi',s)     \exp\left( \int_{\tau_0}^{s } \frac{\tilde \theta '( \sigma)}{ \tilde \theta (\sigma)}  d \sigma\right)  
\end{eqnarray*}
Then, using      \eqref{mathcal-U-bound-lema}, \eqref{chi-1-appendix-P-2} and  the definition of  $\mathcal{Z}_2 (s)$, we get the  following 
$$   \left|  \nabla   e^{(\tau -  s) \Delta}  \left(\chi_1 \mathcal{U}^4 (s)  \exp\left( \int_{\tau_0}^{s } \frac{\tilde \theta '( \sigma)}{ \tilde \theta (\sigma)}  d \sigma\right)  \right) (\xi, s)  \right|   \leq  C \| \nabla  \mathcal{Z}_2 (s)\|_{L^\infty(\mathbb{R}^n)}  +  \frac{C}{ \sqrt{ |\ln \varrho(x)|}}, $$
which yields  \eqref{prove-lea- nabla-macal-Z-2}.

We now  come back to the proof of \eqref{prove-estima-apen-nabla-mahcal-u}. We use \eqref{Duhamel-Z-2}, \eqref{prove-nabla-macal-Z-2} and  \eqref{prove-lea- nabla-macal-Z-2} to  obtain the following
\begin{eqnarray*}
\| \nabla \mathcal{Z}_2 (\tau)\|_{L^\infty (\mathbb{R}^n)} \leq \frac{C_6 + C}{\sqrt{|\ln \varrho (x)|}} + C\int_{\tau_0}^{\tau} \| \nabla  \mathcal{Z}_2 (s)\|_{L^\infty (\mathbb{R}^n)}.
\end{eqnarray*}
Thanks to Gronwall's inequality, we derive the following
$$   \|\nabla  \mathcal{Z}_2(\tau)\|_{L^\infty (\mathbb{R}^n)} \leq \frac{C(C_6)}{\sqrt{|\ln \varrho(x)|}}.$$
In addition to that, from the definition of  $\mathcal{Z}_2,$ we deduce that for all $ | \xi| \leq \alpha_0 \sqrt{ |\ln \varrho (x)|}$,
$$    \mathcal{Z}_2 (\xi, \tau)  =      \mathcal{U} (x, \xi, \tau)    \exp\left( \int_{\tau_0}^{\tau } \frac{\tilde \theta '( \sigma)}{ \tilde \theta (\sigma)}  d \sigma\right). $$
This  implies that
$$   \left| \nabla_\xi \mathcal{U}(x, \xi, \tau) \right|   \leq  \frac{2  C (C_6)}{  \sqrt{|\ln \varrho(x)|} }. $$ 
 Finally, if we take  $C_0  \geq  4 C(C_6)$, then
 $$   \left| \nabla_\xi \mathcal{U}(x, \xi, \tau) \right|   \leq \frac{C_0}{2 \sqrt{|\ln \varrho(x)|}},$$
 which implies  \eqref{prove-estima-apen-nabla-mahcal-u}.

\end{proof}
\section{ Some bounds on terms in equation \eqref{equa-Q} }
In this section,   we give  essential ingredients for the proof  of Lemma \ref{propo-estima-in-P-1}. More precisely,  we will   estimate  some functions involved in     equation \eqref{equa-Q}:  $V, J, B, R, N$ and  $F$.  In fact, as we explained in the proof  Section right  after Lemma \ref{propo-estima-in-P-1}, we  choose  not to prove Lemma \ref{propo-estima-in-P-1}, in order to avoid lenghy  estimates aldready mentioned  by Merle and Zaag in \cite{MZnon97}.  The interested  reader may use  our estimates in this section and  follow the proof  of Lemma 3.2 on  page 1523 in \cite{MZnon97} in order  to check the argument.

\medskip
Let us  first give  some estimates  on $V(y,s)$: 

\begin{lemma}[Expansion and  bounds  on  the  potential $V$]\label{lemma-potential-V} We consider   $V$  defined   in \eqref{defini-potential-V}. Then, the following  holds:  $V$ is bounded  on $\mathbb{R}^n \times [1, + \infty)$       and for all $s \geq 1$
$$ \left| V (y,s) \right|  \leq   C \frac{(1 + |y|^2)}{s}, \forall   y \in \mathbb{R}^n,$$
and   
$$ V (y,s) =  - \frac{\left(|y|^2 - 2n \right)}{4s} + \tilde V (y,s),$$
where $\tilde V$ satisfies the following
$$ \left|  \tilde V (y,s)\right| \leq C (K_0)  \frac{(1 + |y|^4)}{s^2}, \forall   |y| \leq K_0 \sqrt s.$$
\end{lemma}
\begin{proof}
The proof is easily derived from  the explicit formula of  $V$. We  kindly refer  the readers to self-chek or see Lemma $B.1,$ page  1270  in \cite{NZens16}  with $p = 4$.
\end{proof}
We now give   a  bound  on the  quadratic term   $B(q)$.
\begin{lemma}[A bound  on   $B(q)$]\label{lemma-bound-B-Q} Let us  consider $B (q)$ defined in \eqref{defini-B-Q}. If     $\theta (s) \geq 1, $ for all $s$ and $|q| \leq 1$,  then,  the  following  holds
$$ \left| B(q) \right|  \leq     C (K_0) \left(  |q|^2   +   e^{-\frac{s}{3}}  \right).$$
\end{lemma}
\begin{proof}
By using Newton binomial formula, the  conclusion  directly follows. 
\end{proof}
\medskip
Next, we aim at  giving   some    bounds    on   $J(q, \theta (s))$. The following is  our statement: 
\begin{lemma}[Bound on $J(q, \theta (s))$]\label{lemma-bound-T-Q}    For all  $K_0 > 0,  A \geq 1$ and  $\epsilon_0 > 0,$ there exist  $ \eta_9 (\epsilon_0 ) $ and  $T_9 (K_0, \epsilon_0, A) $    such that   for all  $\alpha_0 > 0, C_0>0$    and  $T \leq  T_9$,  $\delta_0  \leq   \frac{1}{2} \hat{\mathcal{U}}(0)$ and    $\eta_0 \leq \eta_9$, the following holds:  If  $ U  \in S(T, K_0, \epsilon_0, \alpha_0, A, \delta_0, C_0,  \eta_0, t)$  for some  $t \in [0, T)$, then,  for all $|y| \leq 2 K_0 \sqrt s, $  $s = - \ln(T-t)$, we have  the following estimates:
\begin{eqnarray}
\left|   \left(  J(q, \theta (s) )   + 4 \frac{\nabla \varphi \cdot \nabla q}{ \varphi +\frac{\lambda^{\frac{1}{3}}e^{-\frac{s}{3}}}{\theta (s) }   }  \right)   \right| & \leq & C( K_0, A)  \left( \frac{|y|^2}{s^2} |q| +   s^{-1} |q|^2  + |\nabla q|^2  \right), \label{estima-T-Q-inside-blowup}\\
\left|   J(q, \theta (s))  \right| &\leq & C (K_0, A)  \left( \frac{|q|}{s} + \frac{|\nabla q|}{\sqrt s} \right),  \label{estima-chi-T-Q}
\end{eqnarray}
where  $q$ is a transformed function of $U$ given in \eqref{defini-q} and $J (q, \theta (s) )$ is  defined in     \eqref{defini-T-Q}. 

\noindent
In particular,  for all   $y \in  \mathbb{R}^n$, we have
\begin{equation}\label{T-Q-forall-y}
\left| (1 - \chi(y,s)) T(q, \theta (s))  \right| \leq  C (K_0, C_0) \min \left( \frac{1}{s},  \frac{|y|^3}{s^{\frac{5}{2}}}   \right).
\end{equation}
\end{lemma}
\begin{proof}
The techniques of the proof of estimates   \eqref{estima-T-Q-inside-blowup}, \eqref{estima-chi-T-Q} and \eqref{T-Q-forall-y} are the same.     Although, function $J(q, \theta)$ is  our work  has  some  differences from   the work of Merle and Zaag in \cite{MZnon97},  we assert that the proof  still holds with    our  model. In order to show this argument,   we  kindly ask to  refer the reader to check   Lemma B.4 in   that work. For that reason, we  only give the proof of  \eqref{estima-T-Q-inside-blowup} and  \eqref{estima-chi-T-Q} here, and we leave  the proof  of \eqref{T-Q-forall-y} for  the reader to be done similarly as for comparison     Lemma B.4 in \cite{MZnon97}. We now consider $|y| \leq  2 K_0 \sqrt{s},$  and   introduce   $G(h )  = - 2\frac{\left|   \nabla \varphi + h \nabla q \right|^2}{  \varphi  + \frac{\lambda^{\frac{1}{3}} e^{-\frac{s}{3}}}{\theta(s)}  + h q   }   +  2 \frac{|\nabla  \varphi|^2}{  \varphi + \frac{\lambda^{\frac{1}{3}} e^{-\frac{s}{3}}}{\theta(s)} }, h \in [0,1]$. Then, we have  the following: 
\begin{eqnarray*}
G'_h (h)   &=&   \frac{2 q \left| \nabla \varphi + h \nabla  q \right|^2}{\left( \varphi + \frac{\lambda^{\frac{1}{3}} e^{-\frac{s}{3}}}{\theta (s) } + h q \right)^2} - 4 \frac{\nabla q (\nabla \varphi  +   h \nabla q)}{  \varphi + \frac{\lambda^{\frac{1}{3}} e^{-\frac{s}{3}}}{\theta (s) } + h q}  ,\\
G''_h(h) & = & - 4q^2 \frac{\left| \nabla \varphi + h \nabla  q \right|^2}{\left( \varphi + \frac{\lambda^{\frac{1}{3}} e^{-\frac{s}{3}}}{\theta (s) } + h q \right)^3} +  8 q \frac{\nabla q (\nabla \varphi  +   h \nabla q)}{ \left( \varphi + \frac{\lambda^{\frac{1}{3}} e^{-\frac{s}{3}}}{\theta (s) } + h q \right)^2} - 4 \frac{\left| \nabla q    \right|^2}{  \varphi  +\frac{\lambda^{\frac{1}{3}} e^{-\frac{s}{3}}}{\theta (s) } + h q  }.
\end{eqnarray*} 
Using a   Taylor expansion of  $G(h)$ on $[0,1]$,   at  $h = 0$,  we get the following:
$$ G(1) = G(0)  + G'(0)  + \int_0^1  (1 - h) G'' (h) dh.$$
Using the following facts
$$ G(1) = J(q, \theta (s)), G(0) = 0,$$
we  write the  following
\begin{eqnarray*}
 T(q, \theta (s)) & = &   \left( \frac{ 2 q|\nabla \varphi|^2}{ \left( \varphi + \frac{\lambda^{\frac{1}{3}} e^{-\frac{s}{3}}}{\theta (s) }  \right)^2} -\frac{ 4 \nabla \varphi \cdot \nabla q}{  \varphi + \frac{\lambda^{\frac{1}{3}} e^{-\frac{s}{3}}}{\theta (s) } } \right) + \int_0^1 (1 - h )  G''_h (h) dh
\end{eqnarray*}
From  the definition of $\varphi$ given in \eqref{defini-varphi}, we can derive  that for all $s \geq 1$ and  $y  \in \mathbb{R}^n$, we have
$$ \frac{  \left| \nabla \varphi (y,s) \right|^2 }{  \varphi^2 (y,s)  }  \leq  C   \frac{|y|^2}{s} \text{ and }   \left|  \nabla \varphi(y,s) \right|  \leq C s^{- \frac{1}{2}} .$$
In addition to that, using    Lemma  \ref{lemma-properties-V-A-s}, we can prove that there exists $s_9 (A, K_0) $ such that  for all $s \geq s_0 \geq  s_9$,   $ h \in [0,1]$ and  $|y| \leq 2K_0 \sqrt{s}$, we have the following
$$ \left|  F'' (h) (y,s)   \right|     \leq  C(A, K_0) \left(   \frac{|q|^2}{s} + \left| \nabla q \right|^2 \right) \leq  C (A, K_0) \left( \frac{|q|}{s} + \frac{|\nabla q |}{ \sqrt{s}}   \right) . $$ Thus,  \eqref{estima-T-Q-inside-blowup} and \eqref{estima-chi-T-Q} follow.
\end{proof}

\medskip
We now  aim at  giving  some estimates on  $R$. The following is our statement: 
\begin{lemma}[Bounds on $R$]\label{lemma-Bound-R} Let us consider  $R$ defined  in \eqref{defini-rest-term}. We  assume that  $\theta (s) \geq 1, $ for all $ s \geq 1$. Then,  for all  $s \geq 1$ and $y \in  \mathbb{R}^n $, the following  holds:  
$$   \left|   R(y, s)   - \frac{ c_1}{ s^2}    \right|  \leq  C \frac{(1 + |y|^3)}{s^3},$$
and  
$$    \left|  \nabla R (y, s)\right|   \leq C \frac{(1 + |y|^3)}{s^3}.$$
In particular,  
$$ \|R(., s)\|_{L^\infty(\mathbb{R})} \leq  \frac{C}{s}.$$
\end{lemma}  
\begin{proof}
 The function $R$, in our work   is different  from the definiton  in \cite{MZnon97} (up to a very small difference). Hence, the  proof  of \cite{MZnon97} holds  in our case with minor adaptation. Accordingly,      we kindly refer  the reader to check       Lemma B.5 page 1541  in that work.  
\end{proof}  
We now  give some  estimates   on  $N$. The control of this term is  a new contributation of our study. In addition to that, it is  a direct consequence  of Proposition \ref{propo-bar-mu-bounded} on the control $\bar \theta (t)$. The following is  our statement:
\begin{lemma}[Bound  on  $N(q, \theta (s))$]\label{lemma-Bound-N-Q} There exists  $K_{10} > 0$ such that  for all  $K _0 \geq K_{10},  A > 0 $ and  $\delta_0 \leq  \frac{1}{2} \left(3 + \frac{9}{8}   \frac{K_0^2}{16} \right)^{-\frac{1}{3}}$,   there exist  $\alpha_{10} (K_0, \delta_{0}) > 0$ and $C_{10} (K_0) > 0$ such that for  every  $\alpha_0 \in   (0, \alpha_{10}]$     we can find  $\epsilon_{10} (K_0, \delta_{0}, \alpha_0) > 0$ such that   for every  $\alpha_0 \in (0,\epsilon_{10}], \eta_0 \leq 1$,  there exists $T_{10}  (K_0) > 0$ such that   all for all  $T \leq T_{10} $, the following holds: Assume that    $U $ is  a nonnegative  solution  of   equation \eqref{equa-U}    on $[0, t_{10}]$ for some  $t_{10}  \leq T_{10}$, and    initial data $U(0) = U_{d_0, d_1}$ given in  \eqref{defini-initial-data}  for some  $(d_0, d_1)\in \mathbb{R} \times  \mathbb{R}^n,$ satisfying   $|d_0|, |d_1| \leq  2 $,   and     $U \in S(T,K_0, \epsilon_0,    \alpha_0, A, \delta_0, C_0,\eta_0,  t) $ for all  $t \in [0, t_{10}]$.  Then,  for all $s = - \ln(T-t) $ with $t  \in  [0, t_{10}]$,  the following estimate holds:
$$ \| N(q, \theta (s))\|_{L^\infty(\mathbb{R}^n)}  \leq  \frac{1}{s^{2019}},$$
where  $N (q, \theta (s))$ is defined  in \eqref{defini-N-term}.
\end{lemma}
\begin{proof}
 Using the  fact that  $U$ is in  $S(t)$,  item $(i)$ in Definition \ref{defini-shrinking-set-S-t}  and  item $(i)$ of Lemma  \ref{lemma-properties-V-A-s},  we derive  that        $$  \| (q + \varphi) (.,s)\|_{L^\infty (\mathbb{R}^n)} \leq  C.  $$ 
 Hence,  it is enough to  find   a  bound  on  the  following quanlity
 $$ \frac{\theta'(s)}{\theta (s)}.$$ 
 (see in  definition  \eqref{defini-N-term}).   As a matter of fact,  using  Proposition \ref{propo-bar-mu-bounded},    it is clear   to     have the following
\begin{eqnarray*}
\left|  \frac{\theta' (s)}{ \theta (s)}   \right|  = \left| \frac{\bar{ \theta }' (t)}{ \bar{\theta}(t)} \right| \left| \frac{dt }{ds} \right|  \leq   C  e^{\frac{8 - 3n - 6}{6} s} |s^n|.
\end{eqnarray*} 
Hence,  there exists $s_{10}  $ large enough such that for all $s \geq s_0 \geq    s_{10}$,  we can write
$$  \| N(q, \theta (s))  \|_{L^\infty(\mathbb{R}^n)}  \leq   C e^{-\frac{s}{6}}|s|^n  \leq  \frac{1}{s^{2019}},$$
which yields the conclusion of the proof.
\end{proof}

\medskip
Finally,  we give  a  bound on   $ F(w,W)$. As a matter of fact, this is  an important bridge  that connects the problems   in  $\mathbb{R}^n$  and  in  a bounded domain. In other words, it is created by the  localization  around  blowup region.   Fortunately,   this term is controled as a  small perturbation in our analysis. More precisely, the following is  our statement:
\begin{lemma}[Bound  on $F(w,W) $]\label{lemma-F-y-s}   Let us consider $F(w,W)$, defiend   in \eqref{defini-F-1}. Then,  there exists $\epsilon_{11} > 0$  such that  $ K_0 > 0, \epsilon_0  \leq  \epsilon_{11}, \alpha_0 > 0,  A > 0, \delta_0  > 0,  C_0 > 0 , \eta_0>0$,   there exists $T_{11} > 0$ such that  for all  $T \leq  T_{11}$,  the following holds:     Assuming that  $U \in S(T, K_0, \epsilon_0, \alpha_0,A, \delta_0, C_0, \eta_0,     t),$ for all $t \in  [0, t_{11}],$ for some  $t_{11} \in [0,T)$, then,   we have
\begin{eqnarray*}
\left\|  F(w,W)    \right\|_{L^\infty(\mathbb{R}^n)} \leq  \frac{1}{s^{2019}},
\end{eqnarray*}
where  $s = - \ln (T-t).$
\end{lemma}
\begin{proof}
From the definition of $F,$ it is enough to   consider   $  |y| \in \left[ \frac{e^{\frac{s}{2}}}{M_0} e^{\frac{s}{2}}, \frac{ 2 e^{\frac{s}{2}}}{M_0}  \right ]$. We now  take $\epsilon_0 \leq \frac{1}{2M_0},$ then, this  domain corresponds  to  the region  $P_3 (t)$ where our solution  $U$ is regared as  a perturbation of  initial data. Using  the fac that  $U$  is in $ S(t),$ then, we can derive from   item $(iii)$  in Definition  \ref{defini-shrinking-set-S-t} that 
\begin{eqnarray*}
\left| W(y,s) \right| &\leq &C(K_0) e^{- \frac{s}{3}}, \\
 \left| \nabla_y W(y,s) \right| & \leq  &  C(K_0)   e^{- \frac{s}{3}}. 
\end{eqnarray*}
     In addition to that, from definition   \eqref{defini-w-small}, we  deduce  that
\begin{eqnarray*}
\left| w (y,s) \right|  &  \leq  & C (K_0, M_0) e^{-\frac{s}{3}}, \\
\left|   \nabla  w (y,s) \right| &  \leq &  C(K_0, M_0) e^{- \frac{s}{3}}.  
\end{eqnarray*}     
On the other hand,  using   the definition   of    $\psi_{M_0}$ given in  \eqref{defini-psi-M-0-cut},    we get the following
$$   \left|  \partial_s \psi_{M_0} -  \Delta \psi_{M_0} + \frac{1}{2} y \cdot      \nabla \psi_{M_0} \right| \leq  C(M_0).$$
   In fact, using  the above estimate, we can get the conclusion  if  $s \geq s_0 (K_0) $. 
\end{proof}

\section{The Dirichlet heat  semi-group  on $\Omega$ }\label{appex-semi-group}
In this section, we aim at giving some  main properties  of    the Dirichlet heat  semi-group  $ \left( e^{t \Delta} \right)_{t > 0} $ (see more details in \cite{QSbook07} or  chapter 16 in  \cite{LSUAMS68}).  In particular, we prove  the parabolic regularity estimate of Lemma  \ref{lemma-regular)linear}.   We    consider   the following equation
\begin{equation}\label{equa-linear}
\left\{ \begin{array}{rcl}
\partial_t U -  \Delta U &= & 0  \text{ in } \Omega \times (0,T),\\ 
U & = &  0 \text{ in }  \partial \Omega \times (0,T),\\
U (x,0)  & = &  U_0(x)  \text{ in } \bar{\Omega}.  
\end{array}
\right.
\end{equation}
 In particular,   one can prove that there exists  $G(x, y, t, \tau), t \geq \tau$  nonnegative, symemtric in $x,y,$ i.e $G(x,y,t,\tau) = G(y,x, t,\tau)$ and  defined  in $\Omega \times \Omega \times (0,T) \times [0,T)$  with the following condition
\begin{equation}
\left\{ 
\begin{array}{rcl}
(\partial_t - \Delta ) G(x,y, t , \tau) & = & \delta (x -y) \delta (t - \tau),\\
G (x,y,\tau,\tau) &  = &  0 \text{ and }  G(x,y, t, \tau)= 0 \text{ if } x \in \partial  \Omega.
\end{array}
\right.
\end{equation}
Moreover, for all  $ f \in L^\infty (\Omega),$ we have
\begin{equation}\label{defi-semi-group}
(e^{t \Delta} f ) (x) = \int_{\Omega} G(x,y, t, 0) f(y) dy.
\end{equation}
Hence, we can write  the solution  of  equation \eqref{equa-linear}  as follows
$$  U (t) =  e^{t \Delta } (U_0) .$$
We now   consider furthermore  the following  non-homogeneous equation
\begin{equation}\label{equa--non-linear}
\left\{ \begin{array}{rcl}
\partial_t U -  \Delta U &= & F  \text{ in } \Omega \times (0,T),\\ 
U & = &  0 \text{ in }  \partial \Omega \times (0,T),\\
U (x,0)  & = &  U_0(x)  \text{ in } \bar{\Omega}.  
\end{array}
\right.
\end{equation}
 If  $F \in C(\Omega \times (0,T)), u_0 \in  C(\Omega)$  and $\Omega$  is $C^2$, bounded domain in  $\mathbb{R}^n$. Then,  we can prove that there  locally exists a classical solution of  problem \eqref{equa--non-linear}.   Then, by using Duhamel principal,  the solution satisfies the following  integral equation
$$   U(t) = e^{t \Delta } (U_0)  + \int_0^t e^{(t-s) \Delta} F(s) ds.$$
Sometimes,  we  also call $G(x,y,t,\tau)$  the Green function. Let us  give  in the following  the  main properties  of the Green function:
\begin{lemma}\label{Green-function} Let us  consider  the Green function  called  $G(x,y, t, \tau)$ above. Then, the  following holds:  for all $ (x, y, t,\tau) \in \Omega \times \Omega \times (0,T) \times [0,T)$ and  integer numbers $r,s$, we have 
$$\left|   \partial_t^r \partial_{x_1^{s_1} ...x_n^{s_n}}^s G(x,y,t,\tau)  \right| \leq C (t - \tau)^{-\frac{n + 2r + s}{2}}  \exp \left( - c(\Omega)\frac{|x - y|^2}{ t - \tau}\right).$$
\end{lemma}
\begin{proof}
We kindly refer the reader to see Theorem 16.3, page 413 in \cite{LSUAMS68}. 
\end{proof}
We now prove in the following   Lemma  \ref{lemma-regular)linear}

\begin{proof}[The proof of   Lemma  \ref{lemma-regular)linear}]
From the defintion  of  the semigroup $e^{t \Delta}$,   it is  easy to derive that  $L(t)  \in C(\bar \Omega \times [0,T])   \cap  C^\infty ( \Omega \times (0,T] )$.   Hence, it is  enough to give  the  proof of  \eqref{norm-nabla-e-t-delta-U-0}. Indeed,   we first  derive the support of  $U_{d_0, d_1} = \{  |x|   \leq  \frac{1}{2} d(0, \partial \Omega)\}$.      We now consider  two following  regions:
\begin{eqnarray*}
\Omega_1  &=&  \{ \frac{\epsilon_0}{8} \leq |x| \leq \frac{7}{8} d (0, \partial \Omega)    \},\\
\Omega_2 & =& \{  |x|  >  \frac{3}{4} d(0, \partial \Omega)  \} \cap \Omega.
\end{eqnarray*}
In addition to that, we can write $L_1 (t)$ as  follows
\begin{equation}\label{repr-L}
L (x,t) =  \int_{\Omega} G (x,y,t,0) U_{d_0,d_1} (y) dy = \int_{ \{|y| \leq  \frac{1}{2} d(0, \partial \Omega) \}} G_{\Omega} (x,y,t,0) U_{d_0,d_1} (y) dy,
\end{equation}
which yields
\begin{equation}\label{repr-nabla-L}
\nabla L(x,t)  = \int_{ \{| y| \leq \frac{1}{2} d(0, \partial \Omega)  \} } \nabla_x G(x,y,t,0) U_{d_0, d_1} (y) dy.
\end{equation}
- We  consider the  case where   $x \in  \Omega_2:$ Thanks to  Lemma \ref{Green-function} and \eqref{repr-nabla-L}, we have 
\begin{eqnarray*}
 \left| \nabla L (x,t) \right|  & \leq &   \int_{ \{| y| \leq \frac{1}{2} d(0, \partial \Omega)  \} } |\nabla_x G(x,y,t,0) | |U_{d_0, d_1} (y) |dy\\
&  \leq  &  \int_{ \{| y| \leq \frac{1}{2} d(0, \partial \Omega)  \} }   \frac{C \exp \left( - c_\Omega \frac{|x - y|^2}{t} \right)}{t^{\frac{n+1}{2}}} | U_{ d_0, d_1} (y)| dy \\
&  \leq   &C  \int_{ \{| y| \leq \frac{1}{2} d(0, \partial \Omega)  \} }  \exp \left( - c_\Omega \frac{|x - y|^2}{t} \right)  \frac{|x-y|^{n+1}}{t^{\frac{n+1}{2}}}  \frac{|U_{ d_0, d_1} (y)|}{ |x -y|^{n+1}}   dy\\
& \leq &  C \int_{ \{| y| \leq \frac{1}{2} d(0, \partial \Omega)  \} }   \frac{|U_{ d_0, d_1} (y)|}{ |x -y|^{n+1}}   dy
\end{eqnarray*}
Because  $x \in   \Omega_2, $ we have the following fact 
$$    \frac{1}{ |x -y|^{n+1}}  \leq   C.   $$
   This  yields  the following 
   $$\left| \nabla L (x,t)  \right|  \leq  C  \int_{ \{| y| \leq \frac{1}{4} d(0, \partial \Omega)\} }   \left| U_{d_0, d_1} (y)\right| dy.   $$
 In addition to that,   using   \eqref{defini-initial-data},  we   have  the following 
\begin{eqnarray*}
 & &\int_{|y| \leq \frac{1}{2} d (0, \partial \Omega)} |U_{d_0, d_1} (y)| dy =\int_{|y| \leq  2 \sqrt{T}|\ln T|} |U_{d_0, d_1} (y)| dy + \int_{2 \sqrt{T}|\ln T| \leq   |y| \leq  \frac{1}{2} d (0, \partial \Omega)} |U_{d_0, d_1} (y)| dy \\
&  =  & \int_{|y| \leq  2 \sqrt{T}|\ln T|}   T^{-\frac{1}{3}}  \left| \varphi \left( \frac{y}{\sqrt{T}}, -\ln s_0  \right)   +  (d_0 + d_1 \cdot  \frac{y}{\sqrt{T |\ln T|}})  \chi_0 \left(  \frac{|y|}{ \sqrt{T|\ln T|}\frac{K_0}{32}}\right) \right|  \chi_1 (y) dy\\
& + &   \int_{2 \sqrt{T}|\ln T| \leq   |y| \leq  \frac{1}{2} d (0, \partial \Omega)} |  (1 - \chi_1 (y)) H^* (y)| dy \leq C,
\end{eqnarray*} 
 which yields 
 \begin{equation}\label{estima-nabla-L-Omega-2}
 |\nabla L (x,t)| \leq  C, \text{ for  all  in  } \Omega_2,
 \end{equation}
 It is similar to  prove  the following estimate
 \begin{equation}\label{estima-nabla-2-L-Omega-2}
  |\nabla^2 L (x,t)| \leq  C, \text{ for  all  in  } \Omega_2.
 \end{equation}
 
 - We consider  the case where $x \in \Omega_1$: Let us define  $\phi (x)$ as a function in   $C^\infty_0 \left(   \mathbb{R}^n\right)$ and  satisfying the  following conditions
 \begin{eqnarray*}
 \phi (x)  & = & 0 \text{ if }    |x| \geq  \frac{11}{12} d (0, \partial \Omega), \\
 \phi (x) &  = & 1 \text{ if }          |x| \leq \frac{7}{8} d (0, \partial \Omega).
 \end{eqnarray*}
Then, we also introduce the following function
$$ L_1 (x,t)  = \phi (x)    \nabla L (x,t).$$
We  now  write an equation satisfied by  $L_1$
\begin{equation}\label{equa-L-1}
\left\{ \begin{array}{rcl}
\partial_t L_1 -  \Delta L_1 &= &   - 2 \nabla \phi \cdot \nabla^2 L - \Delta \phi \nabla L   \text{ in } \Omega \times (0,T),\\ 
L_1 & = &  0 \text{ in }  \partial \Omega \times (0,T),\\
L_1 (x,0)  & = &   \phi \nabla L(0) = \phi \nabla  U_{d_0, d_1} \text{ in } \bar{\Omega}.  
\end{array}
\right.
\end{equation}
Using Duhamel's formula, we get
\begin{equation}\label{Duhamel-prin-L-1}
L_1(t)  = e^{t \Delta} L_1 (0) + \int_0^t e^{(t-s) \Delta} \left[  - 2 \nabla \phi \cdot \nabla^2 L - \Delta \phi \nabla L \right] (s) ds.
\end{equation}
We now aim at proving  the following fact
\begin{eqnarray}
 \| e^{(t-s)\Delta }  ( \Delta \phi \nabla L )(s) \|_{L^\infty (\Omega)} &\leq &  C \|L_1(s)\|_{L^\infty(\Omega)} +C,\label{estima-Delta-phi-nabla-L}, \\
  \| e^{(t-s)\Delta }  (\nabla \phi \cdot \nabla^2 L )(s) \|_{L^\infty (\Omega)} & \leq &  \frac{C \|L_1(s)\|_{L^\infty(\Omega)}}{\sqrt{t-s}} + C \left(1 + \frac{1}{\sqrt{t-s}}\right),\label{estima-nabla-phi-nabla-2-L}.
\end{eqnarray}

\textit{- The proof of  \eqref{estima-Delta-phi-nabla-L}:} We have the following fact
\begin{eqnarray*}
| \Delta \phi \nabla L| &=&  |I_{\{|x| \leq \frac{7}{8} d(0,\partial \Omega) \}}\Delta \phi \nabla L | +  |I_{\{|x| > \frac{7}{8} d(0,\partial \Omega) \}}\Delta \phi \nabla L |\\
& \leq & C |\phi \nabla L| + C = C |L_1| + C.
\end{eqnarray*}
Then, by using  the monotonicity of the operator $e^{(t-s) \Delta},$ we derive directly   \eqref{estima-Delta-phi-nabla-L}.

\textit{ - The proof of  \eqref{estima-nabla-phi-nabla-2-L}:}  From the   definition  of operator  $e^{(t-s) \Delta},$ we can write the following
\begin{eqnarray*}
e^{(t-s) \Delta} (\nabla \phi \cdot \nabla^2 L (s)) = \int_{\Omega} G (x,y,t,s) \nabla \phi (y) \cdot \nabla^2 L (y,s) dy.
\end{eqnarray*}
We consider   $j \in \{1,...,n\},$  and  integrate by part, we get the following
\begin{eqnarray*}
\int_{\Omega}  \sum_{i=1}^n G(x,y,t,s)\partial_{y_i} \phi (y) \partial^2_{y_i y_j} L dy  & = & - \int_{\Omega} (\nabla_y G(x,y,t,s) \cdot \nabla \phi + G(x,y,t,s) \Delta (y)) \partial_{y_j} L (y,s)dy\\
& =&   -\int_{\Omega} \nabla_y G(x,y,t,s) \cdot \nabla \phi  \partial_{y_j} L (y,s)dy \\
&  - &  \int_{\Omega}  G(x,y,t,s) \Delta (y)\partial_{y_j} L (y,s)dy.
\end{eqnarray*}
Using the defintion of $\phi$ in the above and \eqref{estima-nabla-L-Omega-2},  we have the following fact:
\begin{eqnarray*}
|\nabla L|   &  =  & |I_{\{|x| \leq  \frac{7}{8} d (0, \partial \Omega) \}} \nabla L | +|  I_{\{|x| > \frac{7}{8} d (0, \partial \Omega) \}}\nabla L |\\
& = &  |  I_{\{|x| \leq  \frac{7}{8} d (0, \partial \Omega) \}} \phi (x) \nabla L|   +    |  I_{\{|x| > \frac{7}{8} d (0, \partial \Omega) \}}\nabla L |\\
& \leq & |L_1| + C. 
\end{eqnarray*}
Then, 
\begin{eqnarray*}
\left| \int_{\Omega}  \sum_{i=1}^n G(x,y,t,s)\partial_{y_i} \phi (y) \partial^2_{y_i y_j} L (y) dy \right|  & \leq &  (\|L_1(s)\|_{L^\infty (\Omega)} + C)  \left|\int_{\Omega} \nabla_y G(x,y,t,s) \cdot \nabla \phi dy \right| \\
& +  & (\|L_1(s)\|_{L^\infty (\Omega)} + C) \left|\int_{\Omega}  G(x,y,t,s) \Delta \phi dy \right| \\
& \leq & (\|L_1(s)\|_{L^\infty (\Omega)} + C) \left[  \frac{C}{ \sqrt{t-s}} + C\right], 
\end{eqnarray*}
which  implies  \eqref{estima-nabla-phi-nabla-2-L}.
We now use  \eqref{Duhamel-prin-L-1}, \eqref{estima-Delta-phi-nabla-L} and  \eqref{estima-nabla-phi-nabla-2-L} to  deduce the following
\begin{equation}\label{Growall-L-1-t-infty}
\| L_1(t)\|_{L^\infty} \leq C \|\nabla U_{d_0,d_1}\|_{L^1 (\Omega)} + \int_{0}^t \left[ C\left(1 + \frac{1}{\sqrt{t-s}} \right) \|L_1 (s)\|_{L^\infty}   +   C\left(1 + \frac{1}{\sqrt{t-s}} \right)\right] ds.
\end{equation}
Using Gronwall's lemma, we obtain the following estimate
$$  \| L_1(t)\|_{L^\infty}  \leq C  \|\nabla U_{d_0,d_1}\|_{L^1 (\Omega)}.$$
We admit the following fact which we will be proved at the end:
\begin{equation}\label{norm-L-1-of-U-d-0-d-1}
 \| \nabla U_{d_0,d_1}\|_{L^{1} (\Omega)} \leq  CT^{- \frac{1}{2}} + C(\epsilon_0).
\end{equation}
This estimate gives a rough estimation on $L_1$ as follows
\begin{equation}\label{estimate-rough-L-1}
\|L_1 (t)\|_{L^\infty (\Omega)} \leq C T^{-\frac{1}{2}} + C(\epsilon_0).
\end{equation}
Let us  improve  this estimate.  We come back to identity \eqref{Duhamel-prin-L-1} and consider  the set of all   $x \in \Omega$ such that $|x| \geq \frac{\epsilon_0}{8}$. By using  the definition of  $U_{d_0,d_1}$ in \eqref{defini-initial-data}, we first  prove the following fact
\begin{equation}\label{fact-norm_infty-U-d_0-epsilon-0}
\|e^{t \Delta } \left( \nabla U_{d_0,d_1}\right)\|_{L^\infty (|x| \geq \frac{\epsilon_0}{8}, x \in \Omega
)} \leq C(\epsilon_0).
\end{equation}
Indeed,  we write  $e^{t \Delta } \left( \nabla U_{d_0,d_1}\right)$ as follows
\begin{eqnarray*}
e^{t \Delta } \left( \nabla U_{d_0,d_1}\right) &=& \int_{\Omega } G(x,y,t,0) \nabla_y U_{d_0,d_1} (y)dy = \int_{|y| \leq  \frac{1}{2} d (0, \partial \Omega)} G(x,y,t,0) \nabla_y U_{d_0,d_1} (y)dy\\
& = &\int_{|y| \leq  \frac{\epsilon_0}{16}} G(x,y,t,0) \nabla_y U_{d_0,d_1} (y)dy + \int_{  \frac{\epsilon_0}{16} \leq  |y| \leq \frac{1}{2} d(0, \partial \Omega)} G(x,y,t,0) \nabla_y U_{d_0,d_1} (y)dy\\
& =& I_1 + I_2.
\end{eqnarray*}

+ Bound on  $I_1$:  Using integration  by parts, we get the following:
 \begin{eqnarray*}
I_1 =  - \int_{ |y| \leq \frac{\epsilon_0}{16}} \nabla_y G(x,y,t,0) U_{d_0,d_1} (y) dy + \int_{|y| = \frac{\epsilon_0}{16}} G(x,y,t,0) U_{d_0,d_1} (y) \eta (y) dS.
\end{eqnarray*}
From  Lemma \ref{Green-function}, we  derive that 
\begin{eqnarray*}
|I_1 (x,t)| &\leq &  \int_{|y| \leq \frac{\epsilon_0}{16} }  \frac{\exp
\left( - c_{\Omega}\frac{ |x - y|^2}{t}\right)}{t^{\frac{n+1}{2}}}  |U_{d_0,d_1}(y)| dy   + C(\epsilon_0) \\
& \leq &  \int_{|y| \leq \frac{\epsilon_0}{16} }  \exp
\left( - c_{\Omega}\frac{ |x - y|^2}{t}\right)  \frac{|x-y|^{n+1}}{t^{\frac{n+1}{2}}} \frac{1}{|x-y|^{n+1}} |U_{d_0,d_1}(y)| dy + C(\epsilon_0)\\
& \leq & C(\epsilon_0) \|U_{d_0,d_1}\|_{L^1 (\Omega)} + C(\epsilon_0) \leq C_1(\epsilon_0)
\end{eqnarray*}

+ Bound on  $I_2$: It is easy to prove that 
$$ \|\nabla U_{d_0,d_1}(.)\|_{L^{\infty} ( \frac{\epsilon_0}{16} \leq |y| \leq \frac{1}{2} d(0,\partial \Omega) )} \leq  C(\epsilon_0).$$
This  yields directly that 
$$ |I_2 (x,t)|  \leq  C(\epsilon_0) \int_{\frac{\epsilon_0}{16} \leq |y| \leq \frac{1}{2} d(0,\partial \Omega) } G(x,y,t,0) dy \leq C(\epsilon_0). $$
Hence, we get the conclusion the  proof of   \eqref{fact-norm_infty-U-d_0-epsilon-0}.  Using  \eqref{Growall-L-1-t-infty}, \eqref{estimate-rough-L-1} and  \eqref{fact-norm_infty-U-d_0-epsilon-0}, we get the following: for all $|x| \geq \frac{\epsilon_0}{8}, x \in \Omega$
$$ |L_1 (x,t)| \leq  C(\epsilon_0) + C \int_{0}^t \left( 1 + \frac{1}{\sqrt{t-s}}\right) T^{- \frac{1}{2}}  ds \leq   C (\epsilon_0),$$
provided that $T < 1$. This yields that for all $x \in \Omega_1$
\begin{equation}\label{estima-nabla-L-Omega-1}
|\nabla L (x,t)| \leq  C(\epsilon_0).
\end{equation}
Finally, \eqref{norm-nabla-e-t-delta-U-0} follows  from \eqref{estima-nabla-L-Omega-2} and \eqref{estima-nabla-L-Omega-1},  which will  conclude the proof of Lemma    \ref{lemma-regular)linear}. However, in order to finish  the  proof we need to prove  \eqref{norm-L-1-of-U-d-0-d-1}: Indeed,  from the  definition  of $U_{d_0, d_1}$ given    in \eqref{defini-initial-data}, we  write  
\begin{eqnarray*}
  \nabla_x U_{d_0,d_1}(x) =   I_1(x) + I_2 (x) + I_3 (x) + I_4(x),
\end{eqnarray*}
where
\begin{eqnarray*}
I_1 & = & T^{-\frac{1}{3}} \left[ - \frac{3}{4} \left( 3 +  \frac{9}{8} \frac{|x|^2}{T |\ln T|}\right)^{-\frac{4}{3}} \frac{x}{ T |\ln T|}  +  \frac{d_1}{\sqrt{T |\ln T|} } \chi_0 \left( \frac{x}{ \sqrt{T |\ln T|}} \right) \right.  \\
& + & \left. \frac{d_1 \cdot x}{ \sqrt{T|\ln T|}} \chi_0' \left( \frac{x}{ \sqrt{T |\ln T|}} \right) \frac{x}{|x|\sqrt{T|\ln T|}}  \right] \chi_0 \left(   \frac{|x|}{ \sqrt{T} |\ln T|} \right),\\
I_2 & =& T^{-\frac{1}{3}} \left[  \left( 3 +  \frac{9}{8} \frac{|x|^2}{T |\ln T|}\right)^{-\frac{1}{3}}    + \left(  d_0 + \frac{d_1 \cdot x}{\sqrt{T|\ln T|}}\right)\chi_0 \left(   \frac{|x|}{ \sqrt{T |\ln T|}} \right) \right]\\
& \times &\chi_0 '\left(   \frac{|x|}{ \sqrt{T} |\ln T|} \right) \frac{x}{|x| \sqrt{T} |\ln T|},\\
I_3 & =&  \left( 1 - \chi_0 \left(  \frac{x}{ \sqrt{T} |\ln T|}\right)\right) \nabla H^* (x) ,\\
I_4 & =&  - \chi_0' \left(  \frac{x}{ \sqrt{T} |\ln T|}\right) \frac{x}{|x| \sqrt{T}|\ln T|}  H^* (x).
\end{eqnarray*}
As a matter of fact, we have the following
$$   \|\nabla U_{d_0,d_1} \|_{L_1}  \leq \int_{\Omega} |I_1 (x)| dx + \int_{\Omega} |I_2(x)| dx  + \int_{\Omega} |I_3(x)| dx + \int_{\Omega} |I_4 (x)| dx.$$

In particular, we have
\begin{eqnarray*}
\text{Supp} (I_1) &\subset & \{ |x| \leq 2 \sqrt{T} |\ln T| \},\\
\text{ Supp}(I_2) & \subset & \{ \sqrt{T} |\ln T| \leq    |x| \leq  2 \sqrt{T} |\ln T| \},\\
\text{ Supp} (I_3) & \subset & \{ \sqrt{T} |\ln T| \leq   |x| \leq \frac{1}{2} d(0,\partial \Omega)\}, \\
  \text{Supp}(I_4) &  \subset &  \{ \sqrt{T} |\ln T| \leq  |x| \leq  2 \sqrt{T} |\ln T| \}.
\end{eqnarray*}
By some simple  upper bounds on $I_1$ and $I_2$, we can  derive that
$$ \int_{\Omega} |I_1 (x)| dx  \leq CT^{-\frac{1}{2}} + C \text{ and }  \int_{\Omega} |I_2(x)| dx \leq  CT^{-\frac{1}{2}} + C. $$

We now aim at estimating  $I_3$ and $I_4$.

+ \text{ Estimate on   $I_3$:} We write as follows
\begin{eqnarray*}
\int_{\Omega} |I_3| (x) dx  & = & \int_{ \sqrt{T} |\ln T| \leq  |x| \leq \min \left(\frac{1}{2}, \frac{1}{4} d(0,\partial \Omega)\right)} |I_3 (x)| dx + \int_{   \min \left(\frac{1}{2}, \frac{1}{4} d(0,\partial \Omega)\right) \leq | x| \leq  \frac{1}{2} d(0,\partial \Omega)}  |I_3(x)|dx\\
& \leq &   \int_{ \sqrt{T} |\ln T| \leq  |x| \leq \min \left(\frac{1}{2}, \frac{1}{4} d(0,\partial \Omega)\right)}  |I_3 (x)| dx  +C.
\end{eqnarray*}
In addition to that, 
\begin{eqnarray*}
 \int_{ \sqrt{T} |\ln T| \leq  |x| \leq \min \left(\frac{1}{2}, \frac{1}{4} d(0,\partial \Omega)\right)}  |I_3 (x)| dx &\leq &  C \int_{ \sqrt{T} |\ln T| \leq  |x| \leq \min \left(\frac{1}{2}, \frac{1}{4} d(0,\partial \Omega)\right)}  |x|^{-\frac{4}{3}} |\ln |x||^{\frac{1}{3}} dx\\
 & \leq &  CT^{-\frac{1}{2}} + C.
\end{eqnarray*}
This implies that 
$$ \int_{\Omega} |I_3 (x)| dx  \leq CT^{-\frac{1}{2}} + C.$$

+ \text{ Estimate on $I_4$:} We have
\begin{eqnarray*}
\int_{\sqrt{T} |\ln T| \leq  |x| \leq  2 \sqrt{T} |\ln T|  } |I_4(x)| dx \leq  \frac{C}{\sqrt{T |\ln T|}} \int_{ \sqrt{T} |\ln T| \leq  |x| \leq  2 \sqrt{T} |\ln T|} |\ln |x||^{\frac{1}{3}} |x|^{-\frac{2}{3}} dx \leq  CT^{- \frac{1}{2}}.
\end{eqnarray*}
From the above estimates, we  can conclude \eqref{norm-L-1-of-U-d-0-d-1}.  We also finish the proof of Lemma \ref{lemma-regular)linear} .
\end{proof}

\section{Some Parabolic estimates}

In this section, we aim at  giving some estimates   on $U, \nabla U, \nabla^2 U$. More precisely, the following is our statement:
\begin{lemma}[Parabolic estimates on $U$]\label{lemma-parabolic-estimates} We consider  $U$ a solution  to equation \eqref{equa-U} and  $U \in S(T, K_0, \epsilon_0, \alpha_0, A,\delta_0, C_0, \eta_0,t),$ for all $t \in [0,t_1]$ for some  $t_1 \leq T$. Then, the following estimates follows: for all $t \in [0,T)$
\begin{eqnarray}
\|U(.,t)\|_{L^\infty(\Omega)}  & \leq  &   C(K_0,A)(T-t)^{-\frac{1}{3}}, \label{estima-norm-U-T-t1-3}\\
\|\nabla U(.,t)\|_{L^\infty(\Omega)} & \leq &   C(K_0,A)\frac{(T-t)^{-\frac{5}{6}}}{|\ln (T-t)|^\frac{1}{2}}\label{estima-norm-nabla-U-T-t1-3},\\
\|\nabla^2 U (.,t)\|_{L^\infty (\Omega)} & \leq  & C (K_0,A) (T - t)^{-c},\label{estima-rough-nabla-2-U}
\end{eqnarray} 
for some constant $c = c(K_0,A) > 0$.

\medskip
In particular, we have the following  local convergence:   We assume furthermore that $U \in S (t),$ for all $t < T$. Then, for all  $x \in \Omega$   there exist  $ R_x > 0, t_x \in [0,T)  $  such that  the following holds
\begin{equation}\label{estima-partial-U-epsilon-0}
  \| \partial_t  U(.,t)\|_{L^\infty(B(x, R_x))}   \leq C(K_0, A,T,x), \forall t \in [t_x,T). 
  \end{equation}

\end{lemma}
\begin{remark}
We would like to  remark that from  \eqref{estima-partial-U-epsilon-0} and the definition of the shrinking set $S(t)$ (see Definition \ref{defini-shrinking-set-S-t}), we   ensure  for all $x_0 \in \Omega \backslash \{0\},$  $U(x_0,t) $ is convergent as  $t \to T$.
\end{remark}
\begin{proof}
We see that estimates   \eqref{estima-norm-U-T-t1-3} and \eqref{estima-norm-nabla-U-T-t1-3}  directly      follow from  the definition of  the shrinking set and Lemma  \ref{lemma-properties-V-A-s}.   For that reason, we only give here  the proofs of \eqref{estima-rough-nabla-2-U} and  \eqref{estima-partial-U-epsilon-0}.  

 \textit{ -  The proof of  \eqref{estima-rough-nabla-2-U}:}  From  \eqref{defini-bar-u}, \eqref{defini-U-x-t}, we consider $ u$ defined as follows:
 \begin{eqnarray}
  u(x,t)  =  1 - \frac{1}{ 1 + \frac{\lambda^{\frac{1}{3}} U (x,t)}{ \bar\theta (t)} }.\label{rela-u-U-appendix}
 \end{eqnarray}
Then, $u $ satisfies  \eqref{equa-mems-u} and $u(0)$ is in $C^\infty_0 (\Omega)$. We now derive an equation satisfied by $\nabla^2 u$ as follows:
\begin{equation}\label{equa-nabla-2-u} 
\partial_t \nabla^2 u= \Delta (\nabla^2 u ) +  H_1  \nabla^2 u + H_2,
\end{equation}
where  $H_1 =  \frac{2\lambda}{ \bar \theta^3 (t)} \frac{1}{( 1 -u)^3} $   and $  H_2 = (H_{2,i,j})_{i,j \leq n} $ is a square matrix  with
$$   H_{2,i,j}  = 6 \frac{ \partial_{y_i} u \partial_j u}{(1-u)^4}.$$
Using  the definition of $u,$ \eqref{bound-bar-theta} and   two estimates   \eqref{estima-norm-U-T-t1-3} and \eqref{estima-norm-nabla-U-T-t1-3}, we can derive the following fact: for all $t \in [0,T)$, 
\begin{eqnarray*}
 \left\| H_1 (t)\right\|_{L^{\infty} (\Omega)} &\leq & C(K_0,A)(T-t)^{-1},\\
\|H_2(t)\|_{L^\infty (\Omega)} & \leq & C(K_0,A) (T-t)^{- \frac{5}{3}}.
\end{eqnarray*}
We write  $\nabla^2 u$ under  the integral equation  following 
$$  \nabla^2 u(t) = e^{t \Delta } (\nabla^2 u(0))  + \int_{0}^t e^{(t-s)\Delta} \left[ H_1(s) \nabla^2 u  + H_2 (s)   \right](s)  ds. $$
This  implies that 
$$ \|\nabla^2 u (t)\|_{L^\infty (\Omega)}  \leq \| e^{t \Delta } (\nabla^2 u(0))\|_{L^\infty (\Omega)}  + C (K_0,A)\int_0^t\left( \frac{1}{T-s} \|\nabla u (s)\|_{L^\infty (\Omega)}  + (T-s)^{-\frac{5}{3}} \right) ds.  $$
Besides that,  we can prove that  there exists  $c_1 > 0$ such that  
$$ \| e^{t \Delta } (\nabla^2 u(0))\|_{L^\infty (\Omega)}  \leq  C(T-t)^{-c_1}.    $$
Thanks to Growall's lemma, we get the following
$$  \|\nabla^2 u\|_{L^\infty (\Omega)} \leq  C (K_0,A) (T-t)^{- c_2} , \text{ with some constant } c_2 > 0. $$
Finally, from the relation between $u$ and $U$, we can get the conclusion  of \eqref{estima-rough-nabla-2-U}.

\textit{ - The proof of \eqref{estima-partial-U-epsilon-0}:} By using the definitions \eqref{defini-P-2-t} and  \eqref{defini-P-2-t} of $P_2 (t)$ and  $P_3 (t)$, respectively, if we consider an  arbitrary  $x \in  \Omega  \setminus \{0\}$, then, there exist $t_x, r_x$ such that  
$$ \text{ the ball of radius }  r_x, \text{ centred } x \quad   B(x, r_x) \in P_2 (t) \cup P_3 (t), \forall t \in [t_x,T).$$ 
Then,  using the definition of the shrinking set $S(t),$ given in Definition  \ref{defini-shrinking-set-S-t} and the fact  that  $u \in S(t)$ for all $t \in  [t_x, T)$, we derive that there exists $C (K_0,x)$ such that for all $t \in [t_x,T)$, we have 
\begin{eqnarray}
\| U (.,t) \|_{L^{\infty} (B(x,r_x))}  &\leq & C(K_0,x). \label{esitma-U-B-x-r-s-C-x}
\end{eqnarray}
In addition to that, we derive from Proposition \ref{propo-bar-mu-bounded}, we have
\begin{eqnarray}
1 \leq   \bar \theta (t)  \leq  C,  \text{ and }| \bar \theta' (t)| \leq C (T-t)^{\frac{3n-8}{6}} |\ln (T-t)|^{n} \leq (T-t)^{- \frac{11}{12}}, \label{theta-'-leq-11-12}
\end{eqnarray}
for all $t \in [t_x,T)$. 
\medskip

\noindent
 We recall $u$, defined  in \eqref{rela-u-U-appendix}. We now derive an equation satisfied by $ \partial_t u$
\begin{equation}\label{equa-nabla-u-appen}
\partial_t ( \partial  u  ) =  \Delta  \partial_t u    +  H_1  \partial_t u + H_3 (t) ,
\end{equation}
where  
\begin{eqnarray*}
 H_1 (t) &=&   \frac{2 \lambda}{ \bar \theta^3 (t)}  \frac{1}{(1-u)^3},\\
 H_3 (t) & = & -\frac{3 \lambda}{(1-u)^2} \frac{\bar{ \theta}' (t)}{ \bar \theta^4 (t)} .
\end{eqnarray*}
We then introduce the following  cut-off function  : $\phi \in  C^\infty_0 (\mathbb{R}^n)$ which satisfying 
$$   \phi (z) = 1 \text{ if  }  |z -x| \leq \frac{r_x}{2}, \text{ and  } \phi (z) = 0 \text{ if }  |z- x| \geq \frac{3}{4} r_x \text{ and }   0 \leq  \phi (z) \leq 1, \forall z \in \mathbb{R}^n.$$
Particularly, we also define 
$$ v(z,t)  =   \phi (z)  \partial_t u (z,t) \text{ for all } z \in \mathbb{R}^n$$ 
Using   \eqref{equa-nabla-u-appen}, we can derive  an equation satisfied  by  $v(t)$ as follows
\begin{equation}\label{equa-v-t-appen}
\partial_t v = \Delta v  - 2 \text{div} (\nabla \phi  \partial_t u)  +  \Delta \phi \partial_t  u + H_1  v(t),
\end{equation}
Using  \eqref{estima-rough-nabla-2-U},  \eqref{rela-u-U-appendix} \eqref{esitma-U-B-x-r-s-C-x} and  the fact that $U$ is nonnegative solution, we  can deduce  that
$$ \| \nabla \phi \partial_t u(t) \|_{L^\infty (\mathbb{R}^n)} \leq   C(K_0,A, x) (T-t)^{-c},$$
and  
$$  \|\Delta \phi \partial_t u (t)\|_{ L^\infty (\mathbb{R}^n)} \leq C (K_0,A, x) (T-t)^{-c}.  $$
Moreover, we canderive from \eqref{esitma-U-B-x-r-s-C-x} and  \eqref{theta-'-leq-11-12}   that
$$  \|I_{\{|z - x| \leq  r_x\}} H_1 (t)\|_{L^\infty (\mathbb{R}^n)} \leq C(K_0,A,x),$$
and
$$ \|H_3 (t)\|_{L^\infty (\mathbb{R}^n)}  \leq  C(K_0,x) (T-t)^{-\frac{11}{12}}.$$
We now deduce  from \eqref{equa-v-t-appen} that $v$ satisfies the following   integral equation 
$$  v(t) =  e^{(t-t_x)\Delta} v(t_x)  +  \int_{t_x}^t   e^{(t-s) \Delta}  \left[ - 2 \text{div} (\nabla \phi \partial_t u )  + \Delta \phi \partial_t u + H_1 v (s)   \right] ds, $$
where $e^{t\Delta }$ stands for  the heat semigroup on $\mathbb{R}^n$. Then,     we get the following
$$ |v(t)| \leq  C(K_0, A,x)  (1 + (T-t)^{-c+1})   +   2  \left|  \int_{t_x}^t   e^{(t-s) \Delta}  \text{div} (\nabla \phi \partial_t u )  ds \right|.$$
In particular, we have
\begin{eqnarray*}
\left|   e^{(t-s)\Delta}  \text{div} (\nabla \phi \partial_t u )  \right|  \leq  \frac{C}{ \sqrt{t-s}} \| \nabla \phi \partial_t u\|_{L^\infty (\mathbb{R}^n)} \leq C(K_0,x) \frac{(T-t)^{-c}}{ \sqrt{t-s}}.
\end{eqnarray*}
This implies that
$$  |v(t)|  \leq C(K_0, A,x) (1   +  (T-t)^{-c+1})   +  C(K_0, A, x) \int_{t_x}^t  \frac{(T-s)^{-c}}{(t-s)^{\frac{1}{2}}}  ds.$$

+ If  $ -c + \frac{1}{4} \geq 0.$ This  give us   that 
$$ |  v(t) | \leq C(K_0, A,x) ,$$
which yields the conclusion of our proof.

+ Otherwise,  we use  the above estimate to derive that
$$ |v(t)| \leq C(K_0, A,T,x) (T-t_{x})^{-c + \frac{1}{4}}.$$
We can see that  by using a  parabolic estimate as we have done. We can improve  our estimate on  $|v(t)|$ from  $C(K_0,A,x)(T-t)^{-c} $ to $C(K_0,A,x)(T-t)^{-c+ \frac{1}{4}}.$ Hence, we can repeat with a finite steps to  get the conclusion of the proof. We kindly refer the reader to check this argument. 

\end{proof}

\bibliographystyle{alpha}
\bibliography{mybib}

E-mail address:  G. K. Duong: duong@math.univ-paris13.fr

Email address: H.Zaag: Hatem.Zaag@univ-paris13.fr
\end{document}